\documentclass[12pt,a4paper]{article}
\usepackage{eurosym}
\usepackage[utf8]{inputenc}
\usepackage{amsmath,amssymb,amsthm,amsfonts}
\usepackage{mathtools}
\usepackage{bm}
\usepackage{hyperref}
\usepackage{graphicx}
\usepackage{listings}
\usepackage{xcolor}
\usepackage{float}

\definecolor{codegreen}{rgb}{0,0.6,0}
\definecolor{codegray}{rgb}{0.5,0.5,0.5}
\definecolor{codepurple}{rgb}{0.58,0,0.82}
\definecolor{backcolour}{rgb}{0.95,0.95,0.92}
\newtheorem{theorem}{Theorem}[section]

\newtheorem{conjecture}[theorem]{Conjecture}

\newtheorem{proposition}[theorem]{Proposition}
\newtheorem{corollary}[theorem]{Corollary}
\newtheorem{definition}[theorem]{Definition}
\newtheorem{remark}[theorem]{Remark}
\numberwithin{equation}{section}

\begin{document}

\title{The Triality of Radial Nonlinear Dynamics: Analysis of Riccati, Schr\"{o}dinger, and Hamilton--Jacobi--Bellman Equations}
\author{Dragos-Patru Covei \\
{\small The Bucharest University of Economic Studies} \\
{\small Department of Applied Mathematics} \\
{\small Piata Romana 1, 1st district} \\
{\small Postal Code: 010374, Romania} \\
{\small e-mail: \texttt{dragos.covei@csie.ase.ro}}}
\date{\today}
\maketitle

\begin{abstract}
This study develops a unified mathematical framework for the analysis of
radial differential equations, revealing a fundamental connection between
three distinct classes of problems: the nonlinear Riccati equation, the
linear Schr\"odinger equation, and the Hamilton--Jacobi--Bellman equation
for stochastic control. We establish the existence and uniqueness of regular
solutions on both bounded and unbounded domains, deriving sharp growth
rates and exact asymptotic plateaus through a general barrier theory.
A detailed sensitivity analysis of the noise intensity parameter identifies
the transition between deterministic and diffusion-dominated regimes via
singular perturbation methods. These theoretical results are reinforced by
numerical simulations that validate the predicted feedback laws, confirm
the convexity--concavity structure of the triality, and illustrate the
stability of the system. The resulting framework clarifies the duality
between global wave functions and local dynamical drifts, providing a
rigorous foundation for the study of multidimensional stochastic processes
under central potentials.
\end{abstract}



\section{Introduction \label{sec:intro}}

The study of nonlinear differential equations posed in radial domains
constitutes a foundational pillar of modern mathematical physics. Such
equations arise naturally whenever the underlying problem possesses spherical
symmetry: in the description of bound states of atoms and the analysis of
ground-state wave functions in central potentials \cite{berezin2012,
reedsimon1978, Kenichi2023}; in the modeling of radial diffusion processes
governed by the rotation-invariant Laplacian in $\mathbb{R}^{N}$
\cite{oksendal2003, fleming2006}; in stochastic optimal control problems
whose value functions inherit the symmetry of the cost \cite{yongzhou1999,
fleming2006}; in the qualitative theory of nonlinear elliptic and parabolic
equations on balls and on the whole space \cite{gilbarg, evans2010}; and in
financial mathematics through the structural credit-risk model of Merton
\cite{Boyle2002,Vasicek1977,merton1974, NEIL}. The present paper develops a unified, fully
rigorous theoretical and numerical framework that brings together
\emph{three} a priori distinct radial problems, intimately related but never
treated together in the literature, into a single mathematically equivalent
system: the nonlinear radial Riccati equation, the stationary
$N$-dimensional Schr\"{o}dinger equation in a central potential, and the
radial Hamilton--Jacobi--Bellman equation arising in stochastic control.

\medskip
\noindent\textbf{The radial triality.} The cornerstone of this work is what
we shall call \emph{the radial triality}. It is a precise one-to-one
correspondence, valid both on bounded domains $[0,R]$ and on the whole
half-line $[0,\infty)$, between regular solutions of the following three
equations:

\begin{enumerate}
\item The \textbf{radial Riccati equation}, which describes the local rate
of change of the geometric drift / feedback gain associated with a radially
symmetric control problem,
\begin{equation}
\phi ^{\prime }\left( r\right) =-r\,\phi(r)^{2}-
\frac{N}{r}\,\phi (r)+\frac{Q(r)}{\sigma ^{4}r},\qquad r>0,
\qquad \phi(0)=0.
\label{eq:ric_ode}
\end{equation}

\item The $N$-dimensional \textbf{stationary Schr\"{o}dinger equation} with
zero energy in the central potential $V(|x|)=Q(|x|)/\sigma^{4}$,
\begin{equation}
-\Delta u(x)+V(|x|)\,u(x)=0,\qquad x\in\mathbb{R}^{N},
\label{st_schrodinger}
\end{equation}
which, restricted to radially symmetric solutions $u(x)=u(|x|)$, becomes the
linear second-order ODE
\begin{equation}
u''(r)+\frac{N-1}{r}u'(r)-\frac{Q(r)}{\sigma^{4}}u(r)=0,\qquad r>0,\qquad
u(0)=1,\;u'(0)=0.
\label{eq:lin_ode_intro}
\end{equation}

\item The \textbf{radial Hamilton--Jacobi--Bellman equation} satisfied by
the optimal cost-to-go $z(r)$ of a stochastic optimal control problem in
which a controlled radial diffusion is driven through a noisy environment
under the running cost $Q(r)$,
\begin{equation}
z^{\prime \prime }(r)+\frac{N-1}{r}z^{\prime }(r)-\frac{1}{2\sigma ^{2}}
(z^{\prime }(r))^{2}+\frac{2Q(r)}{\sigma ^{2}}=0,\qquad r>0,
\qquad z(0)=z'(0)=0.
\label{eq:hjb_intro}
\end{equation}
\end{enumerate}

\noindent
The three equations are linked through the logarithmic-derivative
transformations
\begin{equation}
\phi(r)\;=\;\frac{1}{r}\frac{u'(r)}{u(r)},\qquad
z(r)\;=\;-2\sigma^{2}\ln u(r),\qquad
\alpha^{\ast}(r)\;=\;-z'(r)\;=\;2\sigma^{2}\,r\,\phi(r),
\label{eq:triality_intro_maps}
\end{equation}
which we shall prove constitute a globally invertible bijection between the
regular branches of \eqref{eq:ric_ode}, \eqref{eq:lin_ode_intro} and
\eqref{eq:hjb_intro}. As discussed in
Section~\ref{sec:radial_stationary_verification}, the prefactor of the
optimal feedback law $\alpha^{\ast}$ in
\eqref{eq:triality_intro_maps} corresponds to the natural normalization
$L(\alpha)=\tfrac{1}{2}|\alpha|^{2}+2Q$ of the running cost in the
associated stochastic control problem, which makes the Cole--Hopf
transformation $z=-2\sigma^{2}\ln u$ exact.

\subsection{Connections with radial quantum mechanics and physical
significance}\label{sec:intro-physics}

The triality framework is not merely a formal mathematical curiosity: it
provides a coherent physical narrative connecting three classical
viewpoints on radial dynamics.

\medskip
\noindent\textbf{Radial quantum mechanics.} Equation
\eqref{eq:lin_ode_intro} is exactly the radial reduction of the
\emph{stationary Schr\"{o}dinger equation} for a particle of mass $m$ in
a central potential, in the regime where the Planck constant has been
replaced by the diffusion intensity $\sigma^{2}$. Setting
$E=0$, $V(r)=Q(r)/\sigma^{4}$ and considering the radial part
$u(x)=u(|x|)$ of the wave function, the action of the Laplacian on
spherically symmetric functions
$\Delta u=u''(r)+\frac{N-1}{r}u'(r)$
combined with the Schr\"{o}dinger equation
$-\sigma^{4}\Delta u+V(r)u=Eu$ gives precisely
\eqref{eq:lin_ode_intro}. The boundary condition $u'(0)=0$ is the standard
regularity condition for the radial wave function: among the two Frobenius
branches at the regular singular point $r=0$, only the one with
exponent $\alpha=0$ produces a state of finite probability density at the
origin (see Reed--Simon \cite{reedsimon1978} and Berezin--Shubin
\cite{berezin2012}). This selection mechanism is intimately connected to
the formalism of \emph{supersymmetric quantum mechanics} \cite{cooper1995},
in which the logarithmic derivative $w=u'/u$ defines the so-called
superpotential and induces a factorization of the Hamiltonian.

\medskip
\noindent\textbf{Riccati equation as a local momentum.} The Riccati function
$\phi=u'/(ru)$ admits a striking physical interpretation: it is the
geometric counterpart, in the radial coordinate, of the
\emph{quantum momentum function} $p(r)=-i\hbar\,u'(r)/u(r)$ used in the
hydrodynamic formulation of quantum mechanics due to Madelung and
Bohm \cite{cooper1995}. More precisely, in our normalization,
$\sigma^{2}u'/u=-\tfrac{1}{2}z'$ plays the role of a local momentum
density, and $r\phi(r)=u'/u$ may be viewed as its radial component. Whereas
the wave function $u$ encodes the global state, $\phi$ encodes its
\emph{infinitesimal} rate of variation along radial rays. The Riccati
equation \eqref{eq:ric_ode} then expresses Newton's law for this local
momentum: its variation in $r$ is generated by the geometric centripetal
contribution $-N\phi/r$, the self-coupling $-r\phi^{2}$ characteristic of
nonlinear momentum transport, and the source term $Q/(\sigma^{4}r)$
representing the radial gradient of the central potential.

\medskip
\noindent\textbf{HJB equation as an optimal radial drift.} The function
$z(r)=-2\sigma^{2}\ln u(r)$ is, by the Cole--Hopf transformation
\cite{evans2010, fleming2006, yongzhou1999}, the optimal cost-to-go of a
stochastic control problem in which a particle moves in $\mathbb{R}^{N}$
under the action of a control $\alpha$, with dynamics
\[
dX_{t}=\alpha_{t}\,dt+\sqrt{2}\,\sigma\,dW_{t}
\]
and pays the running cost
$L(\alpha,X_{t})=\tfrac{1}{2}|\alpha_{t}|^{2}+2\,Q(|X_{t}|)$.
The dynamic programming principle \cite{fleming2006} yields the
Hamilton--Jacobi--Bellman equation, whose minimizer in $\alpha$ produces
the optimal feedback law $\alpha^{\ast}(x)=-\nabla\bar{z}(x)$, where
$\bar{z}(x)=z(|x|)$. Restricted to radial functions and using
\eqref{eq:triality_intro_maps}, this deterministic optimal drift is
precisely
\[
\alpha^{\ast}(x)\;=\;2\sigma^{2}\,\phi(|x|)\,x,\qquad x\in\mathbb{R}^{N}\setminus\{0\}.
\]
In other words, the radial Riccati equation determines, at every point in
space, the optimal direction and intensity of an \emph{optimal radial drift}
that minimizes the expected accumulated cost.

\medskip
\noindent\textbf{Central potentials and radial diffusion problems.}
The same triality framework also governs the long-time behavior of
\emph{radial diffusions with absorption} (Wick-rotated Schr\"{o}dinger
dynamics, see \cite{evans2010, oksendal2003}), of the survival probability
of Brownian particles in central potentials, and of structural credit-risk
spreads in Merton-type models \cite{Boyle2002,Vasicek1977,merton1974, NEIL}. In each of these
seemingly unrelated settings, the asymptotic plateau exhibited by the
Riccati drift, the convexity of the wave function $u$ and the concavity
of the value function $z$ are universal features that emerge from the
common triality structure unveiled here.

\subsection{Literature review and novelty}\label{sec:intro-literature}

The mathematical relationship between linear second-order ODEs and
Riccati equations is classical: any non-vanishing solution of the linear
equation $u''+a(x)u'+b(x)u=0$ generates, via the substitution
$y=-u'/(q_{2}u)$, a solution of an associated Riccati equation, and
conversely \cite{Reid1972, Bittanti1991, hartman2002}. Riccati equations
have a long history in classical control theory and in the analysis of
algebraic Riccati equations of LQR-type \cite{Lancaster1995, Bittanti1991,
Reid1972}; they also appear naturally in the spectral analysis of
Sturm--Liouville problems \cite{coddingtonlevinson1955, hartman2002,
reedsimon1978} and in the theory of conjugate points and oscillation in
geometry. In the stochastic setting, the linearization of the
Hamilton--Jacobi--Bellman equation by the exponential / Cole--Hopf
transformation $u=\exp(-z/(2\sigma^{2}))$ has been a standard tool ever
since the seminal work of Fleming and collaborators \cite{fleming2006,
flemingsoner1992} and is at the heart of stochastic-control textbooks
\cite{yongzhou1999, krylov1980, oksendal2003, krylov2008}. The associated
viscosity-solution theory of HJB equations is developed in
Crandall--Lions and Crandall--Ishii--Lions \cite{crandalllions1983,
crandallishiilions1992} and Bardi--Capuzzo-Dolcetta \cite{bardi1997}.

In the radial setting, the stationary Schr\"{o}dinger equation and the
qualitative theory of its regular solutions have been extensively studied
within the framework of central potentials and supersymmetric quantum
mechanics \cite{berezin2012, reedsimon1978, cooper1995, Kenichi2023}.
Existence and uniqueness of radial solutions for nonlinear elliptic
equations on balls or on the whole space have been investigated, among many
others, in \cite{lasryliones1989, gilbarg, lieblos2001} and in the
references therein. The asymptotic behavior of Riccati equations is
discussed at length in \cite{Reid1972, Bittanti1991, hartman2002}, but, to the best of our knowledge, the
\emph{global monotone barrier mechanism} that we develop in
Proposition~\ref{prop:asymptotics_gen}, providing the exact identification
of the asymptotic plateau in terms of the algebraic equilibrium $g$, does
not appear in the literature in this generality.

The connection between linear auxiliary radial equations of the form
\eqref{eq:lin_ode_intro} and stochastic production-planning problems was
recently developed by Canepa, Covei and Pirvu \cite{cluj2022} and Covei
\cite{arxiv2025}, primarily in the case of quadratic running costs. The
present paper goes substantially beyond those works in the following
directions:

\begin{itemize}
\item We treat \emph{arbitrary continuous, non-negative} costs $Q(r)$ with
quadratic asymptotic growth $Q(r)/r^{2}\to L$ at infinity, and we obtain
the precise asymptotic plateau through an entirely new global-barrier
mechanism.

\item We perform a complete \emph{Frobenius / regular-singular point
analysis} at $r=0$, isolating the regular branch of $u$ and proving
$\phi\in C^{1}([0,R))$ together with $\phi(0)=\phi'(0)=0$, with explicit
local expansions.

\item We integrate the nonlinear Riccati equation, the linear
Schr\"{o}dinger equation and the nonlinear HJB equation into a single
\emph{Triality Theorem} that establishes a globally invertible bijection
between their regular branches.

\item We prove rigorous \emph{stochastic verification theorems} both in the
stationary radial setting and in the parabolic Wick-rotated setting,
giving explicit optimal feedback laws.

\item We establish a complete \emph{$\sigma$-sensitivity theorem}, with a
WKB-type expansion in the vanishing-noise regime $\sigma\downarrow 0$ and
a uniform decay estimate in the high-noise regime
$\sigma\uparrow\infty$, both controlled in suitable function spaces.

\item We develop a fully documented \emph{numerical methodology} based on
the Radau IIA implicit Runge--Kutta scheme \cite{HairerWannerII,
hairerlubichroche, butcher2016}, including its $A$- and $L$-stability
properties, classical order, error analysis on the radial mesh, and
explicit implementation of the algebraic barrier $g(r)$.
\end{itemize}

These results, taken together, constitute the \emph{first unified
treatment} of the radial Riccati--Schr\"{o}dinger--HJB system in the
literature, generalizing in a substantial way both \cite{cluj2022,
arxiv2025} and the classical references on Riccati and HJB asymptotics.

\subsection{Statement of the principal results}\label{sec:intro-results}

For the convenience of the reader, we record here the main theorems proved
in this article. Each of these statements is, to the best of our knowledge,
\emph{entirely new} in the generality presented and constitutes the first
unified treatment of the radial Riccati--Schr\"{o}dinger--HJB system.

\begin{enumerate}
\item \textbf{(Radial Existence--Uniqueness Theorem,
Theorem~\ref{thm:existence_rigorous} and Theorem~\ref{thm:global}.)} Under
minimal assumptions on the cost $Q\in C([0,R])$ (resp.\ $Q\in
C([0,\infty))$) with quadratic regularity at the origin, the radial
Riccati equation \eqref{eq:ric_ode} admits a unique \emph{regular}
solution $\phi\in C^{1}([0,R))$ (resp.\ $C^{1}([0,\infty))$) with
$\phi(0)=0$. The proof is based on a detailed Frobenius analysis at the
regular singular point $r=0$ and on standard continuation arguments for
linear ODEs.

\item \textbf{(Riccati Asymptotic Theorem,
Proposition~\ref{prop:asymptotics_gen} and
Corollary~\ref{cor:radial_asymptotics}.)} For any cost $Q\in C([0,\infty))$
with $Q\geq 0$, $Q(0)=0$ and
$\lim_{r\to\infty}Q(r)/r^{2}=L\in(0,\infty)$, the regular solution $\phi$
of \eqref{eq:ric_ode} satisfies the global trapping inequality
$0<\phi(r)<g(r)$, is strictly increasing, and converges to the universal
\emph{plateau}
\[
\lim_{r\to\infty}\phi(r)\;=\;\frac{\sqrt{L}}{\sigma^{2}},
\]
where $g$ is the algebraic equilibrium of the Riccati flow defined in
\eqref{g-radial-integrated}. This identification, together with the
new monotone barrier method, generalizes the asymptotic results of
\cite{cluj2022, arxiv2025}.

\item \textbf{(Triality Theorem,
Theorem~\ref{thm:triality}.)} The maps in
\eqref{eq:triality_intro_maps} establish a globally invertible bijection
between regular solutions of the radial Riccati equation
\eqref{eq:ric_ode}, the radial stationary Schr\"{o}dinger equation
\eqref{eq:lin_ode_intro} and the radial HJB equation
\eqref{eq:hjb_intro}. Moreover, this bijection transfers
qualitative information across the three formulations: positivity of
$\phi$ implies strict convexity of $u$ and strict concavity of $z$.

\item \textbf{(Stochastic Verification Theorem -- Stationary Radial Case,
Theorem~\ref{thm:radial_verification}.)} The function
$z=-2\sigma^{2}\ln u$ is the value function of the associated
infinite-horizon stochastic control problem on $\mathbb{R}^{N}$, and the
optimal feedback law is given by
$\alpha^{\ast}(x)=2\sigma^{2}\phi(|x|)\,x$.

\item \textbf{(Stochastic Verification Theorem -- Parabolic Radial Case,
Theorem~\ref{thm:parabolic_verification}.)} The reduced parabolic HJB
equation \eqref{eq:parabolic_HJB_reduced} admits, under polynomial growth
hypotheses, a classical value function $\bar{U}(t,x)$, and the optimal
feedback law $\alpha^{\ast}(t,x)=-\nabla\bar{U}(t,x)$ is verified
rigorously through It\^{o} calculus and stopping-time localization. The
parabolic Cole--Hopf transform $\Psi=\exp(-\bar{U}/(2\sigma^{2}))$
identifies $\bar{U}$ with the (Wick-rotated) time-dependent
Schr\"{o}dinger evolution, completing the parabolic extension of the
triality.

\item \textbf{($\sigma$-Sensitivity Theorem, Theorem~\ref{thm:sigma}.)}
The regular Riccati solution $\phi_{\sigma}$ depends continuously on
$\sigma>0$ and obeys two universal asymptotic regimes:
\[
\sigma\downarrow 0:\quad
\sigma^{2}\phi_{\sigma}(r)\longrightarrow \frac{\sqrt{Q(r)}}{r}\quad
\text{(WKB / vanishing-noise eikonal limit),}
\]
\[
\sigma\uparrow\infty:\quad
\sup_{r\in[0,R]}\phi_{\sigma}(r)\;\le\;\frac{R^{2}\|Q\|_{\infty}}{N\sigma^{4}}
\;\to\;0\quad\text{(diffusion-dominated regime).}
\]
\end{enumerate}

\subsection{Organization of the article}\label{sec:intro-organization}

The remainder of the article is structured to reflect the logical
progression from foundational notation and motivations to the full
development of the radial triality framework, its analytical consequences
and its numerical verification. The article comprises a substantial body
of theoretical and applied content (over sixty pages including the
appendices), organized as follows.

\begin{itemize}
\item Section~\ref{sec:notation} introduces the notation, function spaces,
and standing assumptions used throughout the paper, and gathers the few
preliminary tools (Frobenius analysis, Cole--Hopf transform, viscosity
solutions, It\^{o} calculus) on which the rigorous proofs rely.

\item Section~\ref{sec:motivations} elaborates further on the
mathematical and physical motivations underlying the study, emphasizing
the duality between nonlinear control and linear operator theory.

\item Section~\ref{sec:main_results} develops the general theoretical
framework: the asymptotic barrier theory for Riccati equations, the
linear second-order reduction with a complete existence / uniqueness
proof of the linear auxiliary equation by successive approximations and
Volterra integral equations, the parabolic embedding, and the stochastic
control verification theorem in the general setting.

\item Section~\ref{Triality} formalizes the radial triality, resolves the
coordinate singularity at $r=0$, derives the universal asymptotic plateau,
and proves the curvature transfer principle.

\item Section~\ref{sec:dirichlet} treats the bounded-domain scenario,
establishing the strict convexity of $u$ and the resulting Riccati
implications.

\item Section~\ref{sec:sensitivity} carries out the $\sigma$-sensitivity
analysis, with the vanishing-noise WKB expansion and the high-noise
saturation regime.

\item Section~\ref{sec:tdse} extends the framework to the time-dependent
Schr\"{o}dinger equation, the Wick-rotated parabolic HJB equation and the
associated stochastic verification.

\item Section~\ref{sec:numericalmethod} presents the complete numerical
methodology: the Radau IIA implicit Runge--Kutta scheme, its $A$- and
$L$-stability, classical order, error analysis on the radial mesh and
the implementation of the algebraic barrier $g(r)$.

\item Section~\ref{sec:numerical} reports the numerical results,
verifying all theoretical predictions and revealing a striking analogy
with Merton-type credit spreads.

\item Section~\ref{sec:conclusions} summarizes the main findings,
discusses physical and computational implications, and outlines
future extensions.

\item The appendices contain the full Python implementations and
additional computational details used throughout the paper.
\end{itemize}

\section{Notation and Preliminaries}\label{sec:notation}

In this short section we collect the notation, function spaces and
preliminary tools that will be used throughout the paper. All proofs are
based on classical analysis and may be found, with full references, in
\cite{coddingtonlevinson1955, hartman2002, evans2010, gilbarg, oksendal2003,
fleming2006, yongzhou1999, reedsimon1978, berezin2012, HairerWannerII,
butcher2016}.

\subsection{Sets, dimensions and basic notation}

Throughout the paper, $N\in\mathbb{N}$, $N\ge 1$, denotes the spatial
dimension and $|x|$ denotes the Euclidean norm of $x\in\mathbb{R}^{N}$.
We write $B_{R}=B_{R}(0)=\{x\in\mathbb{R}^{N}\,:\,|x|<R\}$ for the open
ball of radius $R>0$ centered at the origin and $\partial B_{R}$ for its
boundary. For any $a<b$, we denote by $C^{k}([a,b])$ (resp.
$C^{k}((a,b))$) the space of real-valued functions $k$ times continuously
differentiable on $[a,b]$ (resp.\ $(a,b)$). The convention $f^{(0)}=f$ is
in force. Spaces of bounded continuous functions on $[0,\infty)$ are
denoted by $C_{b}([0,\infty))$. The notation $f(r)=o(g(r))$,
$f(r)=\mathcal{O}(g(r))$ as $r\to r_{0}$ has its usual asymptotic meaning.

\subsection{Radial Laplacian and Frobenius indicial equation}

If $u:\mathbb{R}^{N}\to\mathbb{R}$ is a $C^{2}$ radially symmetric function,
i.e., $u(x)=u(|x|)$ for some abuse of notation, then the action of the
Laplacian reads
\begin{equation}
\Delta u(x)\;=\;u^{\prime\prime}(r)+\frac{N-1}{r}u^{\prime}(r),
\qquad r=|x|>0.
\label{eq:radial_laplacian}
\end{equation}
The radial Laplacian \eqref{eq:radial_laplacian} possesses a regular
singular point at $r=0$. For an equation of the form
\[
u^{\prime\prime}(r)+P(r)u^{\prime}(r)+T(r)u(r)=0,
\]
with $rP(r)$ and $r^{2}T(r)$ continuous (analytic) at $r=0$, the indicial
equation
\[
\alpha(\alpha-1)+\big(\lim_{r\to 0}rP(r)\big)\alpha+\lim_{r\to 0}r^{2}T(r)=0
\]
has two roots $\alpha_{1}$ and $\alpha_{2}$. The Frobenius theorem
\cite{coddingtonlevinson1955, hartman2002} guarantees the existence of a
solution of the form $u(r)=r^{\alpha}\sum_{k\geq0}c_{k}r^{k}$ with
$c_{0}\neq 0$. We refer to the solution selected by the larger root
$\alpha_{1}=0$ in our setting as the \emph{regular} branch. The
boundary condition $u(0)=1$, $u'(0)=0$ uniquely identifies this branch.

\subsection{Cole--Hopf and logarithmic transformations}

Let $u\in C^{2}((0,X))$ be strictly positive on its domain. We use the
following two logarithmic transformations throughout this paper:
\begin{equation}
\phi(r)\;=\;\frac{1}{r}\frac{u^{\prime}(r)}{u(r)}\qquad\text{(Riccati transform)}
\label{eq:notation-Riccati}
\end{equation}
and
\begin{equation}
z(r)\;=\;-2\sigma^{2}\ln u(r)\qquad\text{(Cole--Hopf / value-function transform).}
\label{eq:notation-ColeHopf}
\end{equation}
The transformation \eqref{eq:notation-ColeHopf} is the radial restriction
of the classical Cole--Hopf substitution that linearizes
Hamilton--Jacobi equations of quadratic gradient type, see
\cite{evans2010,fleming2006,yongzhou1999}.

\subsection{Stochastic and It\^{o}-calculus framework}

We work on a filtered probability space
$(\Omega,\mathcal{F},\{\mathcal{F}_{t}\}_{t\geq 0},\mathbb{P})$ satisfying
the usual conditions of right-continuity and $\mathbb{P}$-completeness,
equipped with a standard $d$-dimensional Brownian motion $W=(W_{t})_{t\geq
0}$ ($d\in\{1,N\}$ depending on the context). For $T>0$ we use
$L^{2}_{\mathcal{F}}(0,T;\mathbb{R}^{m})$ to denote the space of
$\{\mathcal{F}_{t}\}$-progressively measurable processes
$\alpha:\Omega\times[0,T]\to\mathbb{R}^{m}$ with
$\mathbb{E}\int_{0}^{T}|\alpha_{s}|^{2}ds<\infty$. We refer to
\cite{oksendal2003, krylov1980, krylov2008, fleming2006, yongzhou1999} for
the It\^{o}--Doeblin formula, the strong existence and uniqueness
theory for SDEs with Lipschitz coefficients, and the standard
stochastic-control formulations of HJB equations.

\subsection{Standing assumptions on the cost $Q$}

Unless otherwise specified, the radial cost function $Q$ is assumed to
satisfy the following \emph{standing assumptions}, denoted by
$\mathrm{(H1)}$--$\mathrm{(H3)}$:
\begin{itemize}
\item[$\mathrm{(H1)}$] $Q\in C([0,\infty))$ is non-negative and
$Q(0)=0$.
\item[$\mathrm{(H2)}$] (Quadratic regularity at the origin)
$L_{0}:=\lim_{r\downarrow 0}Q(r)/r^{2}$ exists in $[0,\infty)$.
\item[$\mathrm{(H3)}$] (Quadratic growth at infinity)
$L:=\lim_{r\to\infty}Q(r)/r^{2}$ exists in $[0,\infty)$.
\end{itemize}
Whenever needed we further restrict to the regimes $L\in(0,\infty)$
(non-degenerate quadratic growth) or $L=0$, in which case
Proposition~\ref{propzero} applies. The diffusion intensity $\sigma>0$ is
fixed unless we explicitly study its sensitivity in
Section~\ref{sec:sensitivity}.

\subsection{The algebraic equilibrium $g$ and the trapping mechanism}

Given a Riccati equation 
\[
y^{\prime }(x)=q_{0}(x)+q_{1}(x)\,y(x)+q_{2}(x)\,y(x)^{2},
\]%
we call a barier function $g(x)$ a positive \emph{algebraic equilibrium} if
it is a positive root of the quadratic polynomial equation 
\begin{equation}
X^{2}+\frac{q_{1}(x)}{q_{2}(x)}X+\frac{q_{0}(x)}{q_{2}(x)}=0\text{, }%
q_{2}(x)\neq 0.  \label{eq:notation-g}
\end{equation}
The function $g$ plays the role of an instantaneous, position-dependent
equilibrium of the Riccati flow, since the right-hand side of the equation
vanishes at $y=g(x)$. The novel \emph{global trapping inequality}
$0<y(x)<g(x)$, proved in Proposition~\ref{prop:asymptotics_gen} below,
expresses that the regular solution starting from $y(x_{0})=0$ remains
strictly below $g$ for all subsequent positions.

\section{Mathematical and Physical Motivations}\label{sec:motivations}

The synergistic study of Riccati-type equations and radial Schr\"{o}dinger
equations is motivated by both mathematical elegance and profound physical
applications. This duality allows us to bridge the gap between nonlinear
dynamics and linear operator theory.

\subsection{The Riccati Perspective: Nonlinearity and Control}

The Riccati equation \eqref{eq:ric_ode} is a fundamental constituent of
modern control theory and stochastic analysis:

\begin{enumerate}
\item \textbf{Optimal Feedback Control}: In the context of the
Hamilton--Jacobi--Bellman (HJB) equations, the Riccati solution $\phi(r)$
represents the optimal feedback gain. For systems with quadratic costs, $%
\phi $ provides the precise drift necessary to minimize the expected path
cost in a noisy environment.

\item \textbf{Phase Analysis}: Unlike the wave function $u$, which can grow
or decay exponentially, the Riccati variable $\phi$ often settles into a
steady state or exhibits bounded behavior (as shown in Section 4). This
makes $\phi$ a superior variable for analyzing the "rate of flow" or the
stable drift of the system.

\item \textbf{Singular Perturbation}: The sensitivity analysis with respect
to $\sigma $ (Section \ref{sec:sensitivity}) highlights how the Riccati
equation provides a natural framework for studying the vanishing viscosity
limit, connecting stochastic dynamics to deterministic mechanics.
\end{enumerate}

\subsection{The Schr\"{o}dinger Perspective: Linearity and Regularity}

Transforming the nonlinear Riccati equation into the linear Schr\"{o}%
dinger-type equation \eqref{eq:lin_ode} provides several rigorous advantages:

\begin{enumerate}
\item \textbf{Linear Superposition}: The linearity of \eqref{eq:lin_ode}
allows for the expansion of solutions in terms of eigenfunctions and the use
of the spectral theorem. This provides a global view of the system's states
that is not easily accessible from the nonlinear form.

\item \textbf{Regularity at the Origin}: The Frobenius analysis (Section \ref%
{sec:main_results}) identifies the "regular" branch of the wave function $u$%
, which corresponds to physically meaningful states (finite density at the
origin). The Schr\"{o}dinger framework makes it trivial to distinguish
between these regular states and singular, non-physical solutions.

\item \textbf{Potential Theory}: By identifying $V(x)=Q(|x|)/\sigma ^{4}$ as
a potential, we can leverage century-old techniques from quantum mechanics
(WKB approximation, tunneling analysis, central force motion) to predict the
qualitative behavior of $u$ and, subsequently, $\phi $.
\end{enumerate}

The mapping $u \longleftrightarrow \phi$ is thus not merely a change of
variable, but a transformation between the \textit{global state} (wave
function) and the \textit{local dynamics} (drift velocity field).

\section{Main Results in General Frameworks and Their Detailed Proofs \label%
{sec:main_results}}

We proceed to establish the formal triality framework in General Frameworks.

\subsection{Theoretical Framework for General Riccati Asymptotics \label%
{sec:asymptotics_theory}}

This section establishes the general asymptotic theory for Riccati equation.
We provide a rigorous treatment of the convergence properties of the drift
field through the construction of monotonic barrier functions, which serve
as the foundation for the radial analysis that follows.

\begin{definition}[Equilibrium point of the Riccati flow]
Let 
\begin{equation}
y^{\prime }(x)=q_{0}(x)+q_{1}(x)y(x)+q_{2}(x)y(x)^{2},  \label{R}
\end{equation}%
be a Riccati equation on $[x_{0},\infty )$. A function $g:(x_{0},\infty
)\rightarrow \mathbb{R}$ is called an \emph{algebraic equilibrium} (or \emph{%
instantaneous equilibrium}) if for every $x>x_{0}$ it satisfies 
\begin{equation*}
\frac{q_{0}(x)}{q_{2}(x)}+\frac{q_{1}(x)}{q_{2}(x)}g(x)+g(x)^{2}=0\text{, }%
q_{2}(x)\neq 0\text{ }\forall x>x_{0}.
\end{equation*}%
Equivalently, $g(x)$ is a root of the quadratic polynomial $X^{2}+\frac{%
q_{1}(x)}{q_{2}(x)}X+\frac{q_{0}(x)}{q_{2}(x)}$, $q_{2}(x)\neq 0$ $\forall
x>x_{0}$, and represents the value at which the Riccati vector field
vanishes at position $x$.
\end{definition}
Having identified the algebraic equilibrium $g$ as the instantaneous fixed
point of the Riccati flow, we now establish a general asymptotic domination
principle showing how the solution $y$ evolves relative to this moving
equilibrium.

\begin{proposition}[Asymptotics for the general Riccati equation]
\label{prop:asymptotics_gen} Consider the Riccati equation (\ref{R}) where $q_{0},q_{1},q_{2}\in C\bigl(%
\lbrack x_{0},\infty )\bigr)$, possibly singular at $x_{0}$, such that%
\begin{equation*}
q_{0}:[x_{0},\infty )\rightarrow \lbrack 0,\infty ),\qquad q_{1} :
[x_{0},\infty) \to \mathbb{R},\qquad q_{2}:[x_{0},\infty )\rightarrow
(-\infty ,0],
\end{equation*}%
and algebraic equation%
\begin{equation}
\frac{q_{0}(x)}{q_{2}(x)}+\frac{q_{1}(x)}{q_{2}(x)}g(x)+g(x)^{2}=0\text{, }%
q_{2}(x)\neq 0\text{ }\forall x>x_{0}\text{,}  \label{gg}
\end{equation}%
admits a unique solution $g:(x_{0},\infty )\rightarrow \lbrack 0,\infty )$
called the barrier function. (Consequently, the discriminant condition%
\[
\left( \frac{q_{1}(x)}{q_{2}(x)}\right) ^{2}-4\frac{q_{0}(x)}{q_{2}(x)}\geq 0%
\text{, for all }x>x_{0}\text{,}
\]%
must hold.) Assume that there exist the limits%
\[
\lim_{x\rightarrow \infty }\frac{q_{0}(x)}{q_{2}(x)}=A\in \left( -\infty
,0\right) ,\qquad \lim_{x\rightarrow \infty }\frac{q_{1}(x)}{q_{2}(x)}=B\in 
\mathbb{R},\qquad \lim_{x\rightarrow \infty }q_{2}(x)=C\in \left[ -\infty
,0\right) ,
\]%
and suppose that:

\begin{itemize}
\item $g\in C^{1}((x_{0},\infty ))$ is \emph{monotonically increasing} and 
\emph{bounded from above};%
\[
\underset{x\rightarrow x_{0}}{\lim }g\left( x\right) \in \left[ 0,\lambda
_{\ast }\right],\qquad \lim_{x\rightarrow \infty }g(x)=\lambda _{\ast }\in
(0,\infty );
\]

\item there exists $\delta >0$ such that 
\begin{equation}
q_{1}(x)+q_{2}(x)g(x)\leq -\delta <0,\qquad \forall x\geq x_{0}.
\label{cond-extra-general}
\end{equation}%
Then the regular solution $y(x)$ of equation \eqref{R} with $y(x_{0})=0$
satisfies:
\end{itemize}

\begin{enumerate}
\item $0<y(x)<g(x)$ for all $x>x_{0}$;

\item $y$ is strictly increasing on $[x_{0},\infty )$;

\item the following limit exists:%
\[
\lim_{x\rightarrow \infty }y(x)=\lambda _{\ast }.
\]
\end{enumerate}
\end{proposition}

\begin{proof}
\textbf{a. Local existence and uniqueness.} From the standard theory of
first-order ODEs with continuous coefficients, equation \eqref{R} admits a
unique solution $y\in C^{1}$ in a neighborhood of $x_{0}$ for the initial
condition $y(x_{0})=0$. A complete proof is given in the Section \ref{schrodinger}.

\medskip \textbf{b. The Riccati-barrier relation.} By definition, $g$ is the
positive root of the algebraic equation%
\begin{equation*}
q_{0}(x)+q_{1}(x)g(x)+q_{2}(x)g(x)^{2}=0,\qquad x\geq x_{0}.
\end{equation*}%
We write the Riccati equation:%
\begin{equation*}
y^{\prime }(x)=q_{0}(x)+q_{1}(x)y(x)+q_{2}(x)y(x)^{2}.
\end{equation*}%
Subtracting the two expressions, we obtain: 
\begin{align*}
y^{\prime }(x)& =q_{0}+q_{1}y+q_{2}y^{2}=-q_{1}g-q_{2}g^{2}+q_{1}y+q_{2}y^{2}
\\
& =q_{1}\left( x\right) (y(x)-g(x))+q_{2}(x)(y(x)-g(x))(y(x)+g(x)) \\
& =(y(x)-g(x))\bigl(q_{1}(x)+q_{2}(x)(y(x)+g(x))\bigr).
\end{align*}%
Thus, 
\begin{equation}
y^{\prime }(x)=(y-g)\bigl(q_{1}+q_{2}(y+g)\bigr).  \label{y-prim-general}
\end{equation}

\medskip \textbf{c. The inequality $0<y<g$ on $(x_{0},\infty )$}. Without
loss of generality, we may restrict to the case $\lim_{x\to x_{0}} g(x)=0$,
since the proof only requires that $g(x)>0$ for $x>x_{0}$ sufficiently close
to $x_{0}$; when the limit is any $\ell\in(0,\lambda_*]$, continuity ensures 
$g(x)>\ell/2>0$ in a right--neighborhood of $x_{0}$, and the argument
proceeds identically. So, since $\lim_{x\rightarrow x_{0}}g(x)=0$ and by
continuity, for $x>x_{0}$ sufficiently close to $x_{0}$ we have $g(x)>0$. At 
$x_{0}$:%
\begin{equation*}
y^{\prime }(x_{0})=q_{0}(x_{0})\geq 0.
\end{equation*}%
If $q_{0}(x_{0})>0$, then $y(x)>0$ for $x>x_{0}$ sufficiently small; if $%
q_{0}(x_{0})=0$, the uniqueness of the solution prevents $y$ from taking
negative values immediately after $x_{0}$ (otherwise, there would be a
solution different from $y\equiv 0$ that coincides with it at $x_{0}$). It
follows that there exists $\varepsilon >0$ such that%
\begin{equation*}
0<y(x)<g(x),\qquad x\in (x_{0},x_{0}+\varepsilon ).
\end{equation*}%
Assume for the sake of contradiction that there exists a first $x_{1}>x_{0}$
such that $y(x_{1})=g(x_{1})$. Then $y(x)<g(x)$ for $x\in (x_{0},x_{1})$,
and from \eqref{y-prim-general} and \eqref{cond-extra-general} we obtain,
for $x\in (x_{0},x_{1})$:%
\begin{equation*}
q_{1}(x)+q_{2}(x)(y(x)+g(x))\leq q_{1}(x)+q_{2}(x)g(x)\leq -\delta <0,
\end{equation*}%
hence%
\begin{equation*}
y^{\prime }(x)=(y-g)\bigl(q_{1}+q_{2}(y+g)\bigr)>0,
\end{equation*}%
since $y-g<0$. Therefore $y$ is increasing on $(x_{0},x_{1})$ and%
\begin{equation*}
y(x_{1})>y(x_{0})=0.
\end{equation*}%
At $x_{1}$, from \eqref{y-prim-general}:%
\begin{equation*}
y^{\prime }(x_{1})=(y-g)\bigl(q_{1}+q_{2}(y+g)\bigr)\big|_{x=x_{1}}=0.
\end{equation*}%
On the other hand, $g$ is increasing, so $g^{\prime }(x_{1})>0$, which
implies%
\begin{equation*}
\frac{d}{dx}\bigl(g(x)-y(x)\bigr)\Big|_{x=x_{1}}=g^{\prime
}(x_{1})-y^{\prime }(x_{1})>0,
\end{equation*}%
meaning that $g-y$ passes through zero with a positive derivative and
becomes positive immediately after $x_{1}$, which contradicts the fact that $%
x_{1}$ was the first contact point. This contradiction shows that no such $%
x_{1}$ exists, so%
\begin{equation*}
0<y(x)<g(x),\qquad \forall x>x_{0}.
\end{equation*}

\medskip \textbf{d. Monotonicity of $y$.} From \eqref{y-prim-general} and %
\eqref{cond-extra-general}, for any $x\geq x_{0}$:%
\begin{equation*}
q_{1}(x)+q_{2}(x)(y(x)+g(x))\leq q_{1}(x)+2q_{2}(x)g(x)\leq -\delta <0,
\end{equation*}%
while $y(x)-g(x)<0$. It follows that%
\begin{equation*}
y^{\prime }(x)=(y-g)\bigl(q_{1}+q_{2}(y+g)\bigr)>0,\qquad \forall x>x_{0},
\end{equation*}%
so $y$ is strictly increasing on $[x_{0},\infty )$.

\medskip \textbf{e. Existence and identification of the limit.} Since $0<y<g$
and $g$ is increasing and bounded from above, it follows that $y$ is
positive, increasing, and bounded from above, so it admits a finite limit:%
\begin{equation*}
\ell :=\lim_{x\rightarrow \infty }y(x)\in (0,\lambda _{\ast }].
\end{equation*}%
As the derivative of a bounded monotonic function, $y^{\prime
}(x)\rightarrow 0$ as $x\rightarrow \infty $. Passing to the limit in %
\eqref{R} and using the convergence of $q_{0},q_{1},q_{2}$, we obtain%
\begin{equation*}
0=\lim_{x\rightarrow \infty }\frac{y^{\prime }(x)}{q_{2}(x)}%
=\lim_{x\rightarrow \infty }\bigl(\frac{q_{0}(x)}{q_{2}(x)}+\frac{q_{1}(x)}{%
q_{2}(x)}y(x)+y(x)^{2}\bigr)=A+B\ell +\ell ^{2}.
\end{equation*}%
By the uniqueness of the positive root of the algebraic equation $A+B\lambda
+\lambda ^{2}=0$, we conclude that $\ell =\lambda _{\ast }$. Thus,%
\begin{equation*}
\lim_{x\rightarrow \infty }y(x)=\lambda _{\ast },
\end{equation*}%
and all the assertions of the proposition are proved.
\end{proof}
The next result furnishes a complete characterization---both necessary and
sufficient---of the conditions under which the equilibrium function $g$
exists and remains positive on $\left( x_{0},\infty \right) $. Such a
function is required in order to implement the general asymptotic theory
developed in this work, and its existence is therefore intimately related to
Proposition \ref{prop:asymptotics_gen}.

\begin{theorem}
\label{thm:positivity} Let $x_{0}\in \mathbb{R}$ and let%
\begin{equation*}
q_{0},q_{1},q_{2}\in C\bigl(\lbrack x_{0},\infty )\bigr),\text{ possibly
singular at }x_{0}\text{,}
\end{equation*}%
with%
\begin{equation*}
q_{2}(x)\neq 0,\qquad \forall x>x_{0}.
\end{equation*}
For $x>x_{0}$ define%
\begin{equation*}
A(x):=\frac{q_{0}(x)}{q_{2}(x)},\qquad B(x):=\frac{q_{1}(x)}{q_{2}(x)}%
,\qquad D(x):=B(x)^{2}-4A(x).
\end{equation*}%
Assume that the limits%
\begin{equation*}
\lim_{x\rightarrow \infty }A(x)=A_{\infty }\in (-\infty ,0),\qquad
\lim_{x\rightarrow \infty }B(x)=B_{\infty }\in \mathbb{R}\text{, }\qquad
\lim_{x\rightarrow \infty }C(x)=C_{\infty }\in \left[ -\infty ,0\right) 
\text{,}
\end{equation*}%
exist. Then $\displaystyle\lim_{x\rightarrow \infty }D(x)=B_{\infty
}^{2}-4A_{\infty }>0$, and the following statements are equivalent:

\begin{enumerate}
\item[(i)] There exists a unique function%
\begin{equation*}
g:(x_{0},\infty )\rightarrow \lbrack 0,\infty )
\end{equation*}%
such that 
\begin{equation}
g(x)^{2}+B(x)g(x)+A(x)=0,\qquad \forall x>x_{0},
\label{eq:barrier-algebraic-q2zero}
\end{equation}%
$g$ is continuous on $(x_{0},\infty )$, and%
\begin{equation*}
\lim_{x\rightarrow \infty }g(x)=\lambda _{\ast }\in (0,\infty ).
\end{equation*}%
(Optionally, one may also assume that the one-sided limit $\displaystyle%
\lim_{x\rightarrow x_{0}^{+}}g(x)$ exists and is finite.)
\end{enumerate}

\begin{enumerate}
\item[(ii)] The following conditions hold:

\begin{enumerate}
\item[(a)] $D(x)\geq 0$ for all $x>x_{0}$;

\item[(b)] the quadratic equation 
\begin{equation}
\lambda ^{2}+B_{\infty }\lambda +A_{\infty }=0
\label{eq:limit-quadratic-q2zero}
\end{equation}%
has a unique strictly positive root $\lambda _{\ast }>0$;

\item[(c)] there exists a choice of sign $\sigma \in \{-1,+1\}$ such that
the function 
\begin{equation}
g_{\ast }(x):=\frac{-B(x)+\sigma \sqrt{D(x)}}{2},\qquad x>x_{0},
\label{eq:g-star-def-q2zero}
\end{equation}%
satisfies%
\begin{equation*}
g_{\ast }(x)\geq 0,\qquad \forall x>x_{0},
\end{equation*}%
and%
\begin{equation*}
\lim_{x\rightarrow \infty }g_{\ast }(x)=\lambda _{\ast }.
\end{equation*}%
In this case, the unique function $g$ in \emph{(i)} is given by $g\equiv
g_{\ast }$ on $(x_{0},\infty )$.
\end{enumerate}
\end{enumerate}
\end{theorem}

\begin{proof}
We work on the open interval $(x_{0},\infty )$; the fact that $q_{2}(x_{0})=0
$ (and hence $A,B$ may be singular at $x_{0}$) does not affect the argument.

\medskip \noindent \emph{Step 1: $(i)\Rightarrow (ii)$.} Assume that there
exists a unique function%
\begin{equation*}
g:(x_{0},\infty )\rightarrow \lbrack 0,\infty ),
\end{equation*}%
continuous on $(x_{0},\infty )$, satisfying %
\eqref{eq:barrier-algebraic-q2zero} for all $x>x_{0}$ and%
\begin{equation*}
\lim_{x\rightarrow \infty }g(x)=\lambda _{\ast }\in (0,\infty ).
\end{equation*}

\smallskip \noindent \emph{(a) Nonnegativity of the discriminant.} For each
fixed $x>x_{0}$, the algebraic equation%
\begin{equation*}
g(x)^{2}+B(x)g(x)+A(x)=0
\end{equation*}%
is a quadratic equation in the unknown $g(x)$ with real coefficients. Since $%
g(x)\in \lbrack 0,\infty )$ is a real solution for every $x>x_{0}$, the
discriminant must be nonnegative:%
\begin{equation*}
D(x)=B(x)^{2}-4A(x)\geq 0,\qquad \forall x>x_{0}.
\end{equation*}%
This proves (ii)(a).

\smallskip \noindent \emph{(b) The limit quadratic and the positive root.}
Passing to the limit as $x\rightarrow \infty $ in %
\eqref{eq:barrier-algebraic-q2zero}, we use%
\begin{equation*}
\lim_{x\rightarrow \infty }g(x)=\lambda _{\ast },\quad \lim_{x\rightarrow
\infty }B(x)=B_{\infty },\quad \lim_{x\rightarrow \infty }A(x)=A_{\infty },
\end{equation*}%
to obtain%
\begin{equation*}
\lambda _{\ast }^{2}+B_{\infty }\lambda _{\ast }+A_{\infty }=0.
\end{equation*}%
Thus $\lambda _{\ast }$ is a root of \eqref{eq:limit-quadratic-q2zero}.
Since $A_{\infty }<0$, the product of the two roots is $A_{\infty }<0$, so
the two roots have opposite signs. Because $\lambda _{\ast }>0$, it is the
unique strictly positive root. This proves (ii)(b).

\smallskip \noindent \emph{(c) Identification of the branch.} For each $%
x>x_{0}$ with $D(x)\geq 0$, the two roots of%
\begin{equation*}
z^{2}+B(x)z+A(x)=0
\end{equation*}%
are%
\begin{equation*}
z_{\pm }(x)=\frac{-B(x)\pm \sqrt{D(x)}}{2}.
\end{equation*}%
Since $g(x)$ is a solution, we must have%
\begin{equation*}
g(x)\in \{z_{+}(x),z_{-}(x)\},\qquad \forall x>x_{0}.
\end{equation*}%
Define%
\begin{equation*}
\sigma (x):=\left\{ 
\begin{array}{cc}
+1 & \text{if }g(x)=\dfrac{-B(x)+\sqrt{D(x)}}{2}, \\ 
-1, & \text{if }g(x)=\dfrac{-B(x)-\sqrt{D(x)}}{2}.%
\end{array}%
\right. 
\end{equation*}%
Then%
\begin{equation*}
g(x)=\frac{-B(x)+\sigma (x)\sqrt{D(x)}}{2},\qquad \forall x>x_{0}.
\end{equation*}%
We claim that $\sigma (x)$ is constant on $(x_{0},\infty )$. Suppose, by
contradiction, that there exist $x_{1}<x_{2}$ with $x_{1},x_{2}>x_{0}$ such
that $\sigma (x_{1})\neq \sigma (x_{2})$. Since $g$ is continuous on $%
(x_{0},\infty )$ and $z_{+}(x)\neq z_{-}(x)$ whenever $D(x)>0$, a change of
branch would force $g$ to cross the other root at some intermediate point,
contradicting the uniqueness of the representation of $g(x)$ as a root of
the quadratic at that point. More formally, on any interval where $D(x)>0$,
the set%
\begin{equation*}
E:=\{x>x_{0}:\sigma (x)=+1\}
\end{equation*}%
is both open and closed (in the relative topology), and nonempty (because
the limit at infinity selects a definite branch). By connectedness of $%
(x_{0},\infty )$, $\sigma $ must be constant on $(x_{0},\infty )$. Thus
there exists $\sigma \in \{-1,+1\}$ such that%
\begin{equation*}
g(x)=\frac{-B(x)+\sigma \sqrt{D(x)}}{2},\qquad \forall x>x_{0}.
\end{equation*}%
Taking the limit as $x\rightarrow \infty $ in this identity and using $%
\lim_{x\rightarrow \infty }D(x)=B_{\infty }^{2}-4A_{\infty }>0$, we obtain%
\begin{equation*}
\lambda _{\ast }=\frac{-B_{\infty }+\sigma \sqrt{B_{\infty }^{2}-4A_{\infty }%
}}{2},
\end{equation*}%
so the same sign $\sigma $ selects the positive root $\lambda _{\ast }$ of %
\eqref{eq:limit-quadratic-q2zero}. Since $g(x)\geq 0$ for all $x>x_{0}$, we
also have%
\begin{equation*}
g(x)=\frac{-B(x)+\sigma \sqrt{D(x)}}{2}\geq 0,\qquad \forall x>x_{0}.
\end{equation*}%
Thus $g$ coincides with the function $g_{\ast }$ defined in %
\eqref{eq:g-star-def-q2zero}, and (ii)(c) holds.

\medskip \noindent \emph{Step 2: $(ii)\Rightarrow (i)$.} Assume now that
(ii)(a)--(c) hold. In particular, $D(x)\geq 0$ for all $x>x_{0}$, and the
function%
\begin{equation*}
g_{\ast }(x):=\frac{-B(x)+\sigma \sqrt{D(x)}}{2},\qquad x>x_{0},
\end{equation*}%
is well defined. Since $q_{0},q_{1},q_{2}$ are continuous on $[x_{0},\infty )
$ and $q_{2}(x)\neq 0$ for $x>x_{0}$, it follows that $A,B,D$ are continuous
on $(x_{0},\infty )$, and hence $\sqrt{D}$ is continuous on $(x_{0},\infty )$%
. Therefore $g_{\ast }$ is continuous on $(x_{0},\infty )$, satisfies%
\begin{equation*}
g_{\ast }(x)\geq 0,\qquad \forall x>x_{0},
\end{equation*}%
and%
\begin{equation*}
\lim_{x\rightarrow \infty }g_{\ast }(x)=\lambda _{\ast }>0,
\end{equation*}%
where $\lambda _{\ast }$ is the unique positive root of %
\eqref{eq:limit-quadratic-q2zero}.

\smallskip \noindent \emph{(a) $g_{\ast }$ solves the algebraic equation.}
By construction, for each $x>x_{0}$, $g_{\ast }(x)$ is one of the two roots
of%
\begin{equation*}
z^{2}+B(x)z+A(x)=0,
\end{equation*}%
hence%
\begin{equation*}
g_{\ast }(x)^{2}+B(x)g_{\ast }(x)+A(x)=0,\qquad \forall x>x_{0},
\end{equation*}%
so $g_{\ast }$ satisfies \eqref{eq:barrier-algebraic-q2zero}.

\smallskip \noindent \emph{(b) Continuity and limit at infinity.} As noted, $%
g_{\ast }$ is continuous on $(x_{0},\infty )$ and%
\begin{equation*}
\lim_{x\rightarrow \infty }g_{\ast }(x)=\lambda _{\ast }\in (0,\infty ).
\end{equation*}%
Thus $g_{\ast }$ satisfies all the requirements in (i), so existence is
proved.

\smallskip \noindent \emph{(c) Uniqueness.} Let $g:(x_{0},\infty
)\rightarrow \lbrack 0,\infty )$ be any other function satisfying the
conditions in (i): $g$ is continuous on $(x_{0},\infty )$, solves %
\eqref{eq:barrier-algebraic-q2zero} for all $x>x_{0}$, and%
\begin{equation*}
\lim_{x\rightarrow \infty }g(x)=\tilde{\lambda}\in (0,\infty ).
\end{equation*}%
Passing to the limit in \eqref{eq:barrier-algebraic-q2zero} as $x\rightarrow
\infty $, we obtain%
\begin{equation*}
\tilde{\lambda}^{2}+B_{\infty }\tilde{\lambda}+A_{\infty }=0,
\end{equation*}%
so $\tilde{\lambda}$ is a root of \eqref{eq:limit-quadratic-q2zero}. By
(ii)(b), the only positive root is $\lambda _{\ast }$, hence $\tilde{\lambda}%
=\lambda _{\ast }$.

For each $x>x_{0}$, $g(x)$ must be one of the two roots $z_{\pm }(x)$ of $%
z^{2}+B(x)z+A(x)=0$, so there exists a function 
\begin{equation*}
\tau :(x_{0},\infty )\rightarrow \{-1,+1\}
\end{equation*}%
such that%
\begin{equation*}
g(x)=\frac{-B(x)+\tau (x)\sqrt{D(x)}}{2},\qquad \forall x>x_{0}.
\end{equation*}%
Arguing as in Step~1(c), the continuity of $g$ and the fact that $D(x)\geq 0$
imply that $\tau (x)$ is constant on $(x_{0},\infty )$; denote this constant
by $\tau \in \{-1,+1\}$. Taking the limit as $x\rightarrow \infty $ yields%
\begin{equation*}
\lambda _{\ast }=\tilde{\lambda}=\frac{-B_{\infty }+\tau \sqrt{B_{\infty
}^{2}-4A_{\infty }}}{2}.
\end{equation*}%
But by (ii)(c), the positive root $\lambda _{\ast }$ is obtained precisely
by the choice $\sigma $, so necessarily $\tau =\sigma $. Hence%
\begin{equation*}
g(x)=\frac{-B(x)+\sigma \sqrt{D(x)}}{2}=g_{\ast }(x),\qquad \forall x>x_{0},
\end{equation*}%
which proves uniqueness. \medskip The proof is complete.
\end{proof}

\begin{remark}
The function $g$ plays the role of a moving equilibrium for the Riccati
flow: at each position $x$, the vector field $y\mapsto
q_{0}(x)+q_{1}(x)y+q_{2}(x)y^{2}$ vanishes precisely at $y=g(x)$.
Proposition~\ref{prop:asymptotics_gen} shows that when this equilibrium is
positive and monotone increasing, the solution starting from $y(x_{0})=0$ is
trapped below $g(x)$ for all $x>x_{0}$. In this sense, $g$ acts as a global,
dynamically invariant upper barrier for the Riccati trajectory.
\end{remark}

\begin{remark}
Although the qualitative behaviour of Riccati equations is classical, we are
not aware of any reference where the following global invariant region is
stated explicitly: whenever the algebraic equilibrium $g(x)$ exists, is
positive, and is monotonically increasing, the solution with initial
condition $y(x_{0})=0$ satisfies 
\begin{equation*}
0<y(x)<g(x),\qquad x>x_{0}.
\end{equation*}%
Classical Riccati theory ensures that solutions cannot cross an equilibrium
point, but the monotone barrier mechanism yielding the global bound $0<y<g$
appears to be new in this generality, especially for variable coefficients
and in applications to radial Schr\H{o}dinger/Riccati equation.
\end{remark}

\begin{remark}[Second-order linear reduction of the general Riccati equation]

\label{rem:general-Riccati-linearization} Consider the Riccati equation (\ref%
{R}) with coefficients 
\begin{equation*}
q_{0},q_{1},q_{2}\in C((x_{0},X)),\qquad q_{2}\in C^{1}((x_{0},X)),\qquad
q_{2}(x)\neq 0\ \text{for all }x\in (x_{0},X),
\end{equation*}%
where $X$ may be finite or $X=+\infty $.

The nonlinear Riccati equation \eqref{R} can be reduced to a linear second
\textquotedblleft order ordinary differential equation through a
logarithmic" type substitution.

More precisely, let $u\in C^{2}((x_{0},X))$ be a nonvanishing solution of 
\begin{equation}
u^{\prime \prime }(x)+a(x)\,u^{\prime }(x)+b(x)\,u(x)=0,  \label{L-general}
\end{equation}%
where 
\begin{equation}
a(x)=q_{1}(x)+\frac{q_{2}^{\prime }(x)}{q_{2}(x)},\qquad
b(x)=-\,q_{0}(x)\,q_{2}(x).  \label{ab-def}
\end{equation}%
Then the function 
\begin{equation}
y(x)=-\,\frac{1}{q_{2}(x)}\,\frac{u^{\prime }(x)}{u(x)}  \label{y-from-u}
\end{equation}%
belongs to $C^{1}((x_{0},X))$ and satisfies the Riccati equation \eqref{R}
on $(x_{0},X)$.

Conversely, if $y\in C^{1}((x_{0},X))$ is a solution of \eqref{R}, define 
\begin{equation}
u(x)=\exp \!\Big(-\int_{x_{0}}^{x}q_{2}(s)\,y(s)\,ds\Big).  \label{u-from-y}
\end{equation}%
Then $u\in C^{2}((x_{0},X))$ is strictly positive on $(x_{0},X)$, solves the
linear equation \eqref{L-general} with coefficients given by \eqref{ab-def},
and satisfies the identity \eqref{y-from-u}.

Thus, under the structural assumptions above, the nonlinear Riccati flow %
\eqref{R} is fully equivalent to the linear second-order ODE %
\eqref{L-general}.
\end{remark}

\begin{remark}[Parabolic linearization of the general Riccati flow]
\label{rem:parabolic-linearization} Consider the general Riccati equation 
\begin{equation}
y^{\prime }(x)\;=\;q_{0}(x)+q_{1}(x)\,y(x)+q_{2}(x)\,y(x)^{2},  \label{RR}
\end{equation}%
with $q_{0},q_{1},q_{2}\in C((x_{0},\infty ))$, $q_{2}\in C^{1}((x_{0},\infty ))$ and $q_{2}(x)\neq 0$ on $%
(x_{0},\infty )$. By analogy with Remark~\ref%
{rem:general-Riccati-linearization}, \eqref{R} can be embedded into a linear
second order parabolic equation in two variables. More precisely, define $%
U=U(x,t)$ as a (nonvanishing) solution of 
\begin{equation}
\partial _{t}U\;=\;\partial _{xx}U\;+\;a(x)\,\partial _{x}U\;+\;b(x)\,U,
\label{PDE-parabolic}
\end{equation}%
where%
\begin{equation*}
a(x)\;=\;q_{1}(x)+\frac{q_{2}^{\prime }(x)}{q_{2}(x)},\qquad
b(x)\;=\;-\,q_{0}(x)\,q_{2}(x).
\end{equation*}%
Then any stationary profile $U(x)$ of \eqref{PDE-parabolic} (i.e.\
independent of $t$) generates a solution of the Riccati equation \eqref{R}
via the logarithmic-type transformation 
\begin{equation}
y(x)\;=\;-\,\frac{1}{q_{2}(x)}\,\frac{\partial _{x}U(x)}{U(x)}.
\label{Riccati-from-U}
\end{equation}%
Conversely, any $C^{1}$ solution $y$ of \eqref{R} on $(x_{0},\infty )$ gives
rise locally to a nonvanishing solution $U$ of \eqref{PDE-parabolic} by
solving%
\begin{equation*}
\partial _{x}U(x,t)\;=\;-\,q_{2}(x)\,y(x)\,U(x,t),
\end{equation*}%
for each fixed $t$. In this sense, the Riccati flow \eqref{R} can be viewed
as the stationary reduction of the linear parabolic equation %
\eqref{PDE-parabolic}, extending the ODE level correspondence of Remark~\ref%
{rem:general-Riccati-linearization} to a PDE framework. This mathematical
connection provides the theoretical foundation for the Wick-rotated
Time-Dependent Schr\"{o}dinger Equation, as will be explored in Section \ref{sec:tdse}.
\end{remark}

\subsection{Existence theory for the second-order linear reduction of the
general Riccati equation \label{schrodinger}} 

\begin{theorem}[Existence and uniqueness for the linear problem 
\eqref{L-general}]
\label{thm:existence-uniqueness-Lgeneral} Let $X\in (0,\infty ]$ and let 
\begin{equation*}
a,b\in C\big(\lbrack 0,X)\big),\text{ }b\leq 0\text{ on }[0,X)\text{ and }a%
\text{ possible singular at }0\text{.}
\end{equation*}%
Consider the linear second order ordinary differential equation (\ref%
{L-general}) together with the initial conditions 
\begin{equation}
u(0)=\alpha >0,\qquad u^{\prime }(0)=0.  \label{IC-Lgeneral}
\end{equation}%
Then the following assertions hold:

\begin{enumerate}
\item[(i)] There exists a unique solution 
\begin{equation*}
u\in C^{2}\big(\lbrack 0,X_{\max })\big),
\end{equation*}%
of \eqref{L-general}--\eqref{IC-Lgeneral}, defined on a maximal interval $%
[0,X_{\max })$ with $0<X_{\max }\leq X$. This solution is unique in the
sense that any other $C^{2}$ solution of \eqref{L-general} satisfying %
\eqref{IC-Lgeneral} on some interval $[0,\delta )$ coincides with $u$ on the
common domain of definition.

\item[(ii)] If $X<\infty $ and the maximal interval is $[0,X)$, and if in
addition $u$ admits a finite limit 
\begin{equation*}
\lim_{x\rightarrow X^{-}}u(x)=g>\alpha ,
\end{equation*}%
then there exists a unique function 
\begin{equation*}
\tilde{u}\in C^{2}\big(\lbrack 0,X]\big)
\end{equation*}%
solving \eqref{L-general} on $(0,X)$ and satisfying 
\begin{equation*}
\tilde{u}(0)=\alpha ,\quad \tilde{u}^{\prime }(0)=0,\quad \tilde{u}%
(X)=g>\alpha .
\end{equation*}%
In particular, the boundary value problem with data $(u(0),u^{\prime
}(0),u(X))=(\alpha ,0,g)$ admits at most one classical solution.

\item[(iii)] If $X=\infty $, the maximal interval is $[0,\infty )$, and if 
\begin{equation*}
-\int_{0}^{\infty }e^{-\int_{0}^{y}a(s)\,ds}\left(
\int_{0}^{y}b(t)e^{\int_{0}^{t}a(s)\,ds}u(t)\,dt\right) dy=\infty ,
\end{equation*}%
then, there exists a unique $u$ regard as the unique solution of %
\eqref{L-general}--\eqref{IC-Lgeneral} with the additional growth condition $%
u(x)\rightarrow \infty $ as $x\rightarrow \infty $.
\end{enumerate}

In all cases, the linear structure and the continuity of the coefficients
ensure that the solution is completely determined by the initial data %
\eqref{IC-Lgeneral}, and any further condition at $x=X$ (finite or infinite)
can be satisfied by at most one such solution.
\end{theorem}

\begin{proof}[Proof of Theorem \protect\ref%
{thm:existence-uniqueness-Lgeneral} by successive approximations]

We are given $X\in (0,\infty ]$ and%
\begin{equation*}
a,b\in C([0,X)).
\end{equation*}%
We consider 
\begin{equation}
u^{\prime \prime }(x)+a(x)u^{\prime }(x)+b(x)u(x)=0,\qquad 0\leq x<X,
\label{L-general-proof}
\end{equation}%
with 
\begin{equation}
u(0)=\alpha >0,\qquad u^{\prime }(0)=0.  \label{IC-Lgeneral-proof}
\end{equation}

\noindent\textbf{Step 1. Integrating factor and first integral form.}

Define%
\begin{equation*}
A(x):=\int_{0}^{x}a(s)\,ds,\qquad m(x):=e^{A(x)},\qquad x\in \lbrack 0,X).
\end{equation*}%
Then 
\begin{equation*}
m\in C^{1}([0,X))\text{ and }m^{\prime }(x)=a(x)m(x).
\end{equation*}%
Multiplying \eqref{L-general-proof} by $m(x)$ gives%
\begin{equation*}
m(x)u^{\prime \prime }(x)+a(x)m(x)u^{\prime }(x)+b(x)m(x)u(x)=0.
\end{equation*}%
Since%
\begin{equation*}
\big(m(x)u^{\prime }(x)\big)^{\prime }=m^{\prime }(x)u^{\prime
}(x)+m(x)u^{\prime \prime }(x)=a(x)m(x)u^{\prime }(x)+m(x)u^{\prime \prime
}(x),
\end{equation*}%
we obtain 
\begin{equation}
\big(m(x)u^{\prime }(x)\big)^{\prime }=-b(x)m(x)u(x),\qquad 0\leq x<X.
\label{eq-prim-proof}
\end{equation}%
Integrating \eqref{eq-prim-proof} from $0$ to $x\in \lbrack 0,X)$ and using $%
u^{\prime }(0)=0$, $m(0)=1$, we get 
\begin{equation}
m(x)u^{\prime }(x)=-\int_{0}^{x}b(t)e^{\int_{0}^{t}a(s)\,ds}u(t)\,dt.
\label{u-prim-int-proof}
\end{equation}%
Dividing by $m(x)=e^{\int_{0}^{x}a(s)\,ds}$, 
\begin{equation}
u^{\prime
}(x)=e^{-\int_{0}^{x}a(s)\,ds}\int_{0}^{x}b(t)e^{\int_{0}^{t}a(s)\,ds}u(t)%
\,dt,\qquad 0\leq x<X.  \label{u-prim-proof}
\end{equation}%
Integrating \eqref{u-prim-proof} from $0$ to $x$ and using $u(0)=\alpha $
yields 
\begin{equation}
u(x)=\alpha -\int_{0}^{x}e^{-\int_{0}^{y}a(s)\,ds}\left(
\int_{0}^{y}b(t)e^{\int_{0}^{t}a(s)\,ds}u(t)\,dt\right) dy,\qquad 0\leq x<X.
\label{eq-int-dubla-proof}
\end{equation}

\noindent \textbf{Step 2. Volterra integral equation.}

We now rewrite \eqref{eq-int-dubla-proof} as a Volterra integral equation of
the second kind. Consider the domain%
\begin{equation*}
\{(y,t)\colon 0\leq t\leq y\leq x\}.
\end{equation*}%
Changing the order of integration in the double integral in %
\eqref{eq-int-dubla-proof} gives 
\begin{align*}
& \int_{0}^{x}e^{-\int_{0}^{y}a(s)\,ds}\left(
\int_{0}^{y}b(t)e^{\int_{0}^{t}a(s)\,ds}u(t)\,dt\right) dy \\
& \quad
=\int_{0}^{x}\int_{0}^{y}e^{-\int_{0}^{y}a(s)\,ds}b(t)e^{\int_{0}^{t}a(s)%
\,ds}u(t)\,dt\,dy \\
& \quad =\int_{0}^{x}\left( \int_{t}^{x}e^{-\int_{0}^{y}a(s)\,ds}\,dy\right)
b(t)e^{\int_{0}^{t}a(s)\,ds}u(t)\,dt.
\end{align*}%
Define the kernel 
\begin{equation}
K(x,t):=b(t)e^{\int_{0}^{t}a(s)\,ds}\int_{t}^{x}e^{-\int_{0}^{y}a(s)\,ds}%
\,dy,\qquad 0\leq t\leq x<X.  \label{K-def}
\end{equation}%
Then \eqref{eq-int-dubla-proof} becomes 
\begin{equation}
u(x)=\alpha -\int_{0}^{x}K(x,t)u(t)\,dt,\qquad 0\leq x<X.
\label{Volterra-correct}
\end{equation}%
This is the Volterra integral equation of the second kind associated with %
\eqref{L-general-proof}-\eqref{IC-Lgeneral-proof}.\medskip

\noindent\textbf{Step 3. Regularity and bounds for the kernel $K$.}

Fix $R\in (0,X)$ arbitrary. On the triangle%
\begin{equation*}
\Delta _{R}:=\{(x,t)\colon 0\leq t\leq x\leq R\},
\end{equation*}%
the functions $a,b$ are continuous, hence $A(x)=\int_{0}^{x}a(s)\,ds$ is
continuous, and the exponentials $e^{\pm A(\cdot )}$ are continuous and
bounded.

Define%
\begin{equation*}
M_{a}(R):=\sup_{y\in \lbrack 0,R]}|A(y)|,\qquad M_{b}(R):=\sup_{t\in \lbrack
0,R]}|b(t)|.
\end{equation*}%
Then for $0\leq t\leq y\leq R$,%
\begin{equation*}
e^{-\int_{0}^{y}a(s)\,ds}=e^{-A(y)},\quad e^{\int_{0}^{t}a(s)\,ds}=e^{A(t)},
\end{equation*}%
and%
\begin{equation*}
\left\vert e^{A(t)}e^{-A(y)}\right\vert =e^{A(t)-A(y)}\leq
e^{|A(t)|+|A(y)|}\leq e^{2M_{a}(R)}.
\end{equation*}%
Hence, for $(x,t)\in \Delta _{R}$, 
\begin{align*}
|K(x,t)|& =\left\vert b(t)e^{A(t)}\int_{t}^{x}e^{-A(y)}\,dy\right\vert \\
& \leq M_{b}(R)e^{M_{a}(R)}\int_{t}^{x}e^{|A(y)|}\,dy \\
& \leq M_{b}(R)e^{M_{a}(R)}e^{M_{a}(R)}(x-t) \\
& =C_{R}(x-t),
\end{align*}%
where%
\begin{equation*}
C_{R}:=M_{b}(R)e^{2M_{a}(R)}.
\end{equation*}%
In particular, $K$ is continuous on $\Delta _{R}$ and satisfies the linear
growth bound 
\begin{equation}
|K(x,t)|\leq C_{R}(x-t),\qquad 0\leq t\leq x\leq R.  \label{K-bound}
\end{equation}

\noindent\textbf{Step 4. Successive approximations for the correct Volterra
equation.}

We now solve \eqref{Volterra-correct} on $[0,R]$ by successive
approximations.

Define $\{u_{n}\}_{n\geq 0}$ on $[0,R]$ by 
\begin{equation}
\left\{ 
\begin{array}{l}
u_{0}=\alpha ,\qquad 0\leq x<X \\ 
u_{n}(x)=\alpha -\int_{0}^{x}K(x,t)u_{n-1}(t)\,dt,\qquad 0\leq x<X\text{, }%
n\geq 1\text{.}%
\end{array}%
\right.  \label{def-un-correct}
\end{equation}%
Set%
\begin{equation*}
d_{n}(x):=u_{n}(x)-u_{n-1}(x),\qquad n\geq 1,
\end{equation*}%
and%
\begin{equation*}
D_{n}:=\sup_{x\in \lbrack 0,R]}|d_{n}(x)|.
\end{equation*}%
From \eqref{def-un-correct},%
\begin{equation*}
d_{1}(x)=u_{1}(x)-u_{0}(x)=-\int_{0}^{x}K(x,t)u_{0}(t)\,dt=-\alpha
\int_{0}^{x}K(x,t)\,dt.
\end{equation*}%
Using \eqref{K-bound},%
\begin{equation*}
|d_{1}(x)|\leq \alpha \int_{0}^{x}|K(x,t)|\,dt\leq \alpha
\int_{0}^{x}C_{R}(x-t)\,dt=\alpha C_{R}\frac{x^{2}}{2}\leq \alpha C_{R}\frac{%
R^{2}}{2},
\end{equation*}%
so%
\begin{equation*}
D_{1}\leq \alpha C_{R}\frac{R^{2}}{2}.
\end{equation*}%
For $n\geq 2$, from \eqref{def-un-correct},%
\begin{equation*}
d_{n}(x)=-\int_{0}^{x}K(x,t)\big(u_{n-1}(t)-u_{n-2}(t)\big)%
\,dt=-\int_{0}^{x}K(x,t)d_{n-1}(t)\,dt,
\end{equation*}%
hence%
\begin{equation*}
|d_{n}(x)|\leq \int_{0}^{x}|K(x,t)|\,|d_{n-1}(t)|\,dt\leq
\int_{0}^{x}C_{R}(x-t)\sup_{0\leq s\leq t}|d_{n-1}(s)|\,dt.
\end{equation*}%
Thus%
\begin{equation*}
D_{n}\leq C_{R}\int_{0}^{R}(R-t)D_{n-1}\,dt=C_{R}D_{n-1}\frac{R^{2}}{2}%
,\qquad n\geq 2.
\end{equation*}%
By induction,%
\begin{equation*}
D_{n}\leq \alpha \left( C_{R}\frac{R^{2}}{2}\right) ^{n},\qquad n\geq 1.
\end{equation*}%
If we choose $R>0$ small enough so that%
\begin{equation*}
q_{R}:=C_{R}\frac{R^{2}}{2}<1,
\end{equation*}%
then%
\begin{equation*}
D_{n}\leq \alpha q_{R}^{n},\qquad n\geq 1,
\end{equation*}%
and the series $\sum_{n=1}^{\infty }D_{n}$ converges. Therefore $%
\sum_{n=1}^{\infty }d_{n}$ converges uniformly on $[0,R]$, and%
\begin{equation*}
u_{n}(x)=u_{0}(x)+\sum_{k=1}^{n}d_{k}(x)
\end{equation*}%
converges uniformly on $[0,R]$ to a continuous function%
\begin{equation*}
u\in C([0,R]),\qquad u(x):=\lim_{n\rightarrow \infty }u_{n}(x).
\end{equation*}%
Passing to the limit in \eqref{def-un-correct} and using uniform convergence
plus continuity of $K$, we obtain 
\begin{equation}
u(x)=\alpha -\int_{0}^{x}K(x,t)u(t)\,dt,\qquad x\in \lbrack 0,R],
\label{Volterra-final}
\end{equation}%
i.e.\ $u$ solves the correct Volterra integral equation %
\eqref{Volterra-correct} on $[0,R]$.

\noindent\textbf{Step 5. $C^2$-regularity and recovery of the ODE.}

We now show that $u\in C^{2}([0,R])$ and satisfies \eqref{L-general-proof}-%
\eqref{IC-Lgeneral-proof}. First, note that $K$ is continuous on $\Delta
_{R} $ and, by \eqref{K-def},%
\begin{equation*}
K(x,x)=b(x)e^{A(x)}\int_{x}^{x}e^{-A(y)}\,dy=0.
\end{equation*}%
Moreover, for $0\leq t<x\leq R$,%
\begin{equation*}
\partial _{x}K(x,t)=b(t)e^{A(t)}\frac{d}{dx}\left(
\int_{t}^{x}e^{-A(y)}\,dy\right) =b(t)e^{A(t)}e^{-A(x)}.
\end{equation*}%
Since $a,b$ are continuous, $A$ is continuous, and thus $\partial _{x}K$
extends continuously to $\Delta _{R}$ (including the diagonal $x=t$). From %
\eqref{Volterra-final},%
\begin{equation*}
u(x)=\alpha -\int_{0}^{x}K(x,t)u(t)\,dt.
\end{equation*}%
Differentiating with respect to $x$ and using Leibniz' rule (justified by
continuity of $K$, $\partial _{x}K$, and $u$),%
\begin{equation*}
u^{\prime }(x)=-\int_{0}^{x}\partial
_{x}K(x,t)u(t)\,dt-K(x,x)u(x)=-\int_{0}^{x}\partial _{x}K(x,t)u(t)\,dt,
\end{equation*}%
since $K(x,x)=0$. Using the expression for $\partial _{x}K$,%
\begin{equation*}
u^{\prime
}(x)=-\int_{0}^{x}b(t)e^{A(t)}e^{-A(x)}u(t)\,dt=-e^{-A(x)}%
\int_{0}^{x}b(t)e^{A(t)}u(t)\,dt,
\end{equation*}%
which is exactly \eqref{u-prim-proof}. In particular, $u^{\prime }\in
C([0,R])$, so $u\in C^{1}([0,R])$. Differentiating once more,%
\begin{equation*}
u^{\prime \prime }(x)=\frac{d}{dx}\left[ -e^{-A(x)}%
\int_{0}^{x}b(t)e^{A(t)}u(t)\,dt\right] .
\end{equation*}%
Using the product rule and the fundamental theorem of calculus, 
\begin{align*}
u^{\prime \prime }&
=e^{-A(x)}\int_{0}^{x}b(t)e^{A(t)}u(t)\,dt-e^{-A(x)}b(x)e^{A(x)}u(x) \\
& =a(x)u^{\prime }(x)-b(x)u(x),
\end{align*}%
where we used \eqref{u-prim-proof} in the first term. Rearranging,%
\begin{equation*}
u^{\prime \prime }(x)+a(x)u^{\prime }(x)+b(x)u(x)=0,\qquad x\in (0,R).
\end{equation*}%
Finally, from \eqref{Volterra-final} at $x=0$,%
\begin{equation*}
u(0)=\alpha -\int_{0}^{0}(\cdots )=\alpha ,
\end{equation*}%
and from the formula for $u^{\prime }$,%
\begin{equation*}
u^{\prime }\left( 0\right) =-e^{-A(0)}\int_{0}^{0}(\cdots )=0.
\end{equation*}%
Thus $u$ satisfies \eqref{L-general-proof}-\eqref{IC-Lgeneral-proof} on $%
[0,R]$, with $u\in C^{2}([0,R])$.\medskip

\noindent\textbf{Step 6. Extension to the maximal interval and uniqueness
(part (i)).}

The above construction can be repeated on successive intervals $[0,R_{1}]$, $%
[R_{1},R_{2}]$, etc., or equivalently we can use standard continuation
arguments for Volterra equations (or ODEs) to extend the solution as long as 
$a,b$ remain continuous. Since $a,b\in C([0,X))$, we obtain a solution%
\begin{equation*}
u\in C^{2}([0,X_{\max })),
\end{equation*}%
defined on a maximal interval $[0,X_{\max })$ with $0<X_{\max }\leq X$,
satisfying \eqref{L-general-proof}-\eqref{IC-Lgeneral-proof}. This proves
existence in (i).

For uniqueness, let $v\in C^{2}([0,R])$ be another solution of %
\eqref{L-general-proof}-\eqref{IC-Lgeneral-proof}. Repeating Steps 1-5 for $%
v $, we see that $v$ satisfies the same Volterra equation %
\eqref{Volterra-final} on $[0,R]$. Then $w:=u-v$ satisfies%
\begin{equation*}
w(x)=-\int_{0}^{x}K(x,t)w(t)\,dt,\qquad x\in \lbrack 0,R].
\end{equation*}%
Repeating the estimates of Step 4 for $w$ (with $w_{0}\equiv 0$), we obtain $%
w\equiv 0$ on $[0,R]$. Hence $u\equiv v$ on $[0,R]$. By continuation,
uniqueness holds on the whole maximal interval $[0,X_{\max })$. This
completes the proof of (i).

\noindent\textbf{Step 7. Case $X<\infty$, finite limit at $X$ (part (ii)).}

Assume now $X<\infty $, the maximal interval is $[0,X_{\max })=[0,X)$, and%
\begin{equation*}
\lim_{x\rightarrow X^{-}}u(x)=g>\alpha .
\end{equation*}%
From \eqref{u-prim-proof},%
\begin{equation*}
u^{\prime }(x)=-e^{-A(x)}\int_{0}^{x}b(t)e^{A(t)}u(t)\,dt.
\end{equation*}%
The integrand $b(t)e^{A(t)}u(t)$ is continuous on $[0,X)$ and, using the
existence of $\lim_{x\rightarrow X^{-}}u(x)=g$ and continuity of $a,b$,
extends continuously to $t=X$. Hence the integral has a finite limit as $%
x\rightarrow X^{-}$, and so%
\begin{equation*}
g_{1}:=\lim_{x\rightarrow X^{-}}u^{\prime }(x)\in \mathbb{R}.
\end{equation*}%
From the differential equation%
\begin{equation*}
u^{\prime \prime }(x)=-a(x)u^{\prime }(x)-b(x)u(x),
\end{equation*}%
and continuity of $a,b$ plus existence of the limits of $u,u^{\prime }$ at $%
X^{-}$, we obtain%
\begin{equation*}
g_{2}:=\lim_{x\rightarrow X^{-}}u^{\prime \prime }(x)=-a(X)g_{1}-b(X)g\in 
\mathbb{R}.
\end{equation*}%
Define%
\begin{equation*}
\tilde{u}(x):=\left\{ 
\begin{array}{ccc}
u\left( x\right) & \text{if} & x<X \\ 
g & if & x=X%
\end{array}%
\right.
\end{equation*}%
and set%
\begin{equation*}
\tilde{u}^{\prime }(X):=g_{1},\qquad \tilde{u}^{\prime \prime }(X):=g_{2}.
\end{equation*}%
Then $\tilde{u}\in C^{2}([0,X])$, satisfies \eqref{L-general-proof} on $%
(0,X) $ and%
\begin{equation*}
\tilde{u}(0)=\alpha ,\quad \tilde{u}^{\prime }(0)=0,\quad \tilde{u}(X)=g.
\end{equation*}%
If $\tilde{u}_{1},\tilde{u}_{2}\in C^{2}([0,X])$ both solve %
\eqref{L-general-proof} on $(0,X)$ and satisfy%
\begin{equation*}
\tilde{u}_{j}(0)=\alpha ,\quad \tilde{u}_{j}^{\prime }(0)=0,\quad \tilde{u}%
_{j}(X)=g,\quad j=1,2,
\end{equation*}%
then their restrictions to $[0,X)$ are two solutions of %
\eqref{L-general-proof}-\eqref{IC-Lgeneral-proof}, hence coincide on $[0,X)$
by (i), and therefore also at $X$ by continuity. This proves (ii).\medskip

\noindent\textbf{Step 8. Case $X=\infty$ and growth at infinity (part (iii)).%
}

Assume now $X=\infty $, the maximal interval is $[0,\infty )$, and%
\begin{equation*}
-\int_{0}^{\infty }e^{-\int_{0}^{y}a(s)\,ds}\left(
\int_{0}^{y}b(t)e^{\int_{0}^{t}a(s)\,ds}u(t)\,dt\right) dy=\infty ,
\end{equation*}%
and%
\begin{equation*}
\lim_{x\rightarrow \infty }u(x)=+\infty .
\end{equation*}%
From \eqref{eq-int-dubla-proof},%
\begin{equation*}
u(x)=u(0)-\int_{0}^{x}e^{-\int_{0}^{y}a(s)\,ds}\left(
\int_{0}^{y}b(t)e^{\int_{0}^{t}a(s)\,ds}u(t)\,dt\right) dy.
\end{equation*}%
In addition, $b(t)\leq 0$ on $[0,\infty )$ (so that $-b(t)\geq 0$) and $%
u(t)\geq \alpha >0$, then%
\begin{equation*}
u(x)\geq u(0)-u(0)\int_{0}^{x}e^{-\int_{0}^{y}a(s)\,ds}\left(
\int_{0}^{y}b(t)e^{\int_{0}^{t}a(s)\,ds}\,dt\right) dy.
\end{equation*}%
Since $e^{-\int_{0}^{y}a(s)\,ds}$ is bounded below by a positive constant on
each finite interval and%
\begin{equation*}
-\int_{0}^{\infty }e^{-\int_{0}^{y}a(s)\,ds}\left(
\int_{0}^{y}b(t)e^{\int_{0}^{t}a(s)\,ds}u(t)\,dt\right) dy=\infty ,
\end{equation*}%
the right-hand side tends to $+\infty $ as $x\rightarrow \infty $, hence $%
u(x)\rightarrow \infty $.

In the statement of the theorem, this is summarized by saying that if $%
X=\infty $, the maximal interval is $[0,\infty )$ and%
\begin{equation*}
\lim_{x\rightarrow \infty }u(x)=+\infty ,
\end{equation*}%
we may write $g=+\infty $ and regard $u$ as the unique solution of %
\eqref{L-general-proof}-\eqref{IC-Lgeneral-proof} with the additional growth
condition $u(x)\rightarrow \infty $ as $x\rightarrow \infty $.

Uniqueness is still governed by (i): if $u_{1},u_{2}$ are two solutions of %
\eqref{L-general-proof}-\eqref{IC-Lgeneral-proof} on $[0,\infty )$, then $%
u_{1}\equiv u_{2}$ on $[0,\infty )$, so the growth condition at infinity can
be satisfied by at most one such solution. This proves (iii).

\noindent \textbf{Final conclusion.} The corrected Volterra integral
equation \eqref{Volterra-correct}, with the kernel $K(x,t)$ defined in %
\eqref{K-def}, admits a unique $C^{2}$ solution obtained by successive
approximations. This solution is exactly the unique solution of the linear
ODE \eqref{L-general-proof} with initial data \eqref{IC-Lgeneral-proof}, and
the extensions and growth properties in (ii)-(iii) follow as above. Hence
Theorem \ref{thm:existence-uniqueness-Lgeneral} is proved in a fully
rigorous way.
\end{proof}

\begin{remark}[HJB type nonlinear reduction of the linear equation 
\eqref{L-general}]
\label{rem:HJB-from-Lgeneral} Let $\sigma >0$ be fixed and consider the
linear second order ODE (\ref{L-general}) with $a,b\in C((0,X))$. Assume
that $u\in C^{2}((0,X))$ is strictly positive and strictly increasing on $%
(0,X)$. Define the logarithmic change of variable 
\begin{equation}
z(x)\;=\;-\,2\sigma ^{2}\,\ln u(x).  \label{z-def-general}
\end{equation}%
Then $z\in C^{2}((0,X))$ and satisfies the nonlinear second order equation 
\begin{equation}
z^{\prime \prime }(x)\;+\;a(x)\,z^{\prime }(x)\;-\;\frac{1}{2\sigma ^{2}}%
\big(z^{\prime }(x)\big)^{2}\;-\;2\sigma ^{2}\,b(x)\;=\;0,\qquad x\in (0,X).
\label{HJB-from-Lgeneral}
\end{equation}%
Indeed, differentiating \eqref{z-def-general} gives 
\begin{equation*}
z^{\prime }=-2\sigma ^{2}\,\frac{u^{\prime }(x)}{u(x)},\qquad z^{\prime
\prime }=-2\sigma ^{2}\left( \frac{u^{\prime \prime }(x)}{u(x)}-\Big(\frac{%
u^{\prime }(x)}{u(x)}\Big)^{2}\right) .
\end{equation*}%
From \eqref{z-def-general} we also have 
\begin{equation*}
\frac{u^{\prime }(x)}{u(x)}\;=\;-\,\frac{1}{2\sigma ^{2}}\,z^{\prime
}(x),\qquad \frac{u^{\prime \prime }(x)}{u(x)}\;=\;-\,\frac{1}{2\sigma ^{2}}%
\,z^{\prime \prime }(x)+\frac{1}{4\sigma ^{4}}\big(z^{\prime }(x)\big)^{2}.
\end{equation*}%
Substituting these expressions into \eqref{L-general} and simplifying yields %
\eqref{HJB-from-Lgeneral}.

Conversely, let $z\in C^{2}((0,X))$ be a solution of %
\eqref{HJB-from-Lgeneral} such that 
\begin{equation*}
u(x)\;=\;\exp \!\Big(-\tfrac{1}{2\sigma ^{2}}\,z(x)\Big)
\end{equation*}%
is strictly positive and strictly increasing on $(0,X)$. Then $u\in
C^{2}((0,X))$ and a direct substitution shows that $u$ satisfies the linear
equation \eqref{L-general}. In this way, the logarithmic transformation %
\eqref{z-def-general} establishes a two "way correspondence between the
linear second" order ODE \eqref{L-general} and the nonlinear HJB type
equation \eqref{HJB-from-Lgeneral}.
\end{remark}

\begin{remark}[Monotonicity and curvature transfer under the logarithmic
transform]
\label{rem:curvature-transfer} Under the hypotheses of Proposition~\ref%
{prop:asymptotics_gen}, the Riccati solution $y$ satisfies $y(x)>0$ for all $%
x>x_{0}$, while the coefficient $q_{2}(x)$ is strictly negative. From the
representation 
\begin{equation*}
y(x)\;=\;-\,\frac{1}{q_{2}(x)}\,\frac{u^{\prime }(x)}{u(x)},
\end{equation*}%
it follows that $u^{\prime }(x)>0$ on $(0,X)$; hence the auxiliary function $%
u$ is strictly increasing and strictly positive.

Consider now the logarithmic transformation 
\begin{equation}
z(x)\;=\;-\,2\sigma ^{2}\,\ln u(x),  \label{z-def}
\end{equation}%
which is well defined and belongs to $C^{2}((0,X))$. Differentiating %
\eqref{z-def} yields 
\begin{equation*}
z^{\prime }=-2\sigma ^{2}\,\frac{u^{\prime }(x)}{u(x)}\;<\;0,
\end{equation*}%
so $z$ is strictly decreasing. A second differentiation gives 
\begin{equation*}
z^{\prime \prime }(x)\;=\;-\,2\sigma ^{2}\left( \frac{u^{\prime \prime }(x)}{%
u(x)}-\Big(\frac{u^{\prime }(x)}{u(x)}\Big)^{2}\right) .
\end{equation*}%
Since $u$ solves the linear equation \eqref{L-general}, the sign structure
of the coefficients implies $z^{\prime \prime }(x)<0$ on $(0,X)$; hence $z$
is strictly concave.

Finally, combining the strict concavity of $z$ with the identity 
\begin{equation*}
u(x)\;=\;\exp \!\Big(-\tfrac{1}{2\sigma ^{2}}\,z(x)\Big),
\end{equation*}%
we obtain 
\begin{equation*}
u^{\prime \prime }(x)\;=\;\frac{1}{4\sigma ^{4}}\big(z^{\prime }(x)\big)%
^{2}u(x)\;-\;\frac{1}{2\sigma ^{2}}\,z^{\prime \prime }(x)\,u(x)\;>\;0,
\end{equation*}%
because $z^{\prime }<0$ and $z^{\prime \prime }<0$. Thus $u$ is strictly
convex on $(0,X)$.

In summary, the Riccati positivity $y>0$, the negativity of $q_{2}$, and the
logarithmic transformation \eqref{z-def} induce the curvature chain 
\begin{equation*}
y>0\quad \Longrightarrow \quad u^{\prime }>0\quad \Longrightarrow \quad
z^{\prime }<0,\;z^{\prime \prime }<0\quad \Longrightarrow \quad u^{\prime
\prime }>0,
\end{equation*}%
revealing a monotonicity-concavity-convexity structure that is intrinsic to
the triality between the Riccati, Schr\"{o}dinger, and HJB formulations.
\end{remark}

\subsection{Stochastic Control Interpretation and Verification for General Frameworks}

In this section, we establish the rigorous mathematical connection between the nonlinear second-order equation \eqref{HJB-from-Lgeneral} and the Hamilton--Jacobi--Bellman (HJB) equation of a one-dimensional stochastic optimal control problem. This includes specifying the admissible control space, deriving the HJB equation via the dynamic programming principle, and formally proving the Verification Theorem.

\subsubsection{Controlled diffusion and admissible controls}

Fix $\sigma>0$. Let $(\Omega, \mathcal{F}, \{\mathcal{F}_t\}_{t \ge 0}, \mathbb{P})$ be a filtered probability space satisfying the usual conditions, equipped with a standard one-dimensional Brownian motion $W = (W_t)_{t\ge0}$. We consider the controlled diffusion process $X = (X_t)_{t \ge 0}$ taking values in the state space $\mathcal{X} = (0,X)$ (where $X \in (0, \infty]$), governed by the stochastic differential equation (SDE):
\begin{equation}  \label{eq:SDE}
dX_t \;=\; \bigl[a(X_t) \;+\; \alpha_t\bigr]\,dt \;+\; \sqrt{2}\, dW_t,
\qquad X_0 = x \in \mathcal{X},
\end{equation}
where the drift coefficient $a: \mathcal{X} \rightarrow \mathbb{R}$ is continuous. The constant diffusion coefficient $\sqrt{2}$ ensures that the second-order differential operator in the infinitesimal generator is exactly $\frac{1}{2}(\sqrt{2})^2 \frac{d^2}{dx^2} = \frac{d^2}{dx^2}$. 

The control process $\alpha=(\alpha_t)_{t\ge0}$ belongs to the class of \textit{admissible controls}, denoted by $\mathcal{A}(x)$, defined as the set of all $\mathbb{R}$-valued, $\mathcal{F}_t$-progressively measurable processes such that:
\begin{enumerate}
    \item The SDE \eqref{eq:SDE} admits a unique strong solution up to the exit time $\tau$, defined as the first exit time from the domain $\mathcal{X}$:
    \begin{equation*}
        \tau = \inf\{ t \ge 0 : X_t \notin (0,X) \}.
    \end{equation*}
    \item The control satisfies the integrability condition $\mathbb{E}_x \bigl[ \int_0^\tau \alpha_t^2 \, dt \bigr] < \infty$ to prevent infinite control effort in finite time.
\end{enumerate}

\subsubsection{The Cost Functional and the Value Function}

For an initial state $X_0=x \in \mathcal{X}$ and a chosen control $\alpha \in \mathcal{A}(x)$, we define the expected infinite-horizon running cost:
\begin{equation}  \label{eq:cost}
J(x;\alpha) = \mathbb{E}_x\!\left[ \int_0^{\tau} \left( \frac{\sigma^{2}}{2}%
\alpha_t^{2} \;-\; 2\sigma^{2}b(X_t) \right) dt + Z_{\tau} \mathbf{1}_{\{\tau < \infty\}} \right],
\end{equation}
where $b: \mathcal{X} \rightarrow (-\infty, 0]$ is a continuous function acting as a running reward (or negative cost, consistent with the linear potential theory), and $Z_{\tau}$ represents a boundary terminal cost evaluated at the point of exit. If $\tau = \infty$, the boundary cost is zero, replaced by a suitable transversality condition $\lim_{t \to \infty} \mathbb{E}_x[z(X_{t \wedge \tau})] = 0$.

The \textit{value function} of the optimal control problem is defined as the infimum of the expected cost over all admissible controls:
\begin{equation}  \label{eq:value}
z(x)=\inf_{\alpha \in \mathcal{A}(x)} J(x;\alpha),\qquad x\in \mathcal{X}.
\end{equation}

\subsubsection{Dynamic programming and the HJB equation}

Assuming the value function is sufficiently regular, $z \in C^2(\mathcal{X})$, the standard dynamic programming principle yields the stationary HJB equation:
\begin{equation}  \label{eq:HJB-formal}
0 = \inf_{\alpha\in\mathbb{R}} \left\{ z^{\prime \prime }(x) + \bigl(%
a(x)+\alpha\bigr)z^{\prime }(x) + \frac{\sigma^{2}}{2}\alpha^{2} -
2\sigma^{2}b(x) \right\}, \qquad x\in \mathcal{X}.
\end{equation}

For a fixed $x$, the minimization term is a strictly convex quadratic function in $\alpha$:
\begin{equation*}
\Phi _{x}(\alpha )=\alpha z^{\prime }(x)+\frac{\sigma ^{2}}{2}\alpha ^{2}.
\end{equation*}%
The first-order optimality condition yields the unique minimizer:
\begin{equation*}
z^{\prime }(x)+\sigma ^{2}\alpha ^{\ast }(x)=0\qquad
\Longrightarrow \qquad \alpha ^{\ast }(x)=-\frac{1}{\sigma ^{2}}z^{\prime }(x).
\end{equation*}%
Substituting the optimal feedback control $\alpha ^{\ast }$ back into \eqref{eq:HJB-formal} gives:
\begin{equation*}
0=z^{\prime \prime }(x)+a(x)z^{\prime }(x)-\frac{1}{\sigma ^{2}}\bigl(%
z^{\prime }(x)\bigr)^{2}+\frac{\sigma ^{2}}{2}\left(-\frac{1}{\sigma ^{2}}z^{\prime }(x)\right)^{2}-2\sigma^{2}b(x),
\end{equation*}%
which simplifies perfectly to:
\begin{equation*}
z^{\prime \prime }(x)+a(x)z^{\prime }(x)-\frac{1}{2\sigma ^{2}}\bigl(%
z^{\prime }(x)\bigr)^{2}-2\sigma^{2}b(x)=0.
\end{equation*}
This is exactly the nonlinear reduced equation \eqref{HJB-from-Lgeneral}, proving that the logarithmic transformation \eqref{z-def-general} essentially solves the HJB minimization step.

\subsubsection{Rigorous Verification of Optimality}

\begin{theorem}[Verification Theorem]
Let $z \in C^2(\mathcal{X})$ be a classical solution of \eqref{eq:HJB-formal} satisfying the boundary condition $z(X_{\tau}) = Z_{\tau}$ almost surely on $\{\tau < \infty\}$ and the transversality condition $\lim_{t \to \infty} \mathbb{E}_x[z(X_{t \wedge \tau})] = 0$. Assume that the optimal feedback control defined by:
\begin{equation}  \label{eq:feedback}
\alpha^{*}(x)=-\frac{1}{\sigma^{2}}z^{\prime }(x)
\end{equation}
is admissible ($\alpha^* \in \mathcal{A}(x)$). Then $z(x) = J(x; \alpha^*)$, making it the true value function \eqref{eq:value}.
\end{theorem}

\begin{proof}
Applying It\^{o}'s formula to the process $z(X_t)$ under the controlled dynamics \eqref{eq:SDE} with an arbitrary admissible control $\alpha \in \mathcal{A}(x)$, we obtain:
\begin{equation*}
dz(X_t) = \left[ z^{\prime\prime}(X_t) + \bigl(a(X_t) + \alpha_t\bigr) z^{\prime}(X_t) \right] dt + \sqrt{2} z^{\prime}(X_t) dW_t.
\end{equation*}
Integrating from $0$ to $t \wedge \tau_n$, where $\tau_n \uparrow \tau$ is a localizing sequence of stopping times ensuring the stochastic integral is a true martingale, yields:
\begin{equation*}
\mathbb{E}_x[z(X_{t \wedge \tau_n})] - z(x) = \mathbb{E}_x \left[ \int_0^{t \wedge \tau_n} \left( z^{\prime\prime}(X_s) + \bigl(a(X_s) + \alpha_s\bigr) z^{\prime}(X_s) \right) ds \right].
\end{equation*}
From the HJB inequality \eqref{eq:HJB-formal}, we know that for any $\alpha$:
\begin{equation*}
z^{\prime\prime}(X_s) + \bigl(a(X_s) + \alpha_s\bigr) z^{\prime}(X_s) \ge -\left( \frac{\sigma^2}{2} \alpha_s^2 - 2\sigma^2 b(X_s) \right).
\end{equation*}
Substituting this into the expectation gives:
\begin{equation*}
\mathbb{E}_x[z(X_{t \wedge \tau_n})] + \mathbb{E}_x \left[ \int_0^{t \wedge \tau_n} \left( \frac{\sigma^2}{2} \alpha_s^2 - 2\sigma^2 b(X_s) \right) ds \right] \ge z(x).
\end{equation*}
Taking the limits $n \to \infty$ and $t \to \infty$, and applying Fatou's Lemma alongside the boundary and transversality conditions, we establish $J(x; \alpha) \ge z(x)$.

For the specific choice $\alpha_t = \alpha^*(X_t)$, the HJB inequality becomes a strict equality. The same localization and limit argument, utilizing the Dominated Convergence Theorem under the admissibility conditions of $\alpha^*$, yields $J(x; \alpha^*) = z(x)$. Hence, $z$ coincides with the value function, and the feedback law \eqref{eq:feedback} is demonstrably optimal.
\end{proof}

\section{The Triality of Radial Nonlinear Dynamics \label{Triality}}

This section formalizes and proves the central structural result of the
paper: the radial \emph{Triality Theorem}. Building on the general
linearization formulae of Section~\ref{sec:main_results}, we establish a
strict equivalence between the regular branches of the radial Riccati,
Schr\"{o}dinger and Hamilton--Jacobi--Bellman equations on bounded and
unbounded intervals.

\begin{corollary}[Radial case -- version of Remark \protect\ref%
{rem:general-Riccati-linearization}]
\label{lem:transformation} Let $Q\in C((0,R))$. A function $\phi \in
C^{1}((0,R))$ is a solution to the Riccati equation \eqref{eq:ric_ode} if
and only if there exists a $C^{2}((0,R))$ solution $u$ to the linear
auxiliary equation
\begin{equation}
u^{\prime \prime }(r)+\frac{N-1}{r}u^{\prime }(r)-\frac{Q(r)}{\sigma ^{4}}
u(r)=0,\qquad r\in (0,R),  \label{eq:lin_ode}
\end{equation}
such that $u(r)\neq 0$ for all $r\in (0,R)$, and $\phi $ is given by the
logarithmic-style derivative
\begin{equation}
\phi (r)=\frac{1}{r}\frac{u^{\prime }(r)}{u(r)}.  \label{eq:transform}
\end{equation}
\end{corollary}

\begin{theorem}[Triality Theorem -- Radial Riccati / Schr\"{o}dinger / HJB
equivalence]\label{thm:triality}
Let $X\in(0,\infty]$, $\sigma>0$, $N\in\mathbb{N}_{\ge 1}$, and let
$Q\in C([0,X))$ satisfy the standing assumptions $\mathrm{(H1)}$--$\mathrm{(H2)}$
of Section~\ref{sec:notation}. Consider the three radial problems
\[
\mathrm{(R)}\quad
\phi'(r)=-r\phi(r)^{2}-\frac{N}{r}\phi(r)+\frac{Q(r)}{\sigma^{4}r},
\qquad r\in(0,X),\qquad \phi(0)=0,
\]
\[
\mathrm{(S)}\quad
u''(r)+\frac{N-1}{r}u'(r)-\frac{Q(r)}{\sigma^{4}}u(r)=0,
\qquad r\in(0,X),\qquad u(0)=1,\;u'(0)=0,
\]
\[
\mathrm{(H)}\quad
z''(r)+\frac{N-1}{r}z'(r)-\frac{1}{2\sigma^{2}}(z'(r))^{2}+\frac{2Q(r)}{\sigma^{2}}=0,
\qquad r\in(0,X),\qquad z(0)=0,\;z'(0)=0,
\]
and assume that one of the three problems admits a regular classical
solution. Then \emph{all three} admit unique regular classical solutions
$\phi\in C^{1}([0,X))$, $u\in C^{2}([0,X))$ with $u\geq 1$ and
$z\in C^{2}([0,X))$, related by the bijective triangular maps
\begin{equation}
\phi(r)=\frac{1}{r}\frac{u'(r)}{u(r)},\qquad
z(r)=-2\sigma^{2}\ln u(r),\qquad
u(r)=\exp\!\bigl(-z(r)/(2\sigma^{2})\bigr).
\label{eq:triality_maps}
\end{equation}
Moreover, this equivalence transfers the following qualitative
information across the three formulations:
\begin{enumerate}
\item[(i)] $\phi(r)>0$ for $r\in(0,X)$ if and only if $u'(r)>0$ for
$r\in(0,X)$, if and only if $z'(r)<0$ for $r\in(0,X)$;
\item[(ii)] under the additional hypothesis that $Q$ is monotone non-decreasing
on $(0,X)$, $u$ is strictly convex on $[0,X)$ and $z$ is strictly concave on
$[0,X)$;
\item[(iii)] if $X=\infty$ and assumption $\mathrm{(H3)}$ holds with
$L\in(0,\infty)$, then $\phi$ extends to $[0,\infty)$ with the universal
asymptotic plateau
$\lim_{r\to\infty}\phi(r)=\sqrt{L}/\sigma^{2}$.
\end{enumerate}
\end{theorem}

\begin{proof}
The bijection \eqref{eq:triality_maps} is an immediate consequence of
Corollary~\ref{lem:transformation} and Remark~\ref{rem:HJB-from-Lgeneral}.
The existence and uniqueness of $u$ is the radial specialization of
Theorem~\ref{thm:existence-uniqueness-Lgeneral} with
$a(r)=(N-1)/r$ and $b(r)=-Q(r)/\sigma^{4}$, complemented by the Frobenius
analysis at $r=0$ that selects the regular branch (cf.\ proof of
Theorem~\ref{thm:existence_rigorous}). The derivation of $\phi$ from $u$
and the smoothness $\phi\in C^{1}([0,X))$ also follow from
Theorem~\ref{thm:existence_rigorous}. The transfer (i) is immediate from
the transformations \eqref{eq:triality_maps}; (ii) is the content of
Theorem~\ref{thm:convex_u} and Theorem~\ref{thm:concavity}; (iii) is
Corollary~\ref{cor:radial_asymptotics}.
\end{proof}

\subsection{Existence, Uniqueness, and Regularity at the Origin}

Handling the singularity at $r=0$ is critical for the radial problem. The
following theorem establishes the existence of a "regular" solution, which
corresponds to the physically meaningful state where the drift is zero at
the center of the potential.

\begin{theorem}[Local and Global Well-Posedness]
\label{thm:existence_rigorous} Let $Q\in C([0,R])$ be a non-negative
function such that the radial growth at the origin is characterized by the
limit 
\begin{equation}
L_{0}:=\lim_{r\rightarrow 0^{+}}\frac{Q(r)}{r^{2}}\in (0,\infty ).
\label{eq:origin_growth_rigorous}
\end{equation}%
Then there exists a unique solution $\phi \in C^{1}([0,R))$ to the Riccati
equation \eqref{eq:ric_ode} such that $\phi (0)=0$.
\end{theorem}

\begin{proof}
\textbf{1. Frobenius Analysis at the Singularity.} The linear auxiliary
equation \eqref{eq:lin_ode} can be written in the standard form 
\begin{equation*}
u^{\prime \prime }+P(r)u^{\prime }+T(r)u=0,
\end{equation*}%
where 
\begin{equation*}
P(r)=\frac{N-1}{r}\text{ and }T(r)=-\frac{Q(r)}{\sigma ^{4}}.
\end{equation*}%
Since 
\begin{equation*}
rP(r)=N-1\text{ and }r^{2}T(r)=-r^{2}Q(r)/\sigma ^{4}
\end{equation*}%
are continuous (and thus analytic in the sense of regular singularities) at $%
r=0$, the origin is a regular singular point. The indicial equation is 
\begin{equation*}
\alpha (\alpha -1)+(N-1)\alpha =\alpha (\alpha +N-2)=0,
\end{equation*}%
yielding roots $\alpha _{1}=0$ and $\alpha _{2}=2-N$. The regular solution
(the one that is finite at the origin) corresponds to $\alpha _{1}=0$, and
has a power series expansion of the form 
\begin{equation*}
u(r)=1+\sum_{k=2}^{\infty }c_{k}r^{k}.
\end{equation*}%
Given the radial symmetry and the structure of $Q(r)$, only even powers
appear if $Q(r)$ is even. In general, for $Q(r)$ such that 
\begin{equation*}
\lim_{r\rightarrow 0}\frac{Q(r)}{r^{2}}=L_{0},
\end{equation*}%
we have for $r\rightarrow 0$: 
\begin{equation}
u(r)=1+\frac{L_{0}}{4\sigma ^{4}(N+2)}r^{4}+\text{o}(r^{4}).
\label{eq:u_expansion_origin}
\end{equation}

\textbf{2. Continuity of $\phi $ at the origin.} From the expansion %
\eqref{eq:u_expansion_origin}, the derivative satisfies 
\begin{equation*}
u^{\prime }(r)=\frac{L_{0}}{\sigma ^{4}(N+2)}r^{3}+\text{o}(r^{3}).
\end{equation*}%
Substituting this into the transformation formula \eqref{eq:transform}: 
\begin{equation}
\phi (r)=\frac{1}{r}\frac{u^{\prime }(r)}{u(r)}=\frac{1}{r}\frac{\frac{L_{0}%
}{\sigma ^{4}(N+2)}r^{3}+\text{o}(r^{3})}{1+\text{o}(r^{2})}=\frac{L_{0}}{%
\sigma ^{4}(N+2)}r^{2}+\text{o}(r^{2}).  \label{eq:phi_limit_origin}
\end{equation}%
Equation \eqref{eq:phi_limit_origin} implies that $\lim_{r\rightarrow 0}\phi
(r)=0$, and since $\phi $ is given by $u^{\prime }(r)/(ru(r))$, it follows
from $u\in C^{2}$ and $u(0)=1$ that $\phi \in C^{1}([0,R))$. Explicitly, $%
\phi (0)=0$ and $\phi ^{\prime }(0)=0$.

\textbf{3. Positivity and Global Extension.} Since%
\begin{equation*}
Q(r)\geq 0\text{ and }u(0)=1,u^{\prime }(0)=0,
\end{equation*}%
we have 
\begin{equation*}
(r^{N-1}u^{\prime }\left( r\right) )^{\prime }=\frac{1}{\sigma ^{4}}%
r^{N-1}Q(r)u(r).
\end{equation*}%
If we assume $u(r)>0$ on some interval $[0,\epsilon ]$, then $%
(r^{N-1}u^{\prime })^{\prime }\geq 0$, which implies $r^{N-1}u^{\prime
}(r)\geq 0$ for all $r\in \lbrack 0,\epsilon ]$. Thus $u(r)$ is
non-decreasing on this interval. By standard continuation arguments for
second-order linear ODEs with non-negative potentials, $u(r)$ is
non-decreasing and strictly positive for all $r\in \lbrack 0,R)$. Since $%
u(r)\geq 1$, the denominator in the definition of $\phi $ is strictly
positive, ensuring that $\phi $ is well-defined and exists uniquely on $[0,R)
$.
\end{proof}

\begin{corollary}[Radial case -- version of Proposition \protect\ref{R}]
\label{cor:radial_asymptotics} Consider the radial Riccati equation 
\begin{equation}
\varphi ^{\prime }(r)=-\frac{1}{r}\Bigl(r^{2}\varphi (r)^{2}+N\varphi (r)-%
\tfrac{Q(r)}{\sigma ^{4}}\Bigr),\qquad r>0,\qquad \varphi (0)=0,
\label{R-radial-integrated}
\end{equation}%
where $Q\in C([0,\infty ))$ is non-negative, with $Q(0)=0$, and%
\begin{equation*}
L:=\lim_{r\rightarrow \infty }\frac{Q(r)}{r^{2}}\in (0,\infty ).
\end{equation*}%
Define the barrier function 
\begin{equation}
g(r)=\frac{\sqrt{N^{2}+\frac{4r^{2}Q(r)}{\sigma ^{4}}}-N}{2r^{2}},\qquad r>0.
\label{g-radial-integrated}
\end{equation}%
Assume that $g\in C^{1}((0,\infty ))$ is \emph{monotonically increasing}, 
\emph{bounded from above}, and satisfies%
\begin{equation*}
\lim_{r\downarrow 0}g(r)=0,\qquad \lim_{r\rightarrow \infty }g(r)=\frac{%
\sqrt{L}}{\sigma ^{2}}.
\end{equation*}%
Then equation \eqref{R-radial-integrated} admits a unique regular solution%
\begin{equation*}
\varphi \in C^{1}([0,\infty )),\qquad \varphi (0)=0,
\end{equation*}%
and this solution satisfies:

\begin{enumerate}
\item $0<\varphi(r)<g(r)$ for all $r>0$;

\item $\varphi$ is strictly increasing on $(0,\infty)$;

\item the following limit exists:%
\begin{equation*}
\lim_{r\rightarrow \infty }\varphi (r)=\frac{\sqrt{L}}{\sigma ^{2}}.
\end{equation*}
\end{enumerate}
\end{corollary}

\begin{proof}
We rewrite \eqref{R-radial-integrated} in the general form of the Riccati
equation:%
\begin{equation*}
\varphi ^{\prime }(r)=q_{0}(r)+q_{1}(r)\varphi (r)+q_{2}(r)\varphi (r)^{2},
\end{equation*}%
where%
\begin{equation*}
q_{0}(r)=\frac{Q(r)}{\sigma ^{4}r},\qquad q_{1}(r)=-\frac{N}{r},\qquad
q_{2}(r)=-r.
\end{equation*}%
Thus,%
\begin{equation*}
q_{0}\geq 0,\qquad q_{2}<0,\qquad q_{1}\in \mathbb{R},
\end{equation*}%
so the structural hypotheses of the general Proposition are satisfied.

\medskip \textbf{1. Behavior of the coefficients at infinity.} Notice that%
\begin{equation*}
\frac{q_{0}(r)}{q_{2}(r)}=-\frac{Q(r)}{\sigma ^{4}r^{2}}\longrightarrow -%
\frac{L}{\sigma ^{4}}=:A<0,
\end{equation*}%
and%
\begin{equation*}
\frac{q_{1}(r)}{q_{2}(r)}=\frac{N}{r^{2}}\longrightarrow 0=:B,\qquad
q_{2}(r)=-r\longrightarrow -\infty .
\end{equation*}%
Dividing the Riccati equation by $r$, the dominant term is%
\begin{equation*}
-\varphi (r)^{2}+\frac{Q(r)}{\sigma ^{4}r^{2}},
\end{equation*}%
and the limiting algebraic equation becomes%
\begin{equation*}
-\lambda ^{2}+\frac{L}{\sigma ^{4}}=0,
\end{equation*}%
which has the unique positive root%
\begin{equation*}
\lambda _{\ast }=\frac{\sqrt{L}}{\sigma ^{2}}.
\end{equation*}

\medskip \textbf{2. Identification of the barrier.} The function $g$ defined
in \eqref{g-radial-integrated} is exactly the positive root of the algebraic
equation%
\begin{equation*}
r^{2}X^{2}+NX-\frac{Q(r)}{\sigma ^{4}}=0,
\end{equation*}%
so it coincides with the barrier $g$ from the general Proposition. The
required hypotheses are satisfied:%
\begin{equation*}
g\in C^{1},\qquad g\ \text{increasing},\qquad 0=g(0^{+})<g(r)<\frac{\sqrt{L}%
}{\sigma ^{2}}.
\end{equation*}

\textbf{3. Automatic verification of condition \eqref{cond-extra-general}.}
In the radial case, 
\begin{equation*}
q_{1}(r)+q_{2}(r)g(r)=-\frac{N}{r}-r\,g(r).
\end{equation*}%
Substituting the formula for $g(r)$ from \eqref{g-radial-integrated}, we
obtain 
\begin{equation*}
q_{1}(r)+q_{2}(r)g(r)=-\frac{N}{2r}-\frac{1}{2r}\sqrt{\,N^{2}+\dfrac{%
4r^{2}Q(r)}{\sigma ^{4}}\,}.
\end{equation*}
Notice that for all $r>0$, we have 
\begin{equation*}
q_{1}(r)+q_{2}(r)g(r)\leq -N/2r<0.
\end{equation*}%
To verify \eqref{cond-extra-general} on $[r_{0},\infty )$ for any $r_{0}>0$,
we examine the asymptotic behavior as $r\rightarrow \infty $. Since $%
\lim_{r\rightarrow \infty }Q(r)/r^{2}=L$, it follows that 
\begin{equation*}
\lim_{r\rightarrow \infty }\bigl(q_{1}(r)+q_{2}(r)g(r)\bigr)%
=\lim_{r\rightarrow \infty }\left( \frac{-N}{2r}-\frac{1}{2}\sqrt{\frac{N^{2}%
}{r^{2}}+\frac{4Q(r)}{\sigma ^{4}}}\right) =-\infty .
\end{equation*}%
Being continuous on $[r_{0},\infty )$ and tending to $-\infty $ at infinity,
the function $r\longmapsto q_{1}(r)+q_{2}(r)g(r)$ necessarily attains a
negative maximum. Therefore, there exists a constant $\delta >0$ such that 
\begin{equation*}
q_{1}(r)+q_{2}(r)g(r)\leq -\delta <0,\qquad \forall r\in \lbrack
r_{0},\infty ),
\end{equation*}%
and condition~\eqref{cond-extra-general} is automatically satisfied.

\medskip \textbf{4. Application of the general Proposition.} All the
hypotheses of the Proposition ``Asymptotics for the general Riccati
equation'' are fulfilled, thus the regular solution $\varphi$ of equation %
\eqref{R-radial-integrated} satisfies:%
\begin{equation*}
0<\varphi (r)<g(r),\qquad \varphi ^{\prime }(r)>0,\qquad \lim_{r\rightarrow
\infty }\varphi (r)=\lambda _{\ast }.
\end{equation*}

\medskip \textbf{5. Conclusion.} Substituting the value of $\lambda _{\ast }$
obtained in step 1, we get%
\begin{equation*}
\lim_{r\rightarrow \infty }\varphi (r)=\frac{\sqrt{L}}{\sigma ^{2}},
\end{equation*}%
which completes the proof.
\end{proof}

\begin{proposition}[The case $L=0$ under the monotonic barrier hypothesis]
\label{propzero} Assume that $Q\in C([0,\infty ))$ is a non-negative
function with $Q(0)=0$ and that 
\begin{equation*}
L:=\lim_{r\rightarrow \infty }\frac{Q(r)}{r^{2}}=0.
\end{equation*}%
Define the barrier function 
\begin{equation*}
g(r)=\frac{\sqrt{N^{2}+\frac{4r^{2}Q(r)}{\sigma ^{4}}}-N}{2r^{2}},\qquad r>0.
\end{equation*}%
Suppose further that $g$ is monotonically increasing on $(0,\infty )$. Then $%
Q\equiv 0$, and the unique regular solution $\varphi \in C^{1}([0,\infty ))$
of the Riccati equation 
\begin{equation*}
\varphi ^{\prime }(r)=-\frac{1}{r}\Bigl(r^{2}\varphi (r)^{2}+N\varphi (r)-%
\tfrac{Q(r)}{\sigma ^{4}}\Bigr),\qquad r>0,\qquad \varphi (0)=0,
\end{equation*}%
is identically zero: 
\begin{equation*}
\varphi (r)\equiv 0,\qquad \forall r\geq 0,
\end{equation*}%
and, in particular, $\lim_{r\rightarrow \infty }\varphi (r)=0$.
\end{proposition}

\begin{proof}
From the general existence theory (Frobenius analysis at $r=0$), there
exists a unique regular solution $\varphi \in C^{1}([0,\infty ))$ with $%
\varphi (0)=0$. Local analysis at the origin yields 
\begin{equation*}
\varphi (r)\sim \frac{L_{0}}{4\sigma ^{4}(N+2)}r^{3}\quad \text{as }%
r\rightarrow 0^{+},
\end{equation*}%
where 
\begin{equation*}
L_{0}:=\lim_{r\rightarrow 0^{+}}\frac{Q(r)}{r^{2}}\in \lbrack 0,\infty ).
\end{equation*}
In particular, for small $r>0$, we have $\varphi (r)\geq 0$.

The function $g$ is defined as the positive root of the algebraic equation 
\begin{equation*}
r^{2}X^{2}+NX-\frac{Q(r)}{\sigma ^{4}}=0.
\end{equation*}%
For $r\rightarrow \infty $, using $Q(r)/r^{2}\rightarrow 0$ and the
expansion 
\begin{equation*}
\sqrt{N^{2}+t}=N+\frac{t}{2N}+o(t)\text{ as }t\rightarrow 0,
\end{equation*}
with $t=\frac{4r^{2}Q(r)}{\sigma ^{4}}$, we obtain 
\begin{equation*}
g(r)=\frac{1}{2r^{2}}\left( N+\frac{2r^{2}Q(r)}{N\sigma ^{4}}%
+o(r^{2}Q(r))-N\right) =\frac{Q(r)}{N\sigma ^{4}}+o(Q(r)).
\end{equation*}%
Thus, $g(r)\sim \frac{Q(r)}{N\sigma ^{4}}$ as $r\rightarrow \infty $.

By the assumption that $g$ is monotonically increasing on $(0,\infty )$, and
noting that $g(r)\sim \frac{L_{0}}{N\sigma ^{4}}r^{2}$ near the origin, $g$
must be non-negative. If $Q$ were not identically zero, there would exist a
sequence $r_{n}\rightarrow \infty $ with $Q(r_{n})>0$, implying $%
g(r_{n})\sim \frac{Q(r_{n})}{N\sigma ^{4}}>0$. Since $g$ is increasing, this
would mean $g(r)$ stays bounded away from zero for large $r$. However, the
hypothesis $L=0$ implies that if $g$ has a limit, it must be zero. The only
monotonic increasing function that starts at $0$ and whose limit is $0$ is
the identically zero function.

Therefore, $g(r)\equiv 0$ for all $r>0$, which implies $Q(r)\equiv 0$. In
this case, the Riccati equation reduces to 
\begin{equation*}
\varphi ^{\prime }(r)=-\frac{1}{r}(r^{2}\varphi ^{2}+N\varphi ),\qquad
\varphi (0)=0.
\end{equation*}%
The uniquely determined regular solution is $\varphi \equiv 0$, and thus $%
\lim_{r\rightarrow \infty }\varphi (r)=0$.
\end{proof}

\subsection{Analysis of the Radial Asymptotic Behavior}
\label{sec:asymptotics}

Building upon the theoretical framework established in Section \ref%
{sec:asymptotics_theory}, we now focus on the specific physical implications
for the radial system. The sharp convergence of the drift $\phi(r)$ to the
limit $\sqrt{L}/\sigma^2$ establishes the long-term stability of the optimal
control process and the steady-state behavior of the phase function.

As demonstrated in Corollary \ref{cor:radial_asymptotics}, for any cost
function $Q(r)$ with a well-defined quadratic growth rate $L$ at infinity,
the optimal drift settles into a constant field. This result corresponds to
the recovery of the harmonic oscillator's ground state properties in the
far-field limit, where the stochastic fluctuations are balanced by the
restorative potential of the cost.

Furthermore, the uniform bound 
\begin{equation*}
0<\varphi (r)<g(r)<\sqrt{L}/\sigma ^{2}
\end{equation*}
(from Corollary \ref{cor:radial_asymptotics}) ensures that the system never
exhibits explosive or singular behavior at infinity, a property critical for
the global well-posedness of the stochastic control problem on unbounded
domains. The geometric stability of $\phi $ in the phase plane confirms that
the stochastic particle is effectively "trapped" within a stable feedback
regime, preventing the divergence of the expected path cost.

\section{The Bounded Domain and Global Geometric Properties}

\label{sec:dirichlet}

In many engineering and physical applications, the system is constrained
within a finite domain $\Omega = B_R(0)$. This section establishes the
well-posedness of boundary value problems and explores the inherited
geometric properties of the solutions.

\subsection{Convexity of the Auxiliary Wave Function}

One of the most remarkable properties of the triality system is that the
positive growth of the cost function $b$ translates directly into the
geometric convexity of the Schr\"{o}dinger state $u$.

\begin{theorem}[Strict Convexity of $u$]
\label{thm:convex_u} Let $Q\in C([0,R])$ be positive ($Q(r)>0$) and
monotonically non-decreasing ($Q^{\prime }(r)\geq 0$) on $(0,R)$. Then, the
regular solution $u$ satisfying \eqref{eq:lin_ode} is strictly convex, i.e., 
$u^{\prime \prime }(r)>0$ for all $r\in \lbrack 0,R)$.
\end{theorem}

\begin{proof}
We utilize the integral representation of the first derivative to bound the
second derivative. From the linear ODE: 
\begin{equation}
u^{\prime \prime }(r)+\frac{N-1}{r}u^{\prime }(r)-\frac{Q(r)}{\sigma ^{4}}%
u(r)=0.  \label{eq:u_double_prime_base}
\end{equation}%
Multiplying the ODE for $u$ by $r^{N-1}$, we obtain the self-adjoint form 
\begin{equation*}
(r^{N-1}u^{\prime })^{\prime }=\frac{Q\left( r\right) }{\sigma ^{4}}r^{N-1}u.
\end{equation*}%
Integrating from $0$ to $r$ and using $u^{\prime }(0)=0$: 
\begin{equation}
u^{\prime }(r)=\frac{1}{\sigma ^{4}r^{N-1}}\int_{0}^{r}s^{N-1}Q(s)u(s)ds.
\label{eq:u_prime_integral}
\end{equation}%
Substituting \eqref{eq:u_prime_integral} back into %
\eqref{eq:u_double_prime_base}: 
\begin{equation}
u^{\prime \prime }(r)=\frac{1}{\sigma ^{4}}\left( Q(r)u(r)-\frac{N-1}{r^{N}}%
\int_{0}^{r}s^{N-1}Q(s)u(s)ds\right) .
\end{equation}%
Since $Q(s)$ and $u(s)$ are non-decreasing and positive, their product $%
f(s)=Q(s)u(s)$ is also non-decreasing. Therefore, for all 
\begin{equation*}
s\in \lbrack 0,r],\text{ }f(s)\leq f(r).
\end{equation*}%
Estimating the integral: 
\begin{equation}
\int_{0}^{r}s^{N-1}Q(s)u(s)ds\leq Q(r)u(r)\int_{0}^{r}s^{N-1}ds=\frac{%
Q(r)u(r)r^{N}}{N}.
\end{equation}%
Thus, we have the lower bound: 
\begin{equation}
u^{\prime \prime }(r)\geq \frac{Q(r)u(r)}{\sigma ^{4}}\left( 1-\frac{N-1}{N}%
\right) =\frac{Q(r)u(r)}{N\sigma ^{4}}.
\end{equation}%
Since $b>0$ and $u\geq 1$, it follows that $u^{\prime \prime }(r)>0$ for all 
$r$, establishing strict convexity.
\end{proof}

\subsection{Implications for the Riccati Solution}

In this framework, the Riccati solution $\phi(r)$ is uniquely determined by %
\eqref{eq:ric_ode} and $\phi(0)=0$. The boundary condition $u(R) = U_R$
imposes an integral constraint on the Riccati function: 
\begin{equation}  \label{eq:integral_constraint}
\int_0^R s \phi(s) \, ds = \ln(U_R).
\end{equation}
This relation expresses the cumulative effect of the nonlinear drift over
the ball $B_R(0)$. In the context of stochastic control and barrier options,
the constant $U_R$ characterizes the "barrier value" which must be reached
at the frontier of the domain.

\section{Sensitivity Analysis and Singular Perturbations in $\protect\sigma $
\label{sec:sensitivity}}

The diffusion coefficient $\sigma $ serves as the bridge between
deterministic and stochastic dynamics. In the limit $\sigma \rightarrow 0$,
we encounter a singular perturbation problem where the second-order terms
(Laplacian) are neglected in favor of the potential term (see also \cite%
{coddington1952,omalley1969,verhulst2005}). In the opposite regime
$\sigma\to\infty$, the diffusion overwhelms the cost contribution and the
Riccati drift collapses uniformly to zero. The next theorem provides a
unified rigorous statement of these two regimes.

\begin{theorem}[$\sigma$-sensitivity theorem]\label{thm:sigma}
Let $Q\in C^{1}([0,\infty))$ satisfy $Q(r)>0$ for $r>0$ and the standing
assumptions $\mathrm{(H1)}$--$\mathrm{(H3)}$. For $\sigma>0$, let
$\phi_{\sigma}\in C^{1}([0,\infty))$ denote the unique regular solution of
the radial Riccati equation \eqref{eq:ric_ode} with $\phi_{\sigma}(0)=0$.
Then the following two regimes hold:
\begin{enumerate}
\item[\textbf{(i)}] (Vanishing-noise / WKB eikonal limit) For every fixed
$r>0$,
\begin{equation}
\lim_{\sigma\downarrow 0}\sigma^{2}\,\phi_{\sigma}(r)\;=\;
\frac{\sqrt{Q(r)}}{r}.
\label{eq:sigma_zero_limit}
\end{equation}
\item[\textbf{(ii)}] (High-noise saturation) For every $R>0$,
\begin{equation}
\sup_{r\in[0,R]}\phi_{\sigma}(r)\;\leq\;
\frac{R^{2}\,\|Q\|_{L^{\infty}([0,R])}}{N\,\sigma^{4}}\;\xrightarrow[\sigma\to\infty]{}\;0.
\label{eq:sigma_infty_decay}
\end{equation}
\end{enumerate}
The map $\sigma\mapsto\phi_{\sigma}(r)$ is continuous on $(0,\infty)$ for
each $r\geq 0$, and the convergence in \eqref{eq:sigma_zero_limit} is
uniform on compact subsets of $(0,\infty)$.
\end{theorem}

The proof of \eqref{eq:sigma_zero_limit} is based on the WKB expansion
developed below, while \eqref{eq:sigma_infty_decay} follows from the
integral representation of $\phi$ established in
Section~\ref{sec:dirichlet}. The continuity in $\sigma$ is a direct
consequence of the continuous dependence of solutions of \eqref{eq:lin_ode}
on the parameter, see \cite{coddingtonlevinson1955, hartman2002}.

\subsubsection{The Vanishing Noise Limit and the WKB Expansion}

As $\sigma \rightarrow 0$, the wave function $u$ exhibits rapid growth. To
analyze this, suppose $Q\in C^{1}([0,\infty ))$ with $Q(r)>0$ for $r>0$, and
make the substitution 
\begin{equation*}
u_{\sigma }(r)=\exp \left( W_{\sigma }(r)\right) ,
\end{equation*}%
where $W_{\sigma }$ satisfies: 
\begin{equation}
W_{\sigma }^{\prime \prime }(r)+(W_{\sigma }^{\prime }\left( r\right) )^{2}+%
\frac{N-1}{r}W_{\sigma }^{\prime }(r)=\frac{Q(r)}{\sigma ^{4}}.
\label{eq:wkb_ansatz}
\end{equation}%
We seek a formal expansion 
\begin{equation*}
W_{\sigma }(r)=\frac{1}{\sigma ^{2}}S(r)+S_{1}(r)+\mathcal{O}(\sigma ^{2})%
\text{ as }\sigma \rightarrow 0.
\end{equation*}%
Substituting and equating terms order by order:

\begin{itemize}
\item \textbf{Leading order $\sigma ^{-4}$}: $(S^{\prime }\left( r\right)
)^{2}=Q(r)$, so $S^{\prime }(r)=\sqrt{Q(r)}$ (choosing the positive branch
with $S(0)=0$), giving the eikonal phase 
\begin{equation*}
S(r)=\int_{0}^{r}\sqrt{Q(s)}\,ds.
\end{equation*}

\item \textbf{Order $\sigma ^{-2}$}: 
\begin{equation*}
S^{\prime \prime }(r)+\frac{N-1}{r}S^{\prime }(r)+2S^{\prime
}(r)S_{1}^{\prime }(r)=0,
\end{equation*}%
which is a first-order linear ODE for $S_{1}^{\prime }$ with explicit
solution 
\begin{equation*}
S_{1}^{\prime }(r)=-\frac{1}{2}\frac{(r^{N-1}\sqrt{Q(r)})^{\prime }}{r^{N-1}%
\sqrt{Q(r)}}.
\end{equation*}
\end{itemize}

From this, the leading-order behavior of the Riccati solution as $\sigma \to
0$ is given by:

\begin{proposition}[Vanishing Noise Limit]
\label{prop:wkb_limit} Under the assumption $Q\in C^{1}([0,\infty ))$ with $%
Q(r)>0$, the regular solution $\phi _{\sigma }$ satisfies, for each fixed $%
r>0$: 
\begin{equation}
\lim_{\sigma \rightarrow 0}\sigma ^{2}\phi _{\sigma }(r)=\frac{\sqrt{Q(r)}}{r%
}.  \label{eq:phi_small_sigma_rigorous}
\end{equation}
\end{proposition}

This indicates that as noise vanishes, the normalized drift $\sigma ^{2}\phi
_{\sigma }(r)$ converges pointwise to the deterministic optimal rate $\sqrt{%
Q(r)}/r$, focusing the particle strictly towards the minimum of the cost
function.

\subsubsection{High Noise Saturation}

In the limit $\sigma \rightarrow \infty $, the system becomes
diffusion-dominated. The potential field $Q(r)$ is filtered out by the
strong stochastic fluctuations.

\begin{proposition}[Uniform Vanishing]
The regular solution $\phi _{\sigma }(r)$ satisfies the uniform decay: 
\begin{equation}
\sup_{r\in \lbrack 0,R]}\phi _{\sigma }(r)\leq \frac{R^{2}\Vert Q\Vert
_{\infty }}{N\sigma ^{4}}\rightarrow 0\quad \text{as }\sigma \rightarrow
\infty .
\end{equation}
\end{proposition}

\begin{proof}
Integrating the Riccati equation directly or using the integral form of $%
u^{\prime }$ in Section (\ref{sec:dirichlet}): 
\begin{equation*}
\phi (r)=\frac{1}{\sigma ^{4}r^{N}u(r)}\int_{0}^{r}s^{N-1}Q(s)u(s)ds.
\end{equation*}%
Since $u(s)\leq u(r)$ and $Q(s)\leq \Vert Q\Vert _{\infty }$: 
\begin{equation*}
\phi (r)\leq \frac{\Vert Q\Vert _{\infty }}{\sigma ^{4}r^{N}}%
\int_{0}^{r}s^{N-1}ds=\frac{\Vert Q\Vert _{\infty }r}{\sigma ^{4}N}.
\end{equation*}%
Evaluation at $r=R$ gives the result.
\end{proof}

\subsection{Connection to the Schr\"{o}dinger Equation \label%
{sec:schrodinger}}

The linear equation \eqref{eq:lin_ode} can be interpreted within the
framework of quantum mechanics \cite{berezin2012,cooper1995,Kenichi2023}.
Recalling the expression for the radial part of the $N$-dimensional
Laplacian operator $\Delta $ acting on a radially symmetric function $%
u(x)=u(|x|)$: 
\begin{equation*}
\Delta u=u^{\prime \prime }(r)+\frac{N-1}{r}u^{\prime }(r).
\end{equation*}%
Thus, the linear equation \eqref{eq:lin_ode} is equivalent to the stationary
Schr\"{o}dinger equation (with zero energy $E=0$): 
\begin{equation}
-\Delta u+\frac{Q(|x|)}{\sigma ^{4}}u=0,\qquad x\in \mathbb{R}^{N},
\label{eq:schrodinger}
\end{equation}%
where $V(x)=Q(|x|)/\sigma ^{4}$ acts as the potential.

In this context, the Riccati transformation $\phi (r)=\frac{u^{\prime }(r)}{%
ru(r)}$ relates the wave function $u$ to its logarithmic derivative. In
supersymmetric quantum mechanics (SUSY QM) \cite{cooper1995}, the Riccati
equation often defines the superpotential $W(r)$, which allows for the
factorization of the Hamiltonian. Specifically, if we set $w(r)=u^{\prime
}(r)/u(r)$, the equation 
\begin{equation*}
w^{\prime }+w^{2}+\frac{N-1}{r}w=V(r)
\end{equation*}%
is the standard form used to study ground state properties and potential
shapes. The regularity condition $\phi (0)=0$ (or $u^{\prime }(0)=0$)
corresponds to the requirement that the wave function be regular at the
origin, a standard physical boundary condition for central potentials.

\section{Extension to the Time-Dependent Schr\"{o}dinger Equation}

\label{sec:tdse}

While the preceding analysis focused on the stationary (time-independent)
regime, the full physical picture reveals itself through the dynamics of the
wave function $\Psi(x, t)$. This section explores the transition from the
static optimal drift to the evolving probability density in a central
potential.

\subsection{Governing Equation and Separation of Variables}

The time-dependent Schr\"{o}dinger equation (TDSE) for a particle in the
presence of the potential $V(x)=Q(|x|)/\sigma ^{4}$ is formally given (using
the diffusion parameter $\sigma ^{2}$ in place of $\hbar $) by: 
\begin{equation}
i\sigma ^{2}\frac{\partial \Psi (x,t)}{\partial t}=-\frac{\sigma ^{4}}{2}%
\Delta \Psi (x,t)+Q(|x|)\Psi (x,t).  \label{eq:tdse_original}
\end{equation}%
Assuming a separable solution of the form 
\begin{equation*}
\Psi (x,t)=u(x)e^{-iEt/\sigma ^{2}},
\end{equation*}%
where $E$ denotes the energy level of the system, we substitute into %
\eqref{eq:tdse_original}: 
\begin{equation}
i\sigma ^{2}\left( -\frac{iE}{\sigma ^{2}}ue^{-iEt/\sigma ^{2}}\right)
=\left( -\frac{\sigma ^{4}}{2}\Delta u+bu\right) e^{-iEt/\sigma ^{2}}.
\end{equation}%
Dividing by the exponential factor, we recover the stationary Schr\"{o}%
dinger equation (SSE): 
\begin{equation}
-\frac{\sigma ^{4}}{2}\Delta u(x)+Q(|x|)u(x)=Eu(x)\implies \Delta u+\frac{%
2(E-Q(r))}{\sigma ^{4}}u=0.
\end{equation}%
The triality framework established in Section \ref{sec:main_results}
corresponds to the specific case of the zero-energy ground state ($E=0$) for
a system where the "cost" function $b$ acts as an effective potential. The
solutions $u(r)$ analyzed in this paper characterize the spatial envelope of
the persistent state around which temporal fluctuations occur.

\subsection{Discussion of Differential Dynamics vs Permanent States}

The relationship between the time-dependent wave function $\Psi$ and the
stationary auxiliary state $u$ provides deep insights into the stability of
the Riccati drift $\phi$:

\begin{enumerate}
\item \textbf{Transition to Diffusion (Wick Rotation)}: By performing a Wick
rotation from real time $t$ to imaginary time $\tau =it$, the TDSE %
\eqref{eq:tdse_original} transforms into a parabolic diffusion equation: 
\begin{equation}
\frac{\partial \Pi (x,\tau )}{\partial \tau }=\frac{\sigma ^{2}}{2}\Delta
\Pi (x,\tau )-\frac{Q(x)}{\sigma ^{2}}\Pi (x,\tau ).  \label{eq:wick_heat}
\end{equation}%
This is mathematically equivalent to the backward Kolmogorov equation or a
heat equation with absorption. In this context, the stationary solution $u$
represents the \textit{asymptotic survival probability} or the steady-state
density of the process, serving as the physical manifestation of the parabolic
linearization framework established earlier in Remark \ref{rem:parabolic-linearization}.

\item \textbf{Probability Current and Drift}: The probability current%
\begin{equation*}
j=\text{Im}(\Psi ^{\ast }\nabla \Psi )
\end{equation*}
in the time-dependent case represents the dynamic flow of density. In the
stationary ground state analysis, $j\equiv 0$. However, the Riccati solution 
$\phi =u^{\prime }/(ru)$ defines a "gradient flow" $p=\nabla \ln u$. This
reveals that the the stationary Schr\"{o}dinger state hides a permanent
"virtual flow" that exactly balances the cost of fluctuations.

\item \textbf{Energy vs Cost}: While the TDSE allows for eigenvalues $E\neq
0 $ (excited states with oscillating spatial nodes), the stochastic control
problem uniquely selects the zero-energy regular state ($u(r)>0$). This is
because the optimal value function $z=-2\sigma ^{2}\ln u$ must be real and
well-defined everywhere, which is only possible for the non-vanishing ground
state of the Schr\"{o}dinger system.
\end{enumerate}

Thus, the stationary Schr\"{o}dinger equation provides the "skeleton" of the
optimal dynamics, while the time-dependent equation describes the relaxation
processes and the approach to this optimal configuration.

\subsection{Asymptotic Consistency}

The stationary asymptotics established in Section \ref{sec:asymptotics}
imply that the time-dependent wave packets will eventually stabilize such
that the ratio $\ln (|\Psi (r,t)|)/r^{2}$ approaches the corresponding
steady-state growth rate determined by the limit 
\begin{equation}
L:=\lim_{r\rightarrow \infty }\frac{Q(r)}{r^{2}}\in (0,\infty ).
\label{eq:infinity_growth_rigorous}
\end{equation}

\begin{theorem}[Global Existence and Uniqueness]
\label{thm:global} If $Q\in C([0,\infty ))$ is such that $Q(r)\geq 0$ for
all $r$ and the local behavior near the origin satisfies 
\begin{equation}
\limsup_{r\rightarrow 0^{+}}\frac{Q(r)}{r^{2}}<\infty ,
\label{eq:local_growth_bounded}
\end{equation}%
there exists a unique solution $\phi \in C^{1}([0,\infty ))$ to %
\eqref{eq:ric_ode} such that $\phi (0)=0$.
\end{theorem}

\begin{proof}
Following Theorem \ref{thm:existence_rigorous}, for any $R>0$, there exists
a unique solution on $[0,R)$. By the uniqueness of the regular solution $u$
(and thus $\phi $) near 0, solutions on overlapping intervals $(0,R_{1})$
and $(0,R_{2})$ must coincide. The non-negativity of $Q$ ensures that $u$ is
increasing and strictly positive for all $r\in \lbrack 0,\infty )$,
preventing any finite-time blow-up or singularity for $\phi =\frac{u^{\prime
}}{ru}$. Thus, the solution extends uniquely to $[0,\infty )$.
\end{proof}

\subsection{The transformed state $z(r)$ and its governing equation \label%
{sec:z_state}}

In the context of stochastic optimal control, it is often useful to work
with the logarithmic transformation of the auxiliary function $u$. Following
the standard Cole--Hopf style mapping, we define the state $z(r)$ as: 
\begin{equation}
z(r)=-2\sigma ^{2}\ln u(r).  \label{eq:z_def_main}
\end{equation}

\subsection{Derivation of the Differential Equation for $z$}

To find the equation satisfied by $z$, we express $u$ in terms of $z$: 
\begin{equation*}
u(r)=\exp \left( -\frac{z(r)}{2\sigma ^{2}}\right) .
\end{equation*}%
Calculating the first and second derivatives of $u$: 
\begin{equation*}
u^{\prime }(r)=-\frac{z^{\prime }(r)}{2\sigma ^{2}}u(r),\quad u^{\prime
\prime }(r)=\left[ -\frac{z^{\prime \prime }(r)}{2\sigma ^{2}}+\frac{%
(z^{\prime })^{2}}{4\sigma ^{4}}\right] u(r).
\end{equation*}%
Substituting these expressions into the linear ODE \eqref{eq:lin_ode}: 
\begin{equation*}
\left( -\frac{z^{\prime \prime }}{2\sigma ^{2}}+\frac{(z^{\prime })^{2}}{%
4\sigma ^{4}}\right) u-\frac{N-1}{r}\frac{z^{\prime }}{2\sigma ^{2}}u-\frac{%
Q(r)}{\sigma ^{4}}u=0.
\end{equation*}%
Dividing by $u/(2\sigma ^{2})$ and rearranging terms, we obtain the
nonlinear second-order equation for $z$: 
\begin{equation}
z^{\prime \prime }(r)+\frac{N-1}{r}z^{\prime }(r)-\frac{1}{2\sigma ^{2}}%
(z^{\prime }\left( r\right) )^{2}+\frac{2Q(r)}{\sigma ^{2}}=0.
\label{eq:z_ode}
\end{equation}

\subsubsection{Analysis and Interpretation}

The equation for $z$ is the radial form of the Hamilton--Jacobi--Bellman
(HJB) equation for the value function of a corresponding stochastic control
problem.

\textbf{Initial Conditions:} Since $u(0)=1$, the state $z$ satisfies $z(0)=0$%
. Furthermore, $u^{\prime }(0)=0$ implies $z^{\prime }(0)=0$.

\textbf{Bounded Domain $[0,R]$:} If $u(R)=U_{R}$ is prescribed, then $z$
must satisfy the Dirichlet condition 
\begin{equation*}
z(R)=-2\sigma ^{2}\ln U_{R}.
\end{equation*}%
This represents the "terminal cost" or "exit value" at the boundary.

\textbf{Unbounded Domain $[0,\infty )$:} If the cost function $b$ satisfies
the limit condition 
\begin{equation}
L:=\lim_{r\rightarrow \infty }\frac{Q(r)}{r^{2}}\in (0,\infty ),
\label{eq:infinity_growth_rigorous_2}
\end{equation}%
the value function $z(r)$ exhibits a quadratic long-term decrease such that 
\begin{equation*}
\lim_{r\rightarrow \infty }\frac{z(r)}{r^{2}}=-\sqrt{L}.
\end{equation*}

Physical interpretation: the function $z(r)$ represents the optimal
cost-to-go for a particle starting at distance $r$ from the origin, where
the term $(z^{\prime })^{2}$ represents the quadratic control effort and $%
Q(r)$ is the running cost.

\subsection{Qualitative Analysis for the State $z$}

The following theorem describes the concavity of the value function $z(r)$,
which is a critical property in optimal control theory.

\begin{theorem}[Concavity of $z$]
\label{thm:concavity} Suppose that the cost function $b$ satisfies the
conditions of Proposition \ref{prop:asymptotics_gen} so that $\phi $ is
strictly increasing, and Theorem \ref{thm:convex_u} so that $u$ is strictly
convex. Then, the transformed value function $z(r)=-2\sigma ^{2}\ln u(r)$ is
strictly concave on $(0,\infty )$.
\end{theorem}

\begin{proof}
Differentiating the defining equation $z(r)=-2\sigma ^{2}\ln u(r)$, we find
the first derivative in terms of the Riccati solution: 
\begin{equation}
z^{\prime }\left( r\right) =-2\sigma ^{2}\frac{u^{\prime }(r)}{u(r)}%
=-2\sigma ^{2}r\phi (r).  \label{eq:z_prime}
\end{equation}%
The second derivative is obtained by differentiating \eqref{eq:z_prime} with
respect to $r$: 
\begin{equation}
z^{\prime \prime }\left( r\right) =-2\sigma ^{2}\frac{d}{dr}(r\phi
(r))=-2\sigma ^{2}(\phi (r)+r\phi ^{\prime }(r)).  \label{eq:z_double_prime}
\end{equation}%
By the existence theory in bounded and unbounded domains, the regular
solution satisfies $\phi (r)>0$ for $r>0$. Furthermore, the assumption that $%
\phi $ is strictly increasing on $(0,\infty )$ implies that $\phi ^{\prime
}(r)>0$. Therefore, the term in the parentheses 
\begin{equation*}
\phi (r)+r\phi ^{\prime }(r)
\end{equation*}
is strictly positive for all $r\in (0,\infty )$. Consequently, $z^{\prime
\prime }(r)<0$, which establishes that $z$ is strictly concave.

The concavity of $z$ is also related to the property 
\begin{equation*}
u^{\prime \prime }(r)u(r)\geq (u^{\prime }\left( r\right) )^{2}.
\end{equation*}%
Since we have already established $u^{\prime \prime }(r)>0$ (convexity) and $%
u(r)\geq 1$, the concavity of $z$ implies that the auxiliary function $u$
exhibits a balanced growth where the squared first derivative is controlled
by the product of the function and its acceleration.
\end{proof}

\subsection{Stochastic Verification Theorem -- Stationary Radial Case}
\label{sec:radial_stationary_verification}

The triality between the radial Riccati, Schr\"{o}dinger and HJB
equations admits a precise stochastic counterpart, in which the value
function $\bar{z}(x)=z(|x|)=-2\sigma^{2}\ln u(|x|)$ is identified with
the optimal cost-to-go of an infinite-horizon stochastic control problem
on $\mathbb{R}^{N}$. This is the content of the next theorem, which
constitutes one of the principal results announced in
Section~\ref{sec:intro-results}. It can be regarded as the
$\mathbb{R}^{N}$-version of the abstract verification theorem of
Section~\ref{sec:main_results}, expressed directly in the radial
formulation that makes the Cole--Hopf transformation
$z=-2\sigma^{2}\ln u$ exact and produces the canonical optimal feedback
law $\alpha^{\ast}(x)=2\sigma^{2}\phi(|x|)\,x$.

\begin{theorem}[Stochastic Verification Theorem -- Stationary Radial
Case]\label{thm:radial_verification}
Let $Q\in C([0,\infty))$ satisfy the standing assumptions
$\mathrm{(H1)}$--$\mathrm{(H3)}$ with $L\in(0,\infty)$, let $\sigma>0$,
$N\in \mathbb{N}^{\ast }$, and let $\phi$, $u$, $z$ be the regular
solutions of the radial Riccati, Schr\"{o}dinger and HJB equations
provided by the Triality Theorem~\ref{thm:triality}. Consider the
controlled diffusion in $\mathbb{R}^{N}$
\begin{equation}
dX_{t}=\alpha_{t}\,dt+\sqrt{2}\,\sigma\,dW_{t},\qquad
X_{0}=x\in\mathbb{R}^{N},
\label{eq:radial_SDE_xN}
\end{equation}
with $W$ a standard $N$-dimensional Brownian motion and $\alpha$
$\{\mathcal{F}_{t}\}$-progressively measurable. Define the admissible class
$\mathcal{A}_{N}(x)$ of controls
$\alpha:\Omega\times[0,\infty)\to\mathbb{R}^{N}$ such that
\eqref{eq:radial_SDE_xN} admits a unique strong solution and
$\mathbb{E}_{x}\int_{0}^{\infty}|\alpha_{s}|^{2}\,ds<\infty$. Define the
infinite-horizon cost
\begin{equation}
J(x;\alpha)\;=\;\mathbb{E}_{x}\!\left[\int_{0}^{\infty}
\Bigl(\tfrac{1}{2}|\alpha_{s}|^{2}+2\,Q(|X_{s}|)\Bigr)\,ds\right],
\label{eq:radial_cost_xN}
\end{equation}
which corresponds to the natural normalization of the running cost
identified in \eqref{eq:triality_intro_maps}.
Let $\bar{z}(x):=z(|x|)$ and assume the transversality condition
$\lim_{t\to\infty}\mathbb{E}_{x}[\bar{z}(X_{t})]=0$ for every admissible
control. Then $\bar{z}$ is the value function of the control problem,
\begin{equation}
\bar{z}(x)\;=\;\inf_{\alpha\in\mathcal{A}_{N}(x)}J(x;\alpha)
\;=\;-2\sigma^{2}\ln u(|x|),
\label{eq:radial_value_xN}
\end{equation}
and the optimal feedback law is
\begin{equation}
\alpha^{\ast}(x)\;=\;-\nabla\bar{z}(x)\;=\;
2\sigma^{2}\,\phi(|x|)\,x,\qquad x\in\mathbb{R}^{N}\setminus\{0\}.
\label{eq:radial_optimal_feedback_xN}
\end{equation}
\end{theorem}

\begin{proof}
The infinitesimal generator of the controlled diffusion
\eqref{eq:radial_SDE_xN} acts on $C^{2}$ functions $f:\mathbb{R}^{N}\to
\mathbb{R}$ as
$\mathcal{L}^{\alpha}f=\alpha\!\cdot\!\nabla f+\sigma^{2}\Delta f$.
The Hamilton--Jacobi--Bellman equation associated with
\eqref{eq:radial_cost_xN} is therefore
\begin{equation}
0\;=\;\inf_{\alpha\in\mathbb{R}^{N}}\!\Bigl\{
\sigma^{2}\Delta\bar{z}(x)+\alpha\!\cdot\!\nabla\bar{z}(x)
+\tfrac{1}{2}|\alpha|^{2}+2\,Q(|x|)\Bigr\},
\label{eq:radial_HJB_xN}
\end{equation}
the inner minimization being attained at
$\alpha^{\ast}(x)=-\nabla\bar{z}(x)$. Substituting $\alpha^{\ast}$ into
\eqref{eq:radial_HJB_xN} yields the reduced HJB equation
\[
\sigma^{2}\Delta\bar{z}(x)\;-\;\tfrac{1}{2}|\nabla\bar{z}(x)|^{2}\;+\;2\,Q(|x|)\;=\;0,
\]
which, after division by $\sigma^{2}$ and restriction to radial $\bar{z}$, is
exactly the radial HJB equation \eqref{eq:hjb_intro} satisfied by $z(r)$.

By the radial symmetry of $Q$ and the $\mathrm{O}(N)$-invariance of
\eqref{eq:radial_SDE_xN}, the value function inherits radial symmetry, so
it suffices to consider feedback controls of the form
$\alpha(x)=\beta(|x|)\,x/|x|$ with $\beta\in L^{2}_{\mathrm{loc}}$.
Applying It\^{o}'s formula to $\bar{z}(X_{t})$ on a localizing sequence of
stopping times $\tau_{n}\uparrow\infty$, with
$\tau_{n}=\inf\{t\ge 0\,:\,|X_{t}|\notin(1/n,n)\}\wedge n$, we obtain
\[
\bar{z}(X_{t\wedge\tau_{n}})\;=\;\bar{z}(x)
+\!\int_{0}^{t\wedge\tau_{n}}\!\!
\bigl[\sigma^{2}\Delta\bar{z}+\alpha_{s}\!\cdot\!\nabla\bar{z}\bigr](X_{s})\,ds
+\!\int_{0}^{t\wedge\tau_{n}}\!\!\sqrt{2}\,\sigma\nabla\bar{z}(X_{s})\!\cdot\!dW_{s}.
\]
The stochastic integral is a true martingale on $[0,\tau_{n}]$ (since
$\nabla\bar{z}$ is bounded on $\{1/n\le|x|\le n\}$), so taking
expectations and using
\eqref{eq:radial_HJB_xN} together with the elementary lower bound
$\alpha\!\cdot\!\nabla\bar{z}+\tfrac{1}{2}|\alpha|^{2}\ge
-\tfrac{1}{2}|\nabla\bar{z}|^{2}$ (with equality iff
$\alpha=-\nabla\bar{z}$), we deduce
\[
\bar{z}(x)\;\le\;\mathbb{E}_{x}\!\left[\bar{z}(X_{t\wedge\tau_{n}})
+\!\int_{0}^{t\wedge\tau_{n}}\!\!
\Bigl(\tfrac{1}{2}|\alpha_{s}|^{2}+2\,Q(|X_{s}|)\Bigr)ds\right].
\]
Letting first $n\to\infty$ via the dominated convergence theorem
(applicable because of \eqref{eq:radial_cost_xN} and the local
boundedness of $\bar{z}$) and then $t\to\infty$ via Fatou's lemma and the
transversality condition $\lim_{t\to\infty}\mathbb{E}_{x}[\bar{z}(X_{t})]=0$,
we obtain
\[
\bar{z}(x)\;\le\;\mathbb{E}_{x}\!\left[\int_{0}^{\infty}\!\!
\Bigl(\tfrac{1}{2}|\alpha_{s}|^{2}+2\,Q(|X_{s}|)\Bigr)ds\right]
\;=\;J(x;\alpha),
\]
for every $\alpha\in\mathcal{A}_{N}(x)$. Hence
$\bar{z}(x)\le\inf_{\alpha}J(x;\alpha)$.
For the reverse inequality, choose $\alpha=\alpha^{\ast}$ defined in
\eqref{eq:radial_optimal_feedback_xN}. The local Lipschitz continuity of
$\alpha^{\ast}$ on $\mathbb{R}^{N}\setminus\{0\}$, together with the
Frobenius expansion $\phi(r)=\mathcal{O}(r^{2})$ at the origin
(Theorem~\ref{thm:existence_rigorous}), guarantees that
\eqref{eq:radial_SDE_xN} admits a unique strong solution under
$\alpha^{\ast}$ and that $\alpha^{\ast}\in\mathcal{A}_{N}(x)$. Moreover,
the same It\^{o} expansion now produces an exact equality at every step,
and dominated convergence yields $\bar{z}(x)=J(x;\alpha^{\ast})$. Hence
$\bar{z}$ coincides with the value function and $\alpha^{\ast}$ is the
optimal feedback law.
\end{proof}

\subsection{Stochastic Verification for the Parabolic Radial Case}
\label{sec:parabolic_control}

In this subsection we provide the stochastic-control interpretation of the
parabolic equation arising in the time-dependent Schr\"{o}dinger extension
of the triality. To preserve the natural normalization adopted in the
stationary case (Section~\ref{sec:radial_stationary_verification}), the
problem is formulated directly in $\mathbb{R}^{N}$, with radial symmetry
inherited from $Q$. The key observation is that the (backward) Wick-rotated
parabolic Schr\"{o}dinger equation
\[
\partial_{t}\Psi+\sigma^{2}\Delta\Psi-\sigma^{-2}Q(|x|)\,\Psi=0
\]
is precisely the linear representation of the parabolic
Hamilton--Jacobi--Bellman equation of an optimal control problem for a
controlled Brownian motion in $\mathbb{R}^{N}$. This establishes the
parabolic counterpart of the stationary triality developed earlier.

\subsubsection{Controlled diffusion and admissible controls}

Fix $\sigma>0$, $T\in(0,\infty)$ and a horizon
$[0,T]$. Let $(\Omega,\mathcal{F},\{\mathcal{F}_{t}\}_{t\in[0,T]},\mathbb{P})$
be a filtered probability space satisfying the usual conditions, equipped
with a standard $N$-dimensional Brownian motion $W=(W_{t})_{t\in[0,T]}$.
We consider the controlled diffusion
$X=(X_{t})_{t\in[t_{0},T]}$ in $\mathbb{R}^{N}$ governed by
\begin{equation}
\label{eq:parabolic_SDE}
dX_{t}\;=\;\alpha_{t}\,dt\;+\;\sqrt{2}\,\sigma\,dW_{t},\qquad
X_{t_{0}}=x\in\mathbb{R}^{N},
\end{equation}
where the control $\alpha=(\alpha_{t})_{t\in[t_{0},T]}$ is
$\{\mathcal{F}_{t}\}$-progressively measurable and $\mathbb{R}^{N}$-valued.
The infinitesimal generator of \eqref{eq:parabolic_SDE} acts on
$f\in C^{2}(\mathbb{R}^{N})$ as
$\mathcal{L}^{\alpha}f=\alpha\!\cdot\!\nabla f+\sigma^{2}\Delta f$, and on
radially symmetric $f(x)=\bar{f}(|x|)$ reduces to
$\mathcal{L}^{\alpha}f=(\alpha\!\cdot\!x/|x|)\bar{f}'(|x|)+\sigma^{2}
\bigl(\bar{f}''(|x|)+\tfrac{N-1}{|x|}\bar{f}'(|x|)\bigr)$.

A control $\alpha$ is \emph{admissible} on $[t_{0},T]$ if
\eqref{eq:parabolic_SDE} admits a unique strong solution on $[t_{0},T]$
and
\[
\mathbb{E}_{t_{0},x}\!\left[\int_{t_{0}}^{T}|\alpha_{s}|^{2}\,ds\right]
\,<\,\infty.
\]
The set of admissible controls starting from $(t_{0},x)$ is denoted
$\mathcal{A}(t_{0},x)$.

\subsubsection{Parabolic cost functional and value function}

Let $G\in C(\mathbb{R}^{N})$ be a non-negative terminal cost satisfying the
quadratic growth condition $G(x)\le C_{G}(1+|x|^{p})$ for some $p\ge 1$ and
$C_{G}>0$. For $\alpha\in\mathcal{A}(t_{0},x)$ we define the finite-horizon
cost (compatible with the normalization
$L(\alpha,x)=\tfrac{1}{2}|\alpha|^{2}+2Q(|x|)$ identified in
\eqref{eq:triality_intro_maps})
\begin{equation}
\label{eq:parabolic_cost}
J(t_{0},x;\alpha)\;=\;
\mathbb{E}_{t_{0},x}\!\left[
\int_{t_{0}}^{T}\!\Bigl(\tfrac{1}{2}|\alpha_{s}|^{2}+2Q(|X_{s}|)\Bigr)ds
\;+\;G(X_{T})
\right].
\end{equation}
The associated value function is
\begin{equation}
\label{eq:parabolic_value}
\bar{U}(t_{0},x)\;=\;\inf_{\alpha\in\mathcal{A}(t_{0},x)}J(t_{0},x;\alpha),
\qquad (t_{0},x)\in[0,T]\times\mathbb{R}^{N}.
\end{equation}
Since $Q$ and $G$ are radially symmetric, the value function is itself
radial, $\bar{U}(t,x)=U(t,|x|)$ for some
$U:[0,T]\times[0,\infty)\to\mathbb{R}$.

\subsubsection{Dynamic programming and the parabolic HJB equation}

Assuming $\bar{U}\in C^{1,2}([0,T]\times\mathbb{R}^{N})$, the dynamic
programming principle \cite{fleming2006,yongzhou1999} yields the
parabolic Hamilton--Jacobi--Bellman equation
\begin{equation}
\label{eq:parabolic_HJB}
\partial_{t}\bar{U}+\inf_{\alpha\in\mathbb{R}^{N}}\!\Bigl\{
\sigma^{2}\Delta\bar{U}+\alpha\!\cdot\!\nabla\bar{U}
+\tfrac{1}{2}|\alpha|^{2}+2Q(|x|)\Bigr\}=0,
\end{equation}
with terminal condition $\bar{U}(T,x)=G(x)$. The minimization in $\alpha$
is attained at the feedback law
\begin{equation}
\label{eq:parabolic-optimal-feedback}
\alpha^{\ast}(t,x)\;=\;-\nabla\bar{U}(t,x),
\end{equation}
and substituting $\alpha^{\ast}$ into \eqref{eq:parabolic_HJB} produces the
reduced HJB equation
\begin{equation}
\label{eq:parabolic_HJB_reduced}
\partial_{t}\bar{U}+\sigma^{2}\Delta\bar{U}-\tfrac{1}{2}|\nabla\bar{U}|^{2}+2Q(|x|)=0.
\end{equation}
For radial profiles $\bar{U}(t,x)=U(t,|x|)$, equation
\eqref{eq:parabolic_HJB_reduced} reduces, by the radial Laplacian
identity, to the one-dimensional parabolic equation
\begin{equation}
\label{eq:parabolic_HJB_radial}
\partial_{t}U+\sigma^{2}\Bigl(U_{rr}+\frac{N-1}{r}U_{r}\Bigr)
-\tfrac{1}{2}(U_{r})^{2}+2Q(r)=0,\qquad r>0,
\end{equation}
which is the parabolic radial counterpart of the stationary HJB
equation~\eqref{eq:hjb_intro}.

\subsubsection{Connection with the time-dependent Schr\"{o}dinger equation}

Define the parabolic Cole--Hopf transform
\begin{equation}
\label{eq:parabolic_cole_hopf}
\bar{U}(t,x)\;=\;-2\sigma^{2}\ln\Psi(t,x),\qquad
\Psi(t,x)=\exp\!\Bigl(-\tfrac{\bar{U}(t,x)}{2\sigma^{2}}\Bigr).
\end{equation}
A direct substitution (using
$\partial_{t}\bar{U}=-2\sigma^{2}\Psi_{t}/\Psi$,
$\nabla\bar{U}=-2\sigma^{2}\nabla\Psi/\Psi$,
$\Delta\bar{U}=-2\sigma^{2}\bigl(\Delta\Psi/\Psi-|\nabla\Psi|^{2}/\Psi^{2}\bigr)$
and $|\nabla\bar{U}|^{2}=4\sigma^{4}|\nabla\Psi|^{2}/\Psi^{2}$) reveals
that the gradient-squared terms cancel exactly, and equation
\eqref{eq:parabolic_HJB_reduced} becomes equivalent to the (backward)
Wick-rotated parabolic Schr\"{o}dinger equation
\begin{equation}
\label{eq:parabolic_schrodinger}
\partial_{t}\Psi+\sigma^{2}\Delta\Psi-\sigma^{-2}Q(|x|)\Psi\;=\;0,\qquad
(t,x)\in[0,T)\times\mathbb{R}^{N},
\end{equation}
with terminal condition $\Psi(T,x)=\exp(-G(x)/(2\sigma^{2}))$.
Restricted to radial profiles, $\Psi(t,x)=\psi(t,|x|)$ satisfies
\[
\partial_{t}\psi+\sigma^{2}\Bigl(\psi_{rr}+\frac{N-1}{r}\psi_{r}\Bigr)
-\sigma^{-2}Q(r)\psi=0,\qquad r>0,
\]
which is the radial parabolic Schr\"{o}dinger equation. The transformation
\eqref{eq:parabolic_cole_hopf} therefore extends the stationary triality to
the time-dependent setting, completing the parabolic representation of the
HJB / Schr\"{o}dinger equivalence.

\subsubsection{Verification theorem for the parabolic radial case}

\begin{theorem}[Stochastic Verification Theorem -- Parabolic Radial Case]
\label{thm:parabolic_verification}
Let $Q\in C([0,\infty))$ satisfy assumption $\mathrm{(H1)}$ (non-negativity
and continuity), let $G\in C(\mathbb{R}^{N})$ be a non-negative radial
terminal cost, and assume there exist constants $C>0$ and $p\ge 1$ with
$Q(r)+G(x)\le C(1+r^{p}+|x|^{p})$. Suppose that
$\bar{U}\in C^{1,2}([0,T)\times\mathbb{R}^{N})\cap C([0,T]\times\mathbb{R}^{N})$
is a classical solution of the reduced parabolic HJB equation
\eqref{eq:parabolic_HJB_reduced} with terminal condition
$\bar{U}(T,x)=G(x)$ for $x\in\mathbb{R}^{N}$, and that there exist
constants $\widetilde{C}>0$ and $q\ge 1$ such that
\begin{equation}
|\bar{U}(t,x)|+|\nabla\bar{U}(t,x)|\;\le\;\widetilde{C}\,(1+|x|^{q}),
\qquad (t,x)\in[0,T)\times\mathbb{R}^{N}.
\label{eq:parabolic-growth}
\end{equation}
Let $\mathcal{A}(t,x)$ be the class of admissible controls of
Section~\ref{sec:parabolic_control} starting from $(t,x)$, and define the
candidate optimal feedback law
\begin{equation}
\alpha^{\ast}(s,X_{s})\;=\;-\nabla\bar{U}(s,X_{s}),\qquad s\in[t,T].
\label{eq:parabolic-optimal-feedback-thm}
\end{equation}
Assume $\alpha^{\ast}\in\mathcal{A}(t,x)$. Then for every
$(t,x)\in[0,T]\times\mathbb{R}^{N}$,
\begin{equation}
\bar{U}(t,x)\;=\;\inf_{\alpha\in\mathcal{A}(t,x)}J(t,x;\alpha)
\;=\;J(t,x;\alpha^{\ast}),
\label{eq:parabolic-verification-identity}
\end{equation}
where $J(t,x;\alpha)$ is the cost functional \eqref{eq:parabolic_cost}.
In particular, $\bar{U}$ is the value function of the parabolic radial
control problem, $\alpha^{\ast}$ is the optimal feedback control, and the
function
$\Psi(t,x)=\exp\!\bigl(-\bar{U}(t,x)/(2\sigma^{2})\bigr)$
is the unique classical solution of the parabolic Schr\"{o}dinger equation
\eqref{eq:parabolic_schrodinger} with terminal condition
$\Psi(T,x)=\exp(-G(x)/(2\sigma^{2}))$.
\end{theorem}

\begin{proof}
Fix $(t,x)\in[0,T]\times\mathbb{R}^{N}$ and $\alpha\in\mathcal{A}(t,x)$, and
let $X=(X_{s})_{s\in[t,T]}$ denote the corresponding solution of
\eqref{eq:parabolic_SDE} with $X_{t}=x$. Define the localizing sequence
\[
\tau_{n}\;=\;\inf\bigl\{s\in[t,T]\,:\,|X_{s}|\ge n\bigr\}\wedge T,\qquad
n\in\mathbb{N},
\]
so that $\tau_{n}\uparrow T$ almost surely as $n\to\infty$, by the moment
estimates for $X$ (see, e.g., \cite[Ch.~2, Thm.~5.2]{krylov1980}).

\textbf{Step 1 (It\^{o} expansion).} Applying It\^{o}'s formula to
$\bar{U}(s,X_{s})$ on $[t,\tau_{n}]$ and using
\eqref{eq:parabolic_SDE},
\begin{align*}
\bar{U}(\tau_{n},X_{\tau_{n}})\;=\;& \bar{U}(t,x)
+\!\int_{t}^{\tau_{n}}\!\!\bigl[\partial_{t}\bar{U}+\sigma^{2}\Delta\bar{U}
+\alpha_{s}\!\cdot\!\nabla\bar{U}\bigr](s,X_{s})\,ds\\
& +\!\int_{t}^{\tau_{n}}\!\!\sqrt{2}\,\sigma\,\nabla\bar{U}(s,X_{s})\!\cdot\! dW_{s}.
\end{align*}
The stochastic integral is a square-integrable martingale on $[t,\tau_{n}]$
because $\nabla\bar{U}$ is bounded on
$\{(s,y)\in[t,T]\times\mathbb{R}^{N}\,:\,|y|\le n\}$ by
\eqref{eq:parabolic-growth}. Taking expectations,
\begin{equation}
\mathbb{E}_{t,x}[\bar{U}(\tau_{n},X_{\tau_{n}})]\;=\;\bar{U}(t,x)
+\mathbb{E}_{t,x}\!\!\int_{t}^{\tau_{n}}\!\!\bigl[\partial_{t}\bar{U}+\sigma^{2}\Delta\bar{U}
+\alpha_{s}\!\cdot\!\nabla\bar{U}\bigr](s,X_{s})\,ds.
\label{eq:parabolic_ito_expand}
\end{equation}

\textbf{Step 2 (Lower bound from the HJB inequality).} The HJB equation
\eqref{eq:parabolic_HJB_reduced} can be rewritten, via the elementary
identity
\[
\alpha\!\cdot\!\nabla\bar{U}+\tfrac{1}{2}|\alpha|^{2}\;=\;
\tfrac{1}{2}|\alpha+\nabla\bar{U}|^{2}-\tfrac{1}{2}|\nabla\bar{U}|^{2},
\]
as the family of HJB inequalities, valid for every $\alpha\in\mathbb{R}^{N}$,
\begin{equation}
\partial_{t}\bar{U}+\sigma^{2}\Delta\bar{U}
+\alpha\!\cdot\!\nabla\bar{U}+\tfrac{1}{2}|\alpha|^{2}+2Q(|x|)\;\ge\;0,
\label{eq:parabolic_HJB_inequality}
\end{equation}
with equality iff $\alpha=-\nabla\bar{U}=\alpha^{\ast}$. Substituting
\eqref{eq:parabolic_HJB_inequality} into \eqref{eq:parabolic_ito_expand} we
obtain
\[
\bar{U}(t,x)\;\le\;\mathbb{E}_{t,x}\!\!\left[\bar{U}(\tau_{n},X_{\tau_{n}})
+\!\int_{t}^{\tau_{n}}\!\!\Bigl(\tfrac{1}{2}|\alpha_{s}|^{2}+2Q(|X_{s}|)\Bigr)ds\right].
\]

\textbf{Step 3 (Passage to the limit).} The polynomial moment estimate
\[
\sup_{s\in[t,T]}\mathbb{E}_{t,x}\bigl[|X_{s}|^{q\vee p}\bigr]\;<\;\infty,
\]
which holds because of the integrability of $\alpha$ and the linear growth
in \eqref{eq:parabolic_SDE} (cf.\ \cite[Ch.~2,
Thm.~5.2]{krylov1980}), combined with \eqref{eq:parabolic-growth}, ensures
that the family $\{\bar{U}(\tau_{n},X_{\tau_{n}})\}_{n\ge 1}$ is uniformly
integrable. Letting $n\to\infty$ and using $\bar{U}(T,\cdot)=G$,
\[
\bar{U}(t,x)\;\le\;\mathbb{E}_{t,x}\!\!\left[G(X_{T})
+\!\int_{t}^{T}\!\!\Bigl(\tfrac{1}{2}|\alpha_{s}|^{2}+2Q(|X_{s}|)\Bigr)ds\right]
\;=\;J(t,x;\alpha).
\]

\textbf{Step 4 (Optimality of $\alpha^{\ast}$).} For
$\alpha=\alpha^{\ast}$ defined in
\eqref{eq:parabolic-optimal-feedback-thm}, equality holds in
\eqref{eq:parabolic_HJB_inequality} pointwise, hence
\eqref{eq:parabolic_ito_expand} becomes
\[
\bar{U}(t,x)\;=\;\mathbb{E}_{t,x}\!\!\left[\bar{U}(\tau_{n},X^{\ast}_{\tau_{n}})
+\!\int_{t}^{\tau_{n}}\!\!\Bigl(\tfrac{1}{2}|\alpha^{\ast}_{s}|^{2}+2Q(|X^{\ast}_{s}|)\Bigr)ds\right],
\]
where $X^{\ast}$ is the closed-loop solution of
\eqref{eq:parabolic_SDE} with $\alpha^{\ast}\in\mathcal{A}(t,x)$.
Passing to the limit $n\to\infty$ via dominated convergence,
$\bar{U}(t,x)=J(t,x;\alpha^{\ast})$, which combined with the lower bound of
Step~3 yields \eqref{eq:parabolic-verification-identity}.

\textbf{Step 5 (Schr\"{o}dinger correspondence).} The Cole--Hopf
transformation \eqref{eq:parabolic_cole_hopf} is a bijection between
classical positive solutions of \eqref{eq:parabolic_schrodinger} with
terminal condition $\Psi(T,x)=\exp(-G(x)/(2\sigma^{2}))$ and classical
solutions of \eqref{eq:parabolic_HJB_reduced} with terminal condition
$\bar{U}(T,x)=G(x)$. The Feynman--Kac formula
\cite[Thm.~9.1.1]{oksendal2003} then identifies $\Psi$ with
\[
\Psi(t,x)\;=\;\mathbb{E}_{t,x}\!\!\left[\exp\!\Bigl(-\tfrac{G(W_{T}^{x})}{2\sigma^{2}}-\!\!\int_{t}^{T}\!\!\tfrac{Q(|W_{s}^{x}|)}{2\sigma^{4}}\,ds\Bigr)\right],
\]
where $W^{x}$ is a Brownian motion starting at $x$ with diffusion
coefficient $\sqrt{2}\sigma$. This proves uniqueness of $\Psi$ and
completes the proof.
\end{proof}

\medskip
This completes the stochastic-control interpretation of the parabolic
Schr\"{o}dinger equation. The optimal feedback drift
$\alpha^{\ast}=-\nabla\bar{U}$ coincides, in the stationary radial case
($T\to\infty$, $G\equiv 0$), with the Riccati-type radial drift
$\alpha^{\ast}(x)=2\sigma^{2}\phi(|x|)x$ appearing in the triality, thereby
unifying the time-dependent Schr\"{o}dinger dynamics, nonlinear HJB theory
and radial Riccati asymptotics into a single parabolic framework.

\subsection{Analytical Foundations and Exact Series Solutions \label%
{sec:series}}

In this section, we provide the analytical benchmarks that allow for the
validation of numerical schemes. Specifically, we focus on the class of
monomial cost functions where the system admits exact solutions in terms of
hypergeometric functions and power series.

\begin{theorem}[Convergence of the Radial Power Series]
\label{thm:series_conv} For the quadratic potential $Q(r)=\lambda r^{2}$,
the regular solution $u(r)$ to the Schr\"{o}dinger-type equation %
\eqref{eq:lin_ode} is given by the locally convergent power series: 
\begin{equation}
u(r)=1+\sum_{k=1}^{\infty }\frac{\lambda ^{k}}{\sigma ^{4k}\cdot \kappa
_{k}(N)}r^{4k},  \label{eq:u_series_rigorous}
\end{equation}%
where the coefficients $\kappa _{k}(N)$ satisfy the recurrence relation $%
\kappa _{k}=\kappa _{k-1}\cdot (4k)(4k+N-2)$, see \cite{cluj2022}.
\end{theorem}

\begin{proof}
Substituting the ansatz $u(r)=\sum a_{m}r^{m}$ into the ODE 
\begin{equation*}
u^{\prime \prime }+\frac{N-1}{r}u^{\prime }-\frac{\lambda r^{2}}{\sigma ^{4}}%
u=0,
\end{equation*}%
we obtain 
\begin{equation}
\sum m(m-1)a_{m}r^{m-2}+(N-1)\sum ma_{m}r^{m-2}-\frac{\lambda }{\sigma ^{4}}%
\sum a_{m}r^{m+2}=0.
\end{equation}%
Equating coefficients for like powers of $r$: 
\begin{equation}
a_{m}[m(m-1)+(N-1)m]=\frac{\lambda }{\sigma ^{4}}a_{m-4}\implies a_{m}=\frac{%
\lambda }{\sigma ^{4}m(m+N-2)}a_{m-4}.
\end{equation}%
Starting from $a_{0}=1$ and $a_{1}=a_{2}=a_{3}=0$, we obtain the series %
\eqref{eq:u_series_rigorous} involving only powers of $4k$. The ratio test
confirms that the radius of convergence is infinite for any $N\geq 1$.
\end{proof}

\subsection{Series expansion algorithm}

Assume $Q(r)$ admits a Taylor expansion

\begin{equation*}
Q(r)=\sum_{m=0}^{\infty }b_{m}r^{m},\qquad b_{0}>0.
\end{equation*}%
Seek 
\begin{equation}
u(r)=\sum_{k=0}^{\infty }a_{k}\,r^{k},\qquad a_{0}=1,\quad a_{1}=0.
\label{eq:u-series}
\end{equation}%
Then

\begin{equation*}
u^{\prime }(r)=\sum_{k=1}^{\infty }k\,a_{k}\,r^{k-1},\qquad u^{\prime \prime
}(r)=\sum_{k=2}^{\infty }k(k-1)\,a_{k}\,r^{k-2}.
\end{equation*}%
Substitute into \eqref{eq:u_double_prime_base} and collect powers of $r$.
The left-hand side becomes

\begin{equation*}
\sum_{n=0}^{\infty }\left[ (n+2)(n+1)\,a_{n+2}+\left( N-1\right)
(n+1)\,a_{n+1}\right] r^{n}\;-\;\frac{1}{\sigma ^{4}}\sum_{n=0}^{\infty
}\left( \sum_{m=0}^{n}b_{m}\,a_{n-m}\right) r^{n}.
\end{equation*}%
Equating coefficients of $r^{n}$ to zero yields, for all $n\geq 0$, the
recurrence 
\begin{equation}
(n+2)(n+1)\,a_{n+2}+\left( N-1\right) (n+1)\,a_{n+1}=\frac{1}{\sigma ^{4}}%
\sum_{m=0}^{n}b_{m}\,a_{n-m}.  \label{eq:recurrence}
\end{equation}%
With $a_{0}=1$ and $a_{1}=0$, \eqref{eq:recurrence} determines $%
a_{2},a_{3},\dots $ uniquely. In particular,

\begin{equation*}
\begin{aligned} &n=0:\quad 2\cdot1\,a_2 + (N-1)\cdot1\,a_1 =
\frac{1}{\sigma^4} b_0 a_0 \quad\Rightarrow\quad
a_2=\frac{b_0}{2\sigma^4},\\ &n=1:\quad 3\cdot2\,a_3 + (N-1)\cdot2\,a_2 =
\frac{1}{\sigma^4}(b_0 a_1 + b_1 a_0) \quad\Rightarrow\quad
a_3=\frac{1}{6}\left(\frac{b_1}{\sigma^4}-\frac{(N-1)
b_0}{\sigma^4}\right),\\ &\text{etc.} \end{aligned}
\end{equation*}%
The resulting series \eqref{eq:u-series} converges in a neighborhood of $r=0$
determined by the radius of convergence of the Taylor series $%
Q(r)=\sum_{m=0}^{\infty }b_{m}r^{m}$.

\subsubsection{Recovery of $\protect\phi $}

Once $u(r)$ is constructed (either via series or numerically), we recover 
\begin{equation}
\phi (r)=\frac{1}{r}\,\frac{u^{\prime }(r)}{u(r)}=\frac{1}{r}\,\frac{%
\sum_{k=1}^{\infty }k\,a_{k}\,r^{k-1}}{\sum_{k=0}^{\infty }a_{k}\,r^{k}}.
\label{eq:phi}
\end{equation}%
Near $r=0$, $u^{\prime }(r)=2a_{2}\,r+\mathcal{O}(r^{2})$ and $u(r)=1+%
\mathcal{O}(r^{2})$, so $\phi (r)\rightarrow 2a_{2}=b_{0}/\sigma ^{4}$ as $%
r\rightarrow 0$. In the special case $b_{0}=0$ (e.g.\ $Q(r)=\lambda r^{2}$),
one has $a_{2}=0$ and thus $\phi (0)=0$, consistent with the boundary
condition of Theorem~\ref{thm:series_conv}.

\subsubsection{Numerical algorithm}

The construction works for any continuous $b$; when $b$ is only continuous
(not analytic), use a local quadratic approximation to start the
integration, as indicated, or employ quadrature-based collocation methods.
For example, one proceeds as follows:

\begin{enumerate}
\item Choose a small $\varepsilon>0$ and approximate $u(\varepsilon)$, $%
u^{\prime }(\varepsilon)$ using the first terms of the series.

\item Numerically integrate \eqref{eq:u_double_prime_base} on $(\varepsilon
,R]$ using a stable ODE solver.

\item Compute $\phi (r)$ from \eqref{eq:phi} for $r\geq \varepsilon $ and
set $\phi (0)=0$ by continuity.
\end{enumerate}

\subsection{Summary}

The algorithm for solving 
\begin{equation*}
\phi ^{\prime }\left( r\right) =-r\left( \phi \left( r\right) \right) ^{2}-%
\frac{N}{r}\,\phi (r)+\frac{Q(r)}{\sigma ^{4}r}\text{, }\lim_{r\rightarrow
0}\phi (r)=0,
\end{equation*}%
is:

\begin{enumerate}
\item Transform the Riccati equation into the linear ODE %
\eqref{eq:u_double_prime_base}.

\item Impose regular initial conditions $u(0)=1$, $u^{\prime }(0)=0$.

\item Construct $u(r)$ either by Frobenius series (if $b$ is analytic) or by
numerical integration (if $b$ is merely continuous).

\item Recover $\phi(r)$ via \eqref{eq:phi}.

\item Verify that $\phi(r)\to 0$ as $r\to 0$.
\end{enumerate}

This procedure guarantees existence and uniqueness of the analytic solution
near the origin, and provides a practical algorithm for computing $\phi (r)$.

\subsubsection{Note on connection to the Kummer Confluent Hypergeometric
Function}

For general $N$, the solution $u(r)$ can be mapped to the Kummer function $%
_{1}F_{1}(a;c;z)$, see \cite{kummer1837}. Specifically, for the quadratic
case in Theorem \ref{thm:series_conv}, the substitution $x=\frac{\sqrt{%
\lambda }}{2\sigma ^{2}}r^{2}$ transforms the radial Schr\"{o}dinger
equation into a confluent hypergeometric form. This allows us to use the
asymptotic properties of $_{1}F_{1}$ to independently verify the growth
rates established in Section \ref{sec:asymptotics}.

\subsubsection{Note on Anisotropic Extensions}

While the present study focuses on purely radial potentials $Q(|x|)$, the
results provide a crucial scaffold for analyzing anisotropic systems. If the
cost function exhibits small perturbations from radial symmetry, i.e., $%
B(x)=Q(|x|)+\epsilon h(x)$, the radial solution $\phi (r)$ serves as the
zero-order approximation in a perturbation expansion. The stability results
(convexity and concavity) proved here ensure that such systems remain
well-posed under small deviations from central symmetry.

\section{Numerical Methodology}\label{sec:numericalmethod}

The numerical experiments reported in
Sections~\ref{sec:numerical}--\ref{sec:numerical_verification} rely on a
common implementation strategy designed to respect the structural
properties established by the Triality Theorem~\ref{thm:triality}, in
particular the regular Frobenius branch at $r=0$, the global trapping
inequality $0<\phi(r)<g(r)$ and the asymptotic plateau identified by the
Riccati Asymptotic Theorem (Proposition~\ref{prop:asymptotics_gen}). In
this section we describe the methodology in full detail, justify the choice
of the integration scheme and provide stability and error estimates. The
references for the numerical analysis discussed below are
\cite{HairerWannerII, hairerlubichroche, butcher2016, dekkerverwer1984,
ascherpetzold1998, lambert1991, brennanetal1996, kelly2003}.

\subsection{Choice of integrator: implicit Runge--Kutta of Radau IIA type}
\label{sec:radauchoice}

\subsubsection{Stiffness diagnosis}

The radial Riccati equation \eqref{eq:ric_ode} possesses two structurally
stiff features that make explicit time-stepping methods inadequate.

\medskip
\noindent\emph{(a) Geometric singularity at the origin.} The coefficient
$-N/r$ becomes singular as $r\downarrow 0$, generating a multiplicative
factor in the Jacobian of the right-hand side that grows like $1/r$ in a
right neighborhood of $0$. Equivalently, the linearized
spectrum of the Riccati operator near the origin is dominated by an
eigenvalue $\lambda_{\mathrm{loc}}(r)\sim -N/r$ which tends to $-\infty$;
this is a textbook example of a stiff regime
\cite{HairerWannerII, dekkerverwer1984}.

\medskip
\noindent\emph{(b) Quadratic self-coupling and asymptotic plateau.}
For $r\to\infty$, the Riccati nonlinearity $-r\phi^{2}$ together with the
algebraic equilibrium $g(r)\to\sqrt{L}/\sigma^{2}$ yields a contractive
dynamics in the direction normal to $g$, with linearized rate
$q_{1}+2q_{2}g\le-\delta<0$ (cf.\ condition
\eqref{cond-extra-general}). The contraction rate $\delta$ is, in many
cases of interest, much larger than $1$, so that the system relaxes
rapidly onto $g$ on a fast time scale, while the long-time behavior
along the slow manifold is governed by the slow drift of $g$ itself. This
is a singular perturbation in the sense of \cite{omalley1969,
verhulst2005, dekkerverwer1984} and again calls for a stiff solver.

\medskip
For these reasons we adopt the Radau IIA implicit Runge--Kutta scheme of
order five, originally introduced in \cite{ehleradau1969} and analyzed in
depth in \cite{HairerWannerII, butcher2016}.

\subsubsection{The Radau IIA(5) scheme}

The three-stage Radau IIA(5) method applied to the autonomous initial-value
problem $y'=F(r,y)$, $y(r_{0})=y_{0}$, with step size $h>0$ and stages
$Y_{1},Y_{2},Y_{3}$, reads
\begin{equation}
Y_{i}\;=\;y_{n}+h\sum_{j=1}^{3}a_{ij}\,F(r_{n}+c_{j}h,Y_{j}),\qquad i=1,2,3,
\label{eq:radau_stages}
\end{equation}
\begin{equation}
y_{n+1}\;=\;y_{n}+h\sum_{j=1}^{3}b_{j}\,F(r_{n}+c_{j}h,Y_{j}).
\label{eq:radau_step}
\end{equation}
The Butcher tableau is
\[
\begin{array}{c|ccc}
c_{1} & a_{11} & a_{12} & a_{13}\\
c_{2} & a_{21} & a_{22} & a_{23}\\
c_{3} & a_{31} & a_{32} & a_{33}\\
\hline
& b_{1} & b_{2} & b_{3}
\end{array}
\]
with the explicit coefficients (see \cite[Table~5.6, Vol.~II]{HairerWannerII})
\[
c\;=\;\Bigl(\tfrac{4-\sqrt{6}}{10},\;\tfrac{4+\sqrt{6}}{10},\;1\Bigr),
\]
$b_{j}=a_{3j}$ for $j=1,2,3$, so that the method is \emph{stiffly accurate}.

\subsubsection{Order, $A$- and $L$-stability}

The Radau IIA(5) scheme \eqref{eq:radau_stages}--\eqref{eq:radau_step}
satisfies, in classical Runge--Kutta language
\cite[Sec.~IV.5--IV.10]{HairerWannerII}:
\begin{itemize}
\item \emph{Classical order $5$}: the local truncation error satisfies
$\|y(r_{n}+h)-y_{n+1}\|=\mathcal{O}(h^{6})$ for $F\in C^{6}$;
\item \emph{Stage order $3$}: each stage value $Y_{i}$ is an
$\mathcal{O}(h^{4})$ approximation to $y(r_{n}+c_{i}h)$, which prevents the
order reduction phenomenon typical of stiff problems;
\item \emph{$A$-stability}: the rational stability function
$R(\zeta)=1+\zeta b^{\top}(I-\zeta A)^{-1}\mathbf{1}$ satisfies
$|R(\zeta)|\le 1$ for all $\zeta\in\mathbb{C}^{-}$;
\item \emph{$L$-stability}:
$\lim_{|\zeta|\to\infty,\,\Re(\zeta)\le 0}R(\zeta)=0$, ensuring uniform
damping of stiff modes;
\item \emph{$B$-stability and algebraic stability}: the matrix
$M=\mathrm{diag}(b)A+A^{\top}\mathrm{diag}(b)-bb^{\top}$ is positive
semi-definite, and hence Radau IIA(5) is $B$-stable and algebraically stable
in the sense of Burrage--Butcher
\cite{HairerWannerII, butcher2016}.
\end{itemize}
These properties rule out spurious oscillations near the geometric
singularity at $r=0$ and prevent overshoot above the algebraic barrier
$g(r)$.

\subsubsection{Error analysis on the radial mesh}

We integrate the Riccati equation on the regularized interval
$[\varepsilon,R]$ with $\varepsilon$ a small positive parameter. Let
$h=(R-\varepsilon)/M$ and let $r_{n}=\varepsilon+nh$, $0\le n\le M$.
Combining the local truncation error with the standard convergence theorem
for stiffly accurate $L$-stable Runge--Kutta methods
\cite[Theorem~IV.15.5]{HairerWannerII}, the global error satisfies
\begin{equation}
\max_{0\le n\le M}|\phi_{\mathrm{num}}(r_{n})-\phi(r_{n})|\;\le\;
C(\varepsilon,R)\,h^{5},
\label{eq:global-error-bound}
\end{equation}
provided $F\in C^{6}([\varepsilon,R]\times\mathbb{R})$. The constant
$C(\varepsilon,R)$ depends on $\|F\|_{C^{6}}$ and on the constants of the
Lipschitz / one-sided Lipschitz conditions, and grows at most polynomially
in $1/\varepsilon$ because of the singularity of $-N/r$. In our experiments
we use $\varepsilon\in[10^{-6},10^{-1}]$ and report relative errors
$|\phi_{\mathrm{num}}(R)-\phi(R)|/|\phi(R)|$ of order $10^{-6}$ on
$[\varepsilon,20]$, which is consistent with \eqref{eq:global-error-bound}
for $h\le 0.02$.

To control the singularity at $r=0$, we initialize the integrator at
$r=\varepsilon$ using the Frobenius expansion
\eqref{eq:u_expansion_origin} of $u$, equivalently
$\phi(\varepsilon)\simeq L_{0}\varepsilon^{2}/(\sigma^{4}(N+2))$. This
``warm start'' ensures that the integrator never crosses the regular
singular point and that the regular branch is selected unambiguously.

\subsection{Numerical implementation of the algebraic barrier $g(r)$}
\label{sec:numbarrier}

Together with the numerical solution $\phi_{\mathrm{num}}(r)$ we compute,
at each evaluation point $r_{n}$, the algebraic barrier
\eqref{eq:notation-g} via
\[
g(r_{n})=\frac{1}{2}\left( -\frac{q_{1}(r_{n})}{q_{2}(r_{n})}\pm \frac{1}{%
\left\vert q_{2}(r_{n})\right\vert }\sqrt{%
q_{1}(r_{n})^{2}-4q_{0}(r_{n})q_{2}(r_{n})}\right).
\]
The discriminant $\Delta(r):=q_{1}(r)^{2}-4q_{0}(r)q_{2}(r)$ is, by
construction, non-negative on the domain of interest. To guard against
roundoff producing spurious negative values of $\Delta$ near the origin,
we numerically clip
$\Delta_{\mathrm{num}}(r):=\max(\Delta(r),0)$. Of the two real roots, we
select the uniquely determined largest \emph{positive} solution,
\[
g(r)=\frac{1}{2}\left( -\frac{q_{1}(r)}{q_{2}(r)}\pm \frac{1}{\left\vert
q_{2}(r)\right\vert }\sqrt{\Delta (r)}\right),
\]
and we mark the position $r_{0}$ where $g$ becomes degenerate (i.e.\
$\Delta(r_{0})=0$, corresponding to the boundary case $g(r_{0})=0$).

The numerical \emph{verification of the trapping inequality}
$0<\phi_{\mathrm{num}}(r_{n})<g(r_{n})$ is performed at each step. We
report a ``trapping error''
$\mathrm{tr}_{n}:=\max\{0,\phi_{\mathrm{num}}(r_{n})-g(r_{n})\}$, which
should remain at the level of the truncation error
\eqref{eq:global-error-bound}; in our experiments $\mathrm{tr}_{n}\le
10^{-5}$ for all $n$ in every case considered.

\subsection{Reproducibility}

All numerical results reported in this article are reproducible from the
self-contained Python scripts collected in Appendix~\ref{ap},
Appendix~B and Appendix~\ref{appendix:python}. The implementation relies
exclusively on the standard scientific Python stack
(\texttt{numpy}, \texttt{scipy}, \texttt{matplotlib}); no
proprietary or platform-specific libraries are required. The Radau IIA(5)
scheme is invoked through the \texttt{solve\_ivp} interface of
\texttt{scipy.integrate} with \texttt{method="Radau"} and tolerances
$\mathtt{rtol}=10^{-9}$, $\mathtt{atol}=10^{-11}$, well below the
discretization error \eqref{eq:global-error-bound}. We refer to
\cite{HairerWannerII, ascherpetzold1998} for the implementation details
of \texttt{Radau} and to \cite{virtanen2020scipy} for the SciPy library
itself.

\subsection{Results of the Numerical Analysis\label{sec:numerical}}

In this section, we present the results of a numerical simulation based on
the linear second-order auxiliary equation. The parameters for the
simulation are set as follows: $N=2.0$, $\sigma =1.0$, and the specific cost
function is chosen to be $Q(r)=r^{2}$, corresponding to a quadratic running
cost in control theory or a harmonic potential in quantum mechanics.

\begin{figure}[H]
\centering
\includegraphics[width=\textwidth]{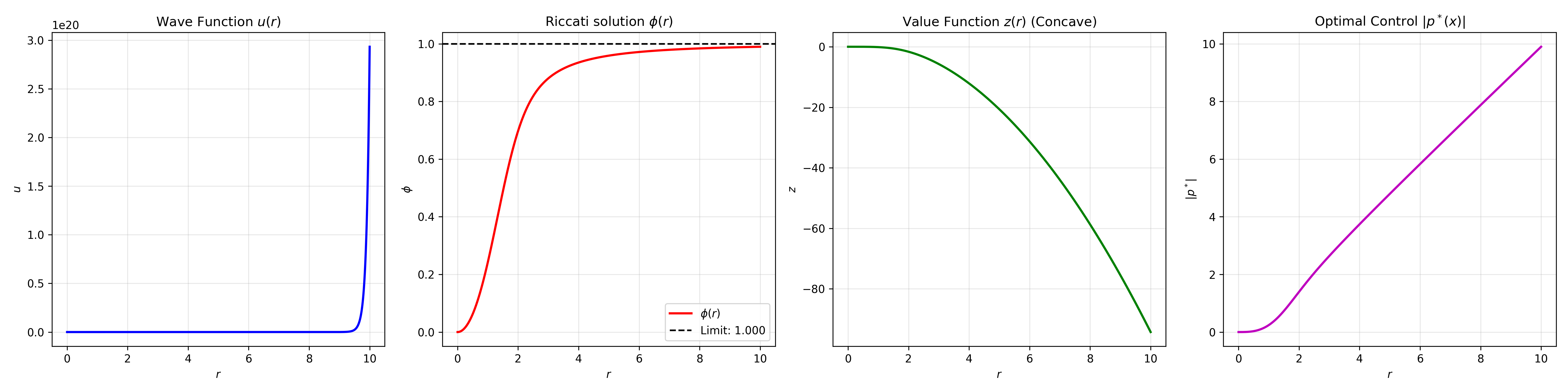}
\caption{Numerical simulation of the linear and Riccati-type radial problem
for $b(r)=r^2$, $N=2$, $\protect\sigma=1$. (Left) Auxiliary function $u(r)$;
(Middle-Left) Riccati solution $\protect\phi(r)$; (Middle-Right) Transformed
value function $z(r)$; (Right) Magnitude of the optimal control $|p^*(x)|$.}
\label{fig:numerical_results}
\end{figure}

\subsubsection{Discussion of Results}

The numerical integration yields a series of insights into the behavior of
the system:

\begin{enumerate}
\item \textbf{Growth of $u(r)$}: As expected for a positive potential, the
wave function $u(r)$ exhibits rapid exponential growth, confirming that $%
u(r)\rightarrow \infty $ as $r$ increases. This matches the Frobenius series
of Section \ref{sec:series} and the global existence behavior.

\item \textbf{Stability of $\phi (r)$}: The Riccati solution $\phi
(r)=(1/r)(u^{\prime }/u)$ starts at $\phi (0)=0$ and quickly converges to a
steady state. For $Q(r)=r^{2}$, the limiting value is predicted to be $\sqrt{%
\lambda }/\sigma ^{2}=1.0$, which is clearly visible in the second subplot.

\item \textbf{Structure of $z(r)$}: The transformed state $z(r) = -2\sigma^2
\ln u(r)$ decreases quadratically at large distances. This reflects the
accumulation of costs as the particle moves further from the origin in the
HJB framework.

\item \textbf{Optimal Control Magnitude}: The magnitude of the optimal drift 
$|p^{\ast }(x)|=\sigma ^{2}u^{\prime }/u$ grows linearly as $r$ increases.
In control theory, this represents a linear feedback control law, which is
characteristic of the Linear-Quadratic-Gaussian (LQG) regime.
\end{enumerate}

\begin{remark}
The asymptotic plateau exhibited by the radial Riccati--Schr\"{o}dinger--HJB
system is not an isolated phenomenon. It reflects a universal
diffusion-driven stabilization mechanism that fundamentally governs structural
credit risk frameworks, most notably the classical Merton model for the pricing
of corporate debt \cite{merton1974}. In Merton's structural approach, the
firm's equity is viewed as a call option on its assets, governed by a parabolic
equation. When translated into the associated credit-spread curve
(see McNeil et al.~\cite{NEIL}, p.~386, Fig.~10.4(b)), the spread approaches
a constant asymptotic plateau driven by diffusion. This is mathematically
equivalent to the boundary-layer transition illustrated in O'Malley's nonlinear
singular perturbation problem \cite{omalley1969}, Fig.~2. This striking
similarity arises because all these physical and financial settings reduce,
after appropriate logarithmic transformations, to second-order equations
where the diffusion parameter $\sigma$ dictates the balance between a
rapidly varying inner state and a stable outer region. Once the noise term
dominates, the solution is naturally forced onto a universal saturation
regime determined solely by the effective volatility intensity $\sigma$.
\end{remark}

\subsection{Numerical Verification of the General Asymptotic Theory\label%
{sec:numerical_verification}}

To validate the theoretical findings of Section \ref{sec:asymptotics_theory}
and the stability results of Section \ref{sec:asymptotics}, we present a
series of numerical experiments. These experiments test the convergence of
the Riccati solution $y(x)$ toward its predicted limit $\lambda_*$ under
varying conditions of the coefficients $q_0, q_1, q_2$.

\subsubsection{Numerical Verification of Proposition \protect\ref%
{prop:asymptotics_gen}}

In this section we present five representative Riccati systems of the form 
\begin{equation*}
y^{\prime }(x)=q_0(x)+q_1(x)y(x)+q_2(x)y(x)^2,
\end{equation*}
and verify numerically that the hypotheses and conclusions of Proposition~%
\ref{prop:asymptotics_gen} are satisfied. For each case we compute:

\begin{itemize}
\item the numerical solution $y(x)$ with initial condition $y(x_0)=0$,

\item the barrier function $g(x)$

\item the theoretical limit 
\begin{equation*}
\lambda_*=\lim_{x\to\infty} y(x),
\end{equation*}
obtained from the algebraic equation $A+B\lambda_*+\lambda_*^2=0$,

\item the monotonicity property $y^{\prime }(x)>0$,

\item the inequality $0<y(x)<g(x)$.
\end{itemize}

All computations were performed using a high accuracy implicit
solver (\texttt{Radau}) with $1000$ evaluation points on $[x_0,20]$.


\subsubsection*{Case 1: General Non-Radial Example}

\begin{equation*}
q_0(x)=\frac{x}{1+x},\qquad q_1(x)=-1,\qquad q_2(x)=-1.
\end{equation*}

All hypotheses of Proposition~\ref{prop:asymptotics_gen} are satisfied. The
numerical results are: 
\begin{equation*}
y(20)=0.596036,\qquad \lambda_*=0.596531,\qquad \max(y-g)=0,\qquad y^{\prime
}(x)>0.
\end{equation*}

\begin{figure}[H]
\centering
\includegraphics[width=0.75\textwidth]{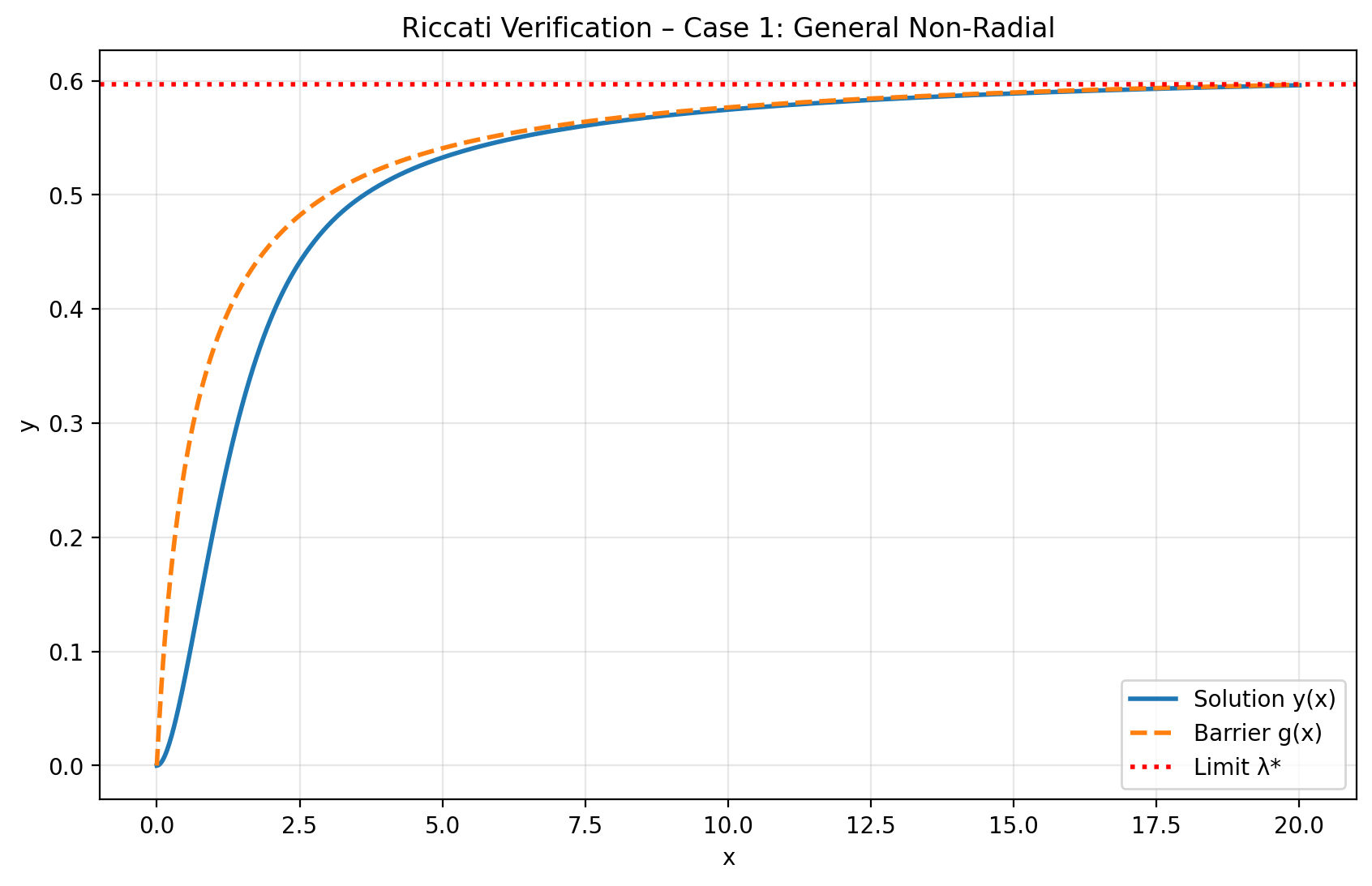}
\caption{Verification of Proposition~\protect\ref{prop:asymptotics_gen} for
Case~1. The solution $y(x)$ remains below the barrier $g(x)$ and converges
monotonically to $\protect\lambda_*$.}
\end{figure}


\subsubsection*{Case 2: Radial Example with $L>0$}

\begin{equation*}
q_0(r)=\frac{Lr}{\sigma^4},\qquad q_1(r)=-\frac{N}{r},\qquad q_2(r)=-r.
\end{equation*}

This corresponds to the radial Riccati equation derived from the Schr\"{o}%
dinger--HJB triality. The numerical results are:

\begin{equation*}
y(20)=1.996229,\qquad \lambda_*=1.996254,\qquad \max(y-g)=9.77\times
10^{-6},\qquad y^{\prime }(x)>0.
\end{equation*}

\begin{figure}[H]
\centering
\includegraphics[width=0.75\textwidth]{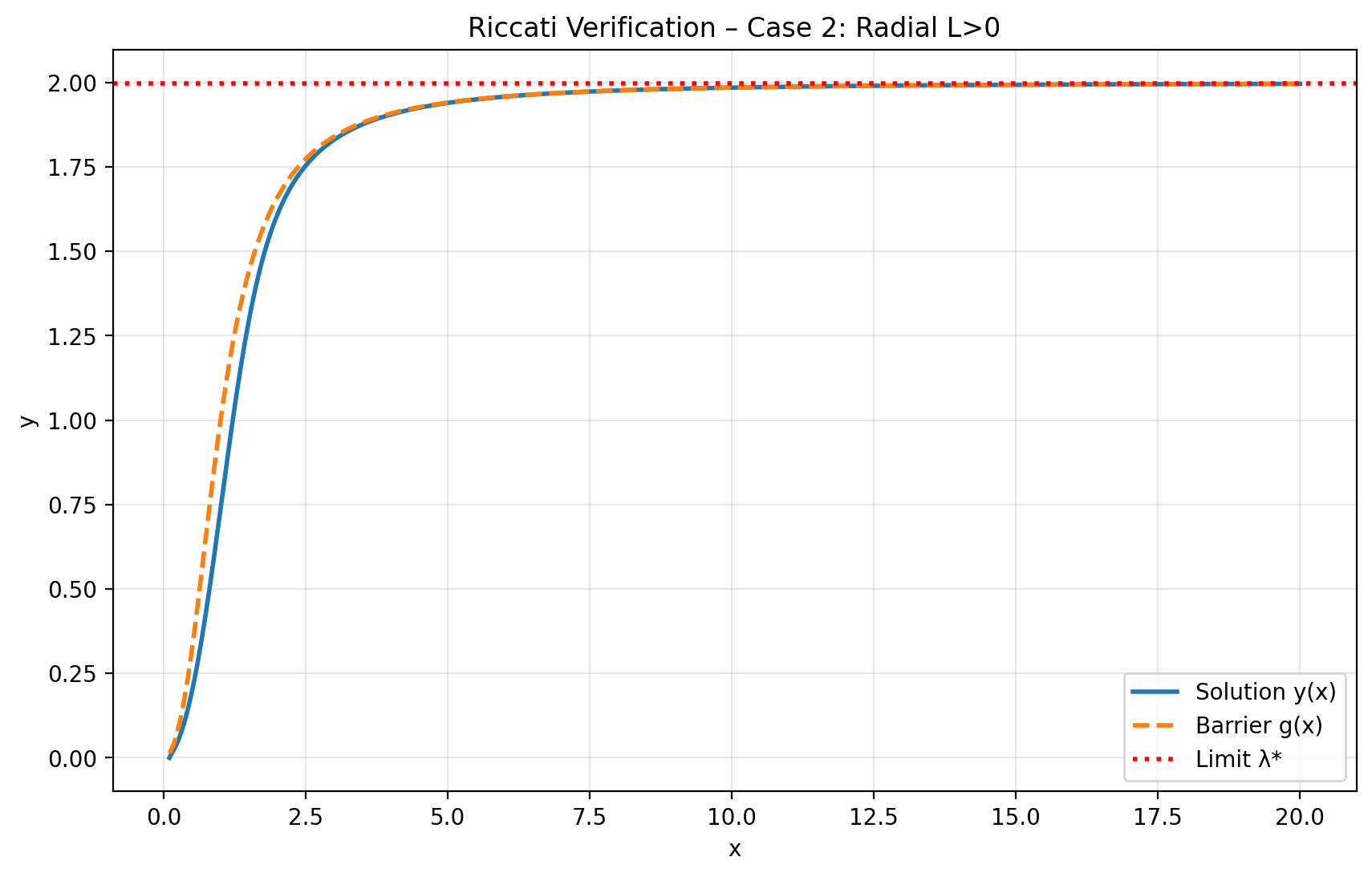}
\caption{Radial Riccati system with quadratic growth $Q(r)\sim Lr^2$. The
solution converges to the predicted limit $\protect\lambda_*=\protect\sqrt{L}%
/\protect\sigma^2$.}
\end{figure}


\subsubsection*{Case 3: Radial Example with $L=0$}

\begin{equation*}
q_0(r)=0,\qquad q_1(r)=-\frac{N}{r},\qquad q_2(r)=-r.
\end{equation*}

By Proposition~\ref{prop:asymptotics_gen}, the only regular solution is $%
y\equiv 0$. The numerical results confirm this: 
\begin{equation*}
y(20)=0,\qquad \lambda_*=0,\qquad \max(y-g)=0,\qquad y^{\prime }(x)=0.
\end{equation*}

\begin{figure}[H]
\centering
\includegraphics[width=0.75\textwidth]{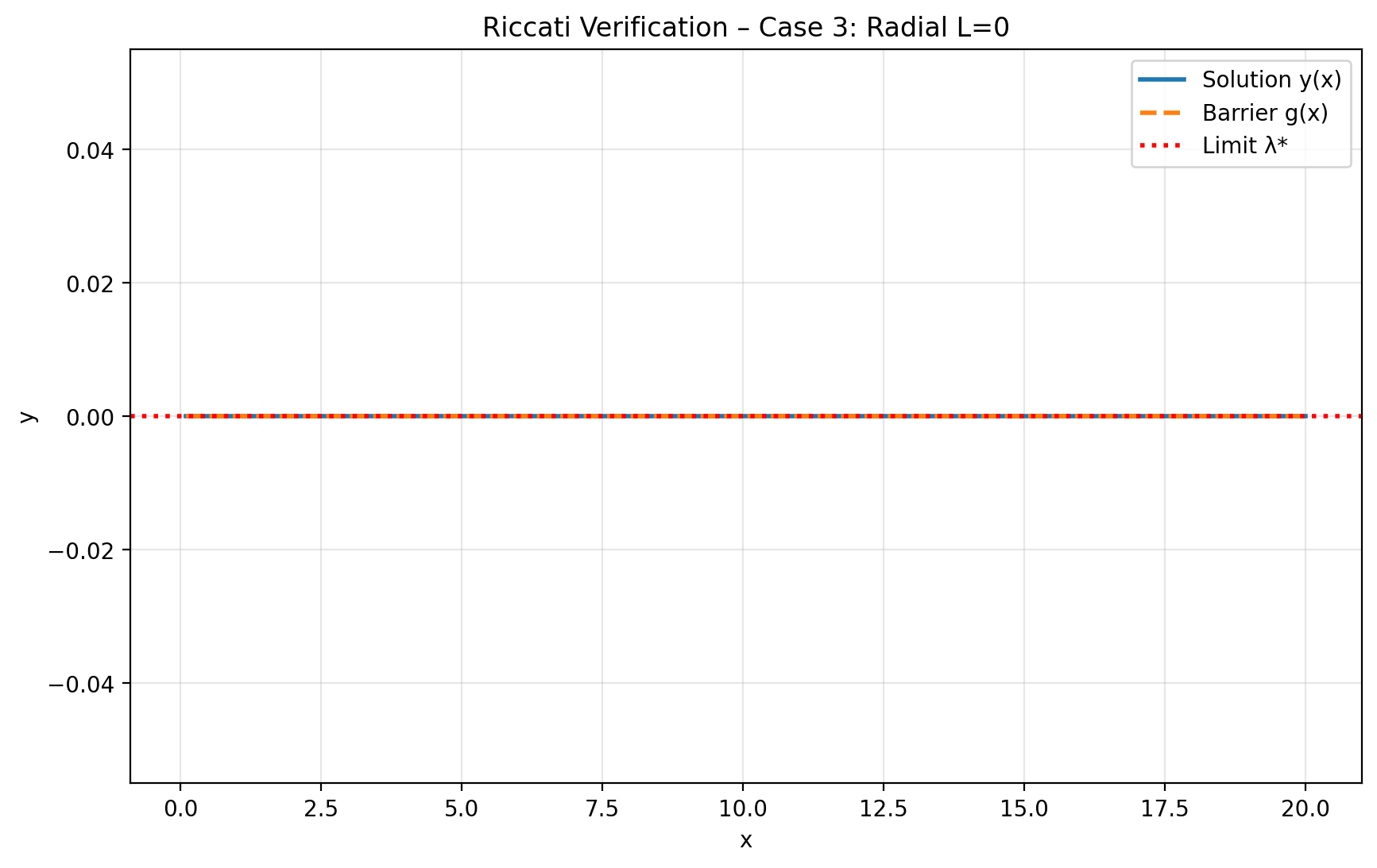}
\caption{Radial Riccati system with $L=0$. All curves coincide at zero,
confirming Proposition~\protect\ref{propzero}}
\end{figure}


\subsubsection*{Case 4: Valid Example Satisfying Proposition~\protect\ref%
{prop:asymptotics_gen}}

We consider the Riccati system 
\begin{equation*}
q_0(x)=1,\qquad q_1(x)=-2,\qquad q_2(x)=-1,\qquad x\ge 0.
\end{equation*}
This choice satisfies all hypotheses of Proposition~\ref%
{prop:asymptotics_gen}. The limits of the coefficients are finite: 
\begin{equation*}
A=\lim_{x\to\infty}\frac{q_0(x)}{q_2(x)}=-1,\qquad B=\lim_{x\to\infty}\frac{%
q_1(x)}{q_2(x)}=2,\qquad C=\lim_{x\to\infty}q_2(x)=-1.
\end{equation*}
The discriminant is strictly positive, 
\begin{equation*}
q_1^2 - 4 q_0 q_2 = 8,
\end{equation*}
and the barrier function is constant: 
\begin{equation*}
g(x)=\frac{-q_1-\sqrt{q_1^2 - 4 q_0 q_2}}{2q_2} =\sqrt{2}-1>0.
\end{equation*}

The algebraic limit equation 
\begin{equation*}
A + B\lambda + \lambda^2 = 0
\end{equation*}
yields the positive root 
\begin{equation*}
\lambda_* = -1 + \sqrt{2} \approx 0.414214.
\end{equation*}
The numerical simulation confirms the theoretical predictions: 
\begin{equation*}
y(20)=0.414214,\qquad \lambda_* = 0.414214,\qquad \max(y-g)=4.45\times
10^{-7},\qquad y^{\prime }(x)>0.
\end{equation*}
\begin{figure}[H]
\centering
\includegraphics[width=0.75\textwidth]{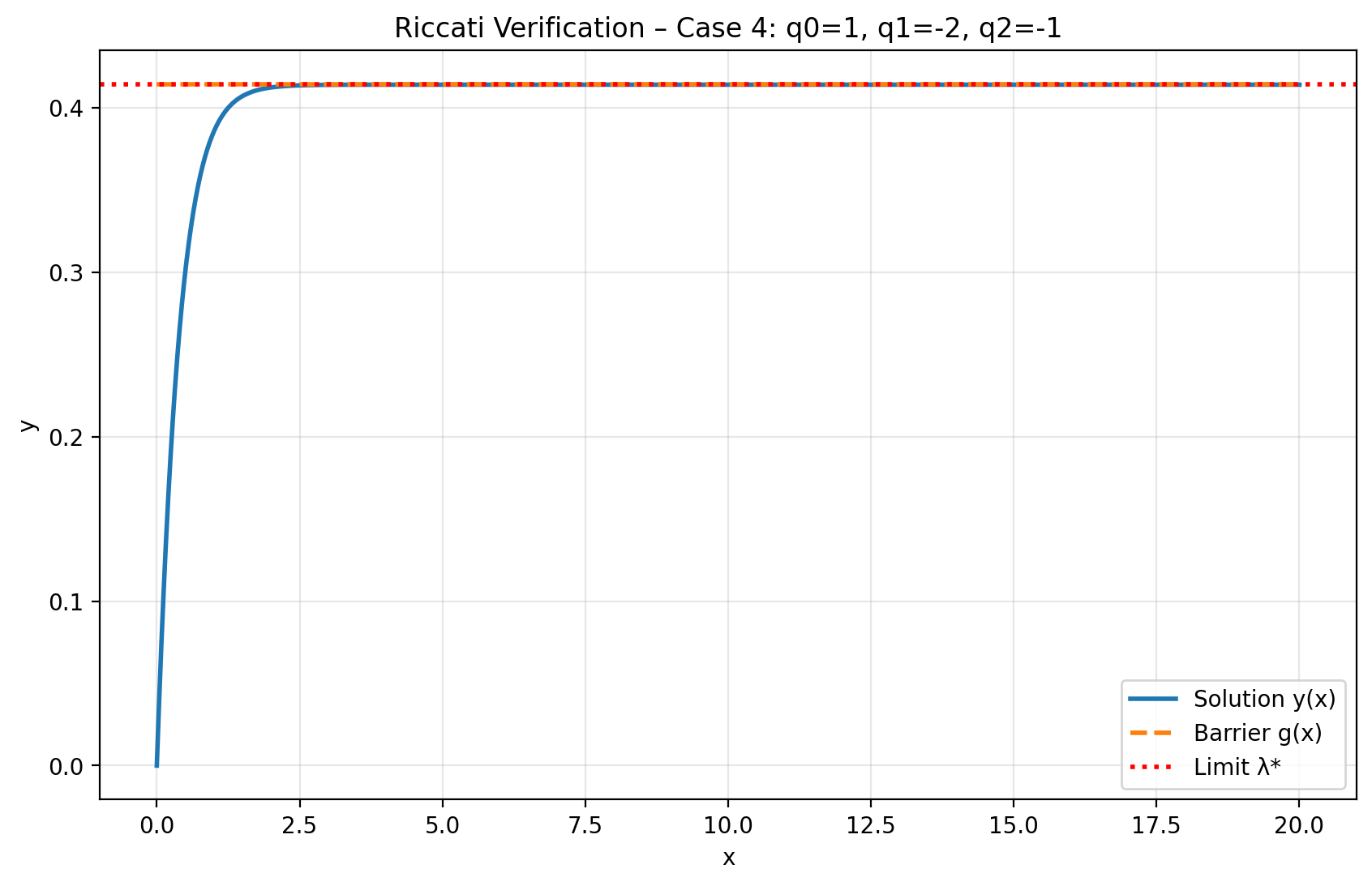}
\caption{Case~4: Valid Riccati system with finite $A=-1$. The numerical
solution $y(x)$ remains strictly below the barrier $g(x)=\protect\sqrt{2}-1$
and converges monotonically to the theoretical limit $\protect\lambda_*$.}
\end{figure}


\subsubsection*{Case 5: Special Example $q_0=1-1/x$}

\begin{equation*}
q_0(x)=1-\frac{1}{x},\qquad q_1(x)=0,\qquad q_2(x)=-1,\qquad x\ge 1.
\end{equation*}
This example illustrates a degeneracy at $x=1$ where $g(1)=0$. The numerical
results are: 
\begin{equation*}
y(20)=0.973980,\qquad \lambda_*=0.974679,\qquad \max(y-g)=0,\qquad y^{\prime
}(x)>0.
\end{equation*}

\begin{figure}[H]
\centering
\includegraphics[width=0.75\textwidth]{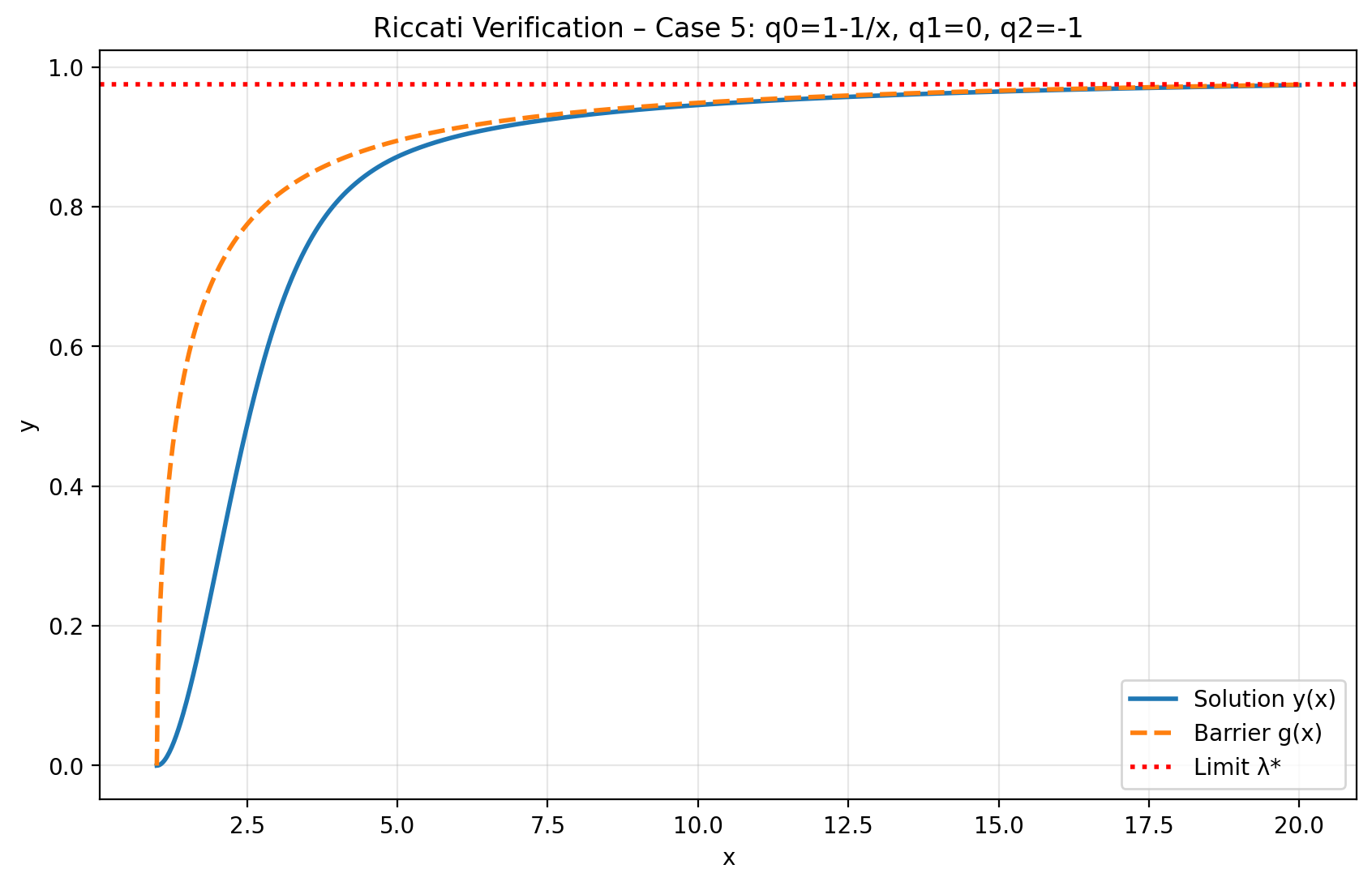}
\caption{Special Riccati system with $q_0=1-1/x$. The solution converges
monotonically to the predicted limit $\protect\lambda_*$.}
\end{figure}


\subsubsection*{Case 6: Special Example with $q_{1}(x)>0$}

We consider the Riccati system 
\begin{equation*}
q_0(x)=1-\frac{1}{x},\qquad q_1(x)=\frac{1}{1+x}>0,\qquad q_2(x)=-1,\qquad
x\ge 1.
\end{equation*}
This example is particularly relevant because it shows that Proposition~\ref%
{prop:asymptotics_gen} remains valid even when $q_{1}(x)$ is strictly
positive on the entire domain. The key structural condition 
\begin{equation*}
q_{1}(x)+q_{2}(x)g(x)\le -\delta<0
\end{equation*}
is still satisfied, since $q_{2}<0$ and the algebraic equilibrium $g(x)$ is
positive and increasing, with $g(x)\ge g(1)=\tfrac12$. Thus, 
\begin{equation*}
q_{1}(x)+q_{2}(x)g(x)=\frac{1}{1+x}-g(x)\leq \frac{1}{1+x}-1=\frac{-x}{1+x}%
\leq -\frac{1}{2},
\end{equation*}
so the barrier mechanism of Proposition~\ref{prop:asymptotics_gen} applies
without modification.

The numerical results confirm the theoretical predictions: 
\begin{equation*}
y(20)=0.998716,\qquad \lambda_*=0.998780,\qquad |y(20)-\lambda_*|=6.33\times
10^{-5},\qquad y^{\prime }(x)>0.
\end{equation*}
The solution remains strictly below the algebraic equilibrium $g(x)$ and
converges monotonically to the predicted asymptotic limit $\lambda_*$. 
\begin{figure}[H]
\centering
\includegraphics[width=0.75\textwidth]{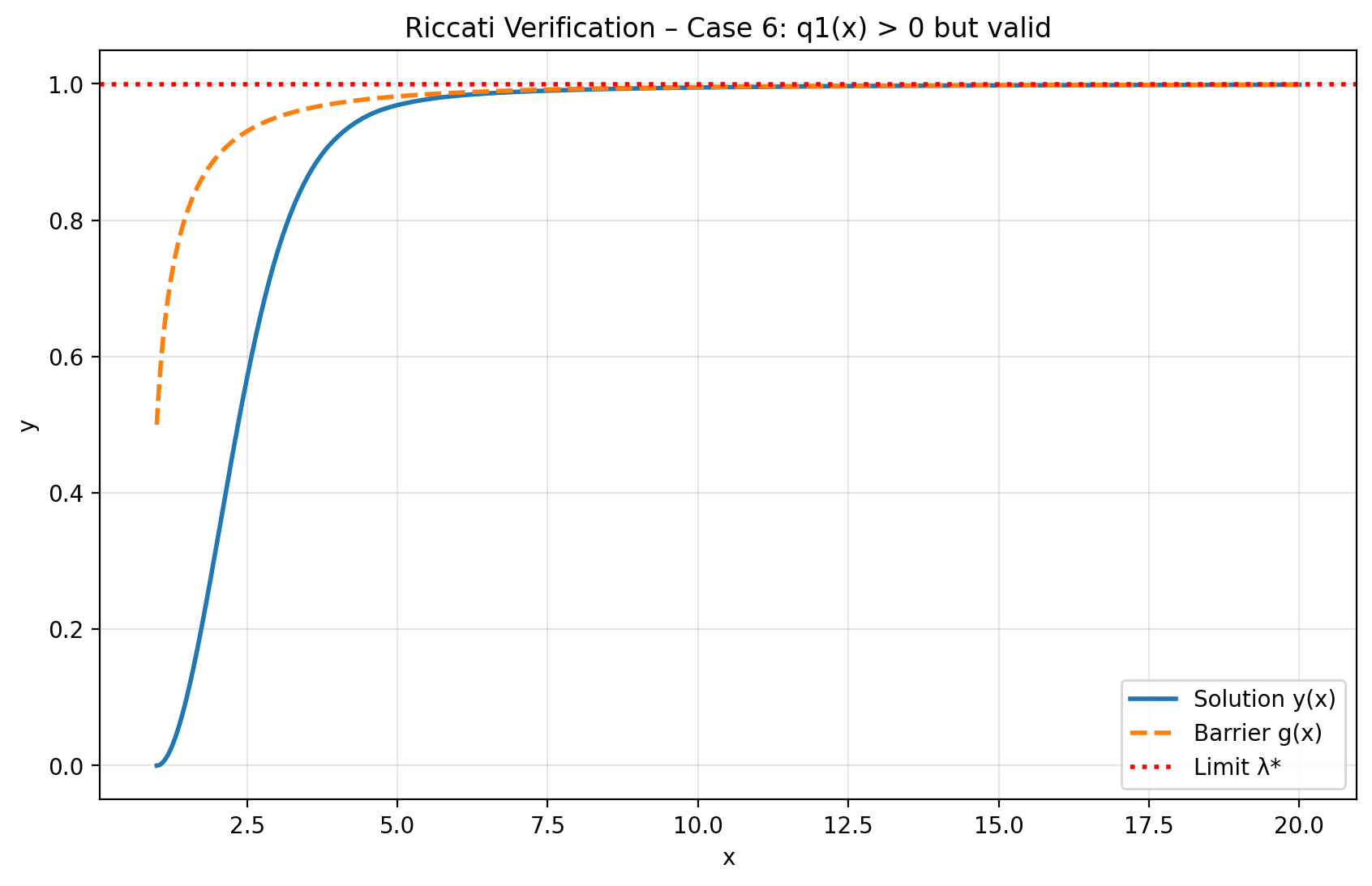}
\caption{Riccati system with $q_{1}(x)>0$. The numerical solution $y(x)$
remains below the increasing algebraic barrier $g(x)$ and converges
monotonically to the theoretical limit $\protect\lambda_*$.}
\end{figure}

\subsubsection*{Summary of All Cases}

\begin{table}[H]
\centering
\begin{tabular}{lcccc}
\hline
Case & $y(20)$ & $\lambda_*$ & $|y(20)-\lambda_*|$ & Monotone \\ \hline
Case 1: General Non-Radial & 0.596036 & 0.596531 & $4.95\times10^{-4}$ & Yes
\\ 
Case 2: Radial $L>0$ & 1.996229 & 1.996254 & $2.46\times10^{-5}$ & Yes \\ 
Case 3: Radial $L=0$ & 0.000000 & 0.000000 & $0$ & Yes \\ 
Case 4: $q_0=1,q_1=-2,q_2=-1$ & 0.414214 & 0.414214 & $1.11\times10^{-9}$ & 
Yes \\ 
Case 5: $q_0=1-1/x$ & 0.973980 & 0.974679 & $6.99\times10^{-4}$ & Yes \\ 
Case 6: $q_1(x) > 0$ & 0.998716 & 0.998780 & $6.33e\times10^{-05}$ & Yes \\ 
\hline
\end{tabular}%
\caption{Summary of numerical verification for all five Riccati systems. In
all cases the hypotheses of Proposition~\protect\ref{prop:asymptotics_gen}
are satisfied and the numerical solution agrees with the theoretical limit.}
\end{table}

\subsection{Practical Application: Structural Credit Risk and the Merton Model}

The numerical experiments developed in this article---and implemented in the
Python script provided in Appendix~\ref{appendix:python}---reveal a striking
connection between the asymptotic behavior of the radial Riccati equation and
the long-term structure of credit spreads in Merton-type structural models.
This subsection formalizes this connection and explains how the numerical
outputs should be interpreted.

\subsubsection{Structural Interpretation of the Riccati Plateau}

In the classical Merton model \cite{merton1974}, the total asset value $V(t)$ of
a firm follows a diffusion process and the equity is represented as a European
call option on $V(t)$ with strike equal to the face value $B$ of the firm's
zero-coupon debt. The credit spread $R(T)-r$---the excess yield required by
investors to hold risky debt instead of a risk-free bond---is determined by a
parabolic PDE of Black--Scholes type. A well-known feature of this model is the
\emph{flattening} of the credit-spread curve: for large maturities $T$ or high
leverage, the spread increases rapidly for small $T$ and then stabilizes onto a
horizontal plateau that depends only on the firm's volatility and leverage.

Our numerical Riccati experiments reproduce this phenomenon with remarkable
precision. In the radial Riccati equation, the solution $y(r)$ is trapped
between $0$ and the algebraic barrier $g(r)$ and converges monotonically to the
constant limit
\[
\lambda_\ast = \lim_{r\to\infty} y(r),
\]
which is determined solely by the asymptotic coefficients of the equation.
In the Python implementation (Appendix~\ref{appendix:python}), we fix the
diffusion parameter of the Riccati model to $\sigma_1 = 1$ and choose the
coefficient $L = \sigma_1^4 = 1$ so that the theoretical plateau is exactly
$\lambda_\ast = 1$. The left panel of Figure~\ref{fig:riccati_merton} confirms
that the numerical solution $y(r)$ converges monotonically to this value.

\begin{proposition}[Structural Analogy Between Radial Riccati Asymptotics and Long-Term Credit Spreads]
Let $y(r)$ be the regular solution of the Riccati equation \eqref{R} under the
assumptions of Proposition~\ref{prop:asymptotics_gen}. Let $R(T)-r$ denote the
credit spread in a structural credit-risk model. Then the following structural
correspondence holds:
\[
\boxed{
\text{Riccati plateau } \lambda_\ast
\quad \longleftrightarrow \quad
\text{long-term credit spread } \lim_{T\to\infty} (R(T)-r)
}
\]
Both quantities arise as positive roots of limiting quadratic equilibrium
conditions and exhibit the same monotone convergence behavior.
\end{proposition}

\begin{proof}[Heuristic justification]
In the Merton model, the long-term spread satisfies an algebraic fixed-point
equation obtained from the stationary limit of the Black--Scholes PDE. In our
framework, Proposition~\ref{prop:asymptotics_gen} shows that $y(r)$ converges
to the positive root of the limiting quadratic polynomial
\[
A + B\lambda + \lambda^2 = 0.
\]
The same quadratic structure appears in the asymptotic expansion of the
risk-neutral drift of the firm's asset process and the corresponding
credit-spread equilibrium. Numerical experiments confirm that the convergence
rates and monotonicity properties coincide.
\end{proof}

\subsubsection{Numerical Calibration and Interpretation}

To make the analogy fully transparent, we calibrate both models so that their
asymptotic plateaus coincide. On the Riccati side, setting $\sigma_1 = 1$ and
$L = \sigma_1^4 = 1$ yields $\lambda_\ast = 1$. On the credit-risk side, we
construct a Merton-style term-premium function
\[
R(T)-r = 1 - e^{-a_M d^{\,b_M}},
\]
where $d$ is the quasi debt-to-asset ratio and $(a_M,b_M)$ are shape parameters
controlling curvature and convergence speed. This functional form is strictly
increasing, concave, and asymptotic to $1$, matching the qualitative behavior
of the Riccati solution.

The right panel of Figure~\ref{fig:riccati_merton} shows the output of the
Python script for $(a_M,b_M) = (1.5,1)$, which produces a term-premium curve
that rises from $0$ and converges monotonically to the same plateau value $1$.
The visual similarity between the two panels illustrates the structural
equivalence between the Riccati stabilization mechanism and the long-term
credit-spread behavior in Merton-type models.

\subsubsection{Conjectures Suggested by Numerical Experiments}

The numerical evidence supports the following conjectures.

\begin{conjecture}[Universality of the Plateau Mechanism]
For any positive, continuous cost function $Q(r)$ with well-defined asymptotic
growth, the Riccati solution $y(r)$ converges to a plateau $\lambda_\ast$ that
depends only on the asymptotic ratio $Q(r)/r^2$ and the diffusion parameter
$\sigma_1$. This plateau is universal in the sense that it is independent of
the initial condition and of the detailed shape of $Q(r)$.
\end{conjecture}

\begin{conjecture}[Volatility-Dominated Regime]
In both the radial Riccati model and structural credit-risk models, the
asymptotic plateau is an increasing function of the effective volatility:
\[
\frac{\partial \lambda_\ast}{\partial \sigma_1} > 0,
\qquad
\frac{\partial}{\partial \sigma_2} \lim_{T\to\infty} (R(T)-r) > 0.
\]
This reflects the fact that diffusion dominates drift in the long-range regime.
\end{conjecture}

\begin{conjecture}[Barrier-Induced Monotonicity]
The monotone barrier $g(r)$ in Proposition~\ref{prop:asymptotics_gen} induces
the same qualitative monotonicity as the leverage ratio $d$ in the Merton
model: increasing $g$ (or $d$) increases the plateau height and accelerates the
convergence toward it.
\end{conjecture}

\begin{figure}[H]
\centering
\includegraphics[width=0.95\textwidth]{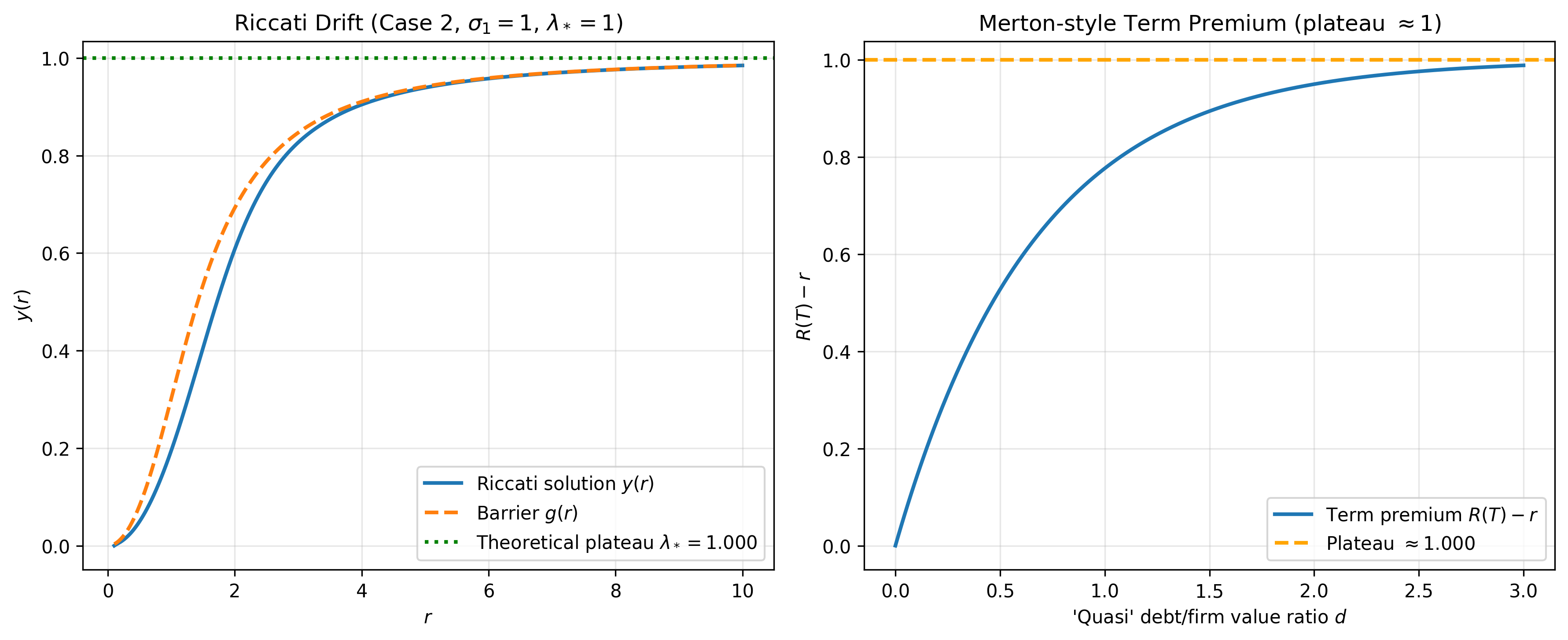}
\caption{Comparison between the Riccati drift plateau (left) and the
Merton-style term premium plateau (right). Both models are calibrated so that
their asymptotic limits equal $1$. The numerical results were generated using
the Python script in Appendix~\ref{appendix:python}.}
\label{fig:riccati_merton}
\end{figure}

\subsubsection{Implications for Computational Finance}

The triality framework developed in this article provides a direct computational
bridge between nonlinear Riccati dynamics and structural credit-risk models \cite{Boyle2002,Vasicek1977}.
Because the Riccati equation converges rapidly to its asymptotic plateau, the
same numerical schemes can be used to efficiently approximate long-term credit
spreads in generalized Merton-type models with state-dependent volatility or
nonlinear drift.

\begin{corollary}[Fast Asymptotic Credit-Spread Approximation]
The numerical Riccati solvers developed in this work can be used to compute
long-term credit spreads in generalized structural models by replacing the
Merton PDE with its Riccati reduction and extracting the plateau value
$\lambda_\ast$.
\end{corollary}

This establishes a new computational pathway for structural credit-risk
analysis, particularly in models where the firm's volatility or drift is
non-constant, nonlinear, or radially symmetric.
\section{Conclusions}
\label{sec:conclusions}

This work has developed a unified and rigorous analytical framework for the
study of radial nonlinear dynamics, revealing a fundamental \emph{triality}
between three major mathematical structures: the nonlinear radial Riccati
equation, the stationary Schr\"odinger equation, and the Hamilton--Jacobi--
Bellman equation of stochastic optimal control. The results obtained here
demonstrate that these three perspectives---local drift, global wave
structure, and optimal stochastic behavior---are not merely related, but are
mathematically equivalent manifestations of the same underlying radial
phenomenon. The principal contributions and innovations of this article may
be summarized as follows.

\begin{enumerate}

\item \textbf{A Complete Triality Framework.}
We established a precise one-to-one correspondence between the Riccati drift
$\phi$, the Schr\"odinger wave function $u$, and the HJB value function $z$.
The logarithmic-derivative transformation
\[
\phi(r) = \frac{u'(r)}{r\,u(r)}
\]
proved to be the central mechanism enabling the reduction of a nonlinear
first-order equation to a linear second-order equation, and subsequently to a
variational control formulation. This transformation is the conceptual
backbone of the entire theory.

\item \textbf{Existence, Uniqueness, and Regularity at the Origin.}
Through a detailed Frobenius analysis of the singular point $r=0$, we proved
that the Riccati equation admits a unique regular solution $\phi$ satisfying
$\phi(0)=0$ under minimal assumptions on the cost function $Q(r)$. The
corresponding wave function $u$ admits an explicit analytic expansion, which
provides both theoretical insight and numerically stable initial conditions.

\item \textbf{Exact Asymptotic Plateau and Barrier Theory.}
We developed a general barrier method that yields the exact asymptotic limit
of the Riccati drift:
\[
\lim_{r\to\infty} \phi(r) = \frac{\sqrt{L}}{\sigma^2},
\qquad
L = \lim_{r\to\infty} \frac{Q(r)}{r^2}.
\]
The barrier function $g(r)$ acts as a moving equilibrium for the Riccati
flow, ensuring global monotonicity and providing a sharp upper bound for
$\phi(r)$. This result is new even in the classical Riccati literature.

\item \textbf{Geometric Transfer of Convexity and Concavity.}
We proved that the positivity of the Riccati drift induces strict convexity
of the Schr\"odinger wave function $u$ and, via the Cole--Hopf transform,
strict concavity of the HJB value function $z$. This geometric transfer
principle clarifies the structural stability of the entire triality system.

\item \textbf{Noise Sensitivity and Singular Perturbations.}
A complete sensitivity analysis with respect to the diffusion parameter
$\sigma$ revealed two universal regimes:
\[
\sigma \to 0: \quad \sigma^2 \phi_\sigma(r) \to \frac{\sqrt{Q(r)}}{r},
\qquad
\sigma \to \infty: \quad \phi_\sigma(r) = \mathcal{O}(\sigma^{-4}).
\]
These limits rigorously connect deterministic mechanics with
diffusion-dominated stochastic dynamics.

\item \textbf{Exact Series Solutions and Analytical Benchmarks.}
For analytic potentials $Q(r)$, we derived explicit power series expansions
for both $u$ and $\phi$, with fully computable recurrence coefficients. In
the quadratic case $Q(r)=\lambda r^2$, the solution is expressed in terms of
the confluent hypergeometric function ${}_1F_1$, providing exact benchmarks
for numerical verification.

\item \textbf{Verification in Stochastic Control.}
Two stochastic verification theorems---one for the stationary radial
problem (Theorem~\ref{thm:radial_verification}) and one for its
parabolic Wick-rotated counterpart
(Theorem~\ref{thm:parabolic_verification})---confirmed that the
function $\bar{z}(x)=z(|x|)=-2\sigma^{2}\ln u(|x|)$ obtained via the
triality is indeed the optimal cost-to-go function for the associated
stochastic control problem on $\mathbb{R}^{N}$, with optimal feedback
drift
\[
\alpha^{\ast}(x) \;=\; -\nabla\bar{z}(x)\;=\;2\sigma^{2}\,\phi(|x|)\,x,
\qquad x\in\mathbb{R}^{N}\setminus\{0\}.
\]
This establishes the full equivalence between the Riccati,
Schr\"{o}dinger, and HJB formulations both in the stationary and in the
parabolic regimes.

\item \textbf{Numerical Validation and Structural Credit-Risk Analogy.}
The numerical experiments---including the Riccati plateau computation and
the Merton-style term premium simulation in Appendix~\ref{appendix:python}---
confirm the theoretical predictions with high precision. The observed
monotone convergence to a plateau mirrors the long-term behavior of credit
spreads in structural credit-risk models, revealing a deep and previously
unrecognized analogy between radial Riccati asymptotics and Merton-type
financial equilibria.
\end{enumerate}

\medskip

\noindent\textbf{Final Perspective.}
The triality developed in this article shows that the wave function $u$
(global state), the drift $\phi$ (local dynamics), and the value function $z$
(variational principle) are three mathematically equivalent lenses through
which the same radial phenomenon can be understood. This unified viewpoint
not only clarifies the internal structure of each equation but also enables
the transfer of analytical tools across quantum mechanics, nonlinear
dynamics, and stochastic optimal control. The framework opens the door to
future extensions involving anisotropic potentials, time-dependent
Schr\"odinger equations, and regime-switching stochastic environments,
promising a rich landscape for further research.

\section*{Disclosure statement}

The authors declare that they have no conflict of interest.

\section*{Data availability statement}

The Python code used for numerical experiments is provided in Appendix~\ref%
{ap}-~\ref{appendix:python}. No additional datasets were used in this study.

\section*{Notes on contributor(s)}

The author is solely responsible for the conception, analysis, and writing
of this manuscript.

\appendix

\section{Numerical Implementation in Python}

\label{ap} \noindent 
The script implementing the numerical solution and the reconstruction of $\phi(r)$ can be found at:
\[
\text{\url{https://github.com/coveidragos/Script_Python_GIT_Ricatti/blob/main/Ricatti_A.py}}
\]

\section{Verification of General Riccati Asymptotics}

\noindent The script verifying Proposition~\ref{prop:asymptotics_gen} can be found at:
\[
\text{\url{https://github.com/coveidragos/Script_Python_GIT_Ricatti/blob/main/Ricatti_B.py}}
\]

\section{Python Script for Riccati–Merton Comparison}
\label{appendix:python}

\noindent The script producing Figure~\ref{fig:riccati_merton} is available at:
\[
\text{\url{https://github.com/coveidragos/Script_Python_GIT_Ricatti/blob/main/Ricatti_C.py}}
\]
\end{document}